\documentclass[11pt]{article} %

\usepackage[utf8]{inputenc} %
\usepackage{geometry} %
\usepackage{fancyhdr} 
\usepackage{amsthm,amsfonts}
\usepackage{verbatim}
\usepackage{booktabs}
\usepackage{dcolumn,multirow,tabularx}
\usepackage{bm}
\usepackage{bbm}

\usepackage[colorlinks]{hyperref} %

\usepackage[utf8]{inputenc} 
\usepackage[T1]{fontenc}
\usepackage{url}
\usepackage{ifthen}
\usepackage{cite}
\usepackage{amsmath} %
\usepackage{graphicx}
\usepackage{amssymb}
\usepackage{mathtools} %
\usepackage{enumitem} %
\usepackage{xcolor}
\usepackage{caption}

\usepackage{pgfplots}
\pgfplotsset{compat=newest}
\usepgfplotslibrary{groupplots}
\usepgfplotslibrary{dateplot}

\usepackage{algpseudocode,algorithm,algorithmicx}  
\usepackage{swag} %

\theoremstyle{definition}
\newtheorem{estimator}{Estimator}
\newtheorem{remark*}{Remark}

\def \E {\mathop{{}\mathbb{E}}}

\def \real {\mathbb{R}}
\def \mprob {\mathbb{P}}
\def \scriptJ {\mathcal{J}}
\def \muhat {\widehat{\mu}}
\def \mustar {\mu^*}
\def \rtil {\widetilde{r}}
\def \mutil {\widetilde{\mu}}
\def \scriptP {\mathcal{P}}

\def \betahat {\widehat{\beta}}
\def \betastar {\beta^*}
\def \Diam {\operatorname{Diam}}

\begin{document}

\begin{center}

	{\bf{\LARGE{Estimating location parameters in entangled single-sample distributions}}}

	\vspace*{.25in}

 	\begin{tabular}{ccccc}
 		{\large{Ankit Pensia}$^*$} & \hspace*{.5in} & {\large{Varun Jog$^\dagger$}} & \hspace*{.5in} & {\large{Po-Ling Loh$^{\ddagger\mathsection}$}}\\
 		{\large{\texttt{ankitp@cs.wisc.edu}}} & \hspace*{.5in} & {\large{\texttt{vjog@ece.wisc.edu}}} & \hspace*{.5in} & {\large{\texttt{ploh@stat.wisc.edu}}}
 			\end{tabular}
 \begin{center}
 Departments of Computer Science$^*$, Electrical \& Computer Engineering$^\dagger$, and Statistics$^\ddagger$\\
 University of Wisconsin-Madison\\
Department of Statistics, Columbia University$^\mathsection$
 \end{center}

	\vspace*{.2in}

July 2019

	\vspace*{.2in}

\end{center}

\begin{abstract}
We consider the problem of estimating the common mean of independently sampled data, where samples are drawn in a possibly non-identical manner from symmetric, unimodal distributions with a common mean. This generalizes the setting of Gaussian mixture modeling, since the number of distinct mixture components may diverge with the number of observations. We propose an estimator that adapts to the level of heterogeneity in the data, achieving near-optimality in both the i.i.d.\ setting and some heterogeneous settings, where the fraction of ``low-noise'' points is as small as $\frac{\log n}{n}$. Our estimator is a hybrid of the modal interval, shorth, and median estimators from classical statistics; however, the key technical contributions rely on novel empirical process theory results that we derive for independent but non-i.i.d.\ data.
In the multivariate setting, we generalize our theory to mean estimation for mixtures of radially symmetric distributions, and derive minimax lower bounds on the expected error of any estimator that is agnostic to the scales of individual data points. Finally, we describe an extension of our estimators applicable to linear regression. In the multivariate mean estimation and regression settings, we present computationally feasible versions of our estimators that run in time polynomial in the number of data points.
\end{abstract}

\section{Introduction}

Many modern data sets involve various forms of heterogeneity that lead to new challenges in estimation and prediction. Whereas much of classical statistics focuses on convergence guarantees for i.i.d.\ observations, both the independence and identical distribution assumptions may be called into question in specific scientific applications~\cite{Liu88, DunEtal07, SteChr09, ZhuEtal15, FlaEtal16}.

We focus on the problem of estimating a common mean when data are generated independently, but from non-identical distributions. The special case where each sample is drawn from a normal distribution with a potentially different variance was studied by Chierichetti et al.~\cite{ChiEtAl14}, who showed the existence of a gap between estimation error rates of the maximum likelihood estimator when both the set of variances and their assignments are known and the best possible estimator in the case where the variances of the distributions are completely unknown. Furthermore, Chierichetti et al.~\cite{ChiEtAl14} present a mean estimator in the unknown variance case based on calculating the ``shortest gap'' between samples, and derive upper bounds on the estimation error of their algorithm. The work of Chierichetti et al.~\cite{ChiEtAl14} is motivated by aggregating user ratings in crowdsourcing, where the rating reported by each user will be drawn from a distribution centered around the true quality of the item, but with a different variance corresponding to the expertise of the user. Importantly, only one observation is available from each distribution, although the aggregate data are drawn from a Gaussian mixture.

A natural question is whether the estimators proposed by Chierichetti et al.~\cite{ChiEtAl14} might also be seen to perform well in non-Gaussian settings. For instance, one might ask whether concentration inequalities for sub-Gaussian random variables might be plugged into the analysis in natural ways to obtain good upper bounds. Furthermore, although Chierichetti et al.~\cite{ChiEtAl14} derive lower bounds for the behavior of the best possible estimator in the unknown variance setting, their work leaves open the question of whether their proposed estimator actually performs optimally, and for which collections of variances.

Parameter estimation of mixture models is well-studied in statistics and computer science~\cite{Lin95, Das99, AroKan01, KanEtal05, AchMcS05}. However, our setting is somewhat different from the canonical setting, since the number of observations is not large relative to the number of component distributions; rather, the number of components in the mixture could be as large as the number of observations. On the other hand, the parameters of the component mixtures are fundamentally entangled via a common mean parameter, which is the quantity we wish to estimate. Consequently, although much of the literature in statistical estimation for mixture models requires the underlying component distributions to possess certain tail characteristics such as Gaussianity or log-concavity, such assumptions are not necessary to obtain small estimation error in our setting. Some interesting recent work~\cite{PraEtal19} has considered the problem of mean estimation in heavy-tailed distributions by treating the data as a mixture of i.i.d.\ data from a lighter-tailed component and an outlier component, where the outlier samples are first removed using a screening step. At a philosophical level, our work is somewhat related in the sense that we consider regimes where a fraction of points are drawn from ``nicer'' distributions and the remaining points are arbitrarily heavy-tailed. On the other hand, a fundamental difference is that we do not assume that the number of points drawn from each mixture component is tending to infinity. Furthermore, we are chiefly interested in settings where the fraction of points drawn from low-variance distributions tends to zero as the number of samples increases.

In this paper, we revisit the problem of common mean estimation and substantially generalize the setting beyond Gaussian mixtures. In particular, the only assumption we impose on each of the component distributions is symmetry and unimodality about a common mean. Although our proposed estimator is similar to the estimator studied by Chierichetti et al.~\cite{ChiEtAl14}, we use a rather different approach for the analysis, which allows us to obtain bounds without assuming Gaussianity, sub-Gaussianity, or even finite variances of individual distributions. Our analysis is inspired by ideas in empirical process theory, and the upper bounds involve percentiles of the overall mixture distribution, making them useful even in the case of Cauchy-type distributions with heavy tails. Furthermore, in the multivariate mean setting, we sharpen the estimation error rates of Chierichetti et al.~\cite{ChiEtAl14} for isotropic Gaussian data, and show that the results hold more generally for mixtures of radially symmetric, unimodal distributions.

Our proposed estimators are connected to classical estimators appearing in the statistics literature, notably the modal interval estimator~\cite{Cher64} and the shorth estimator~\cite{AndEtAl72}. However, existing analysis of these estimators has generally been asymptotic and limited to i.i.d.\ data. In fact, it is well-known that in the i.i.d.\ setting, both the modal interval and shorth estimators have an ${n}^{-\frac{1}{3}}$ convergence rate~\cite{KimPol90}, compared to the faster $n^{-\frac{1}{2}}$ convergence rate of the sample mean---on the other hand, our analysis shows that these estimators enjoy superior performance to the sample mean when a substantial fraction of the component distributions have variances that are extremely large or even infinite. This underscores the fundamental fact that estimators which are suboptimal in a ``clean'' data setting may be preferable from the point of view of robustness.

The main contributions of our paper may be summarized as follows:
\begin{itemize}
\item Provide a rigorous analysis of the modal interval (Theorems~\ref{cor:modal},~\ref{ThmLepski}, and~\ref{LemModalD}), shorth (Theorems~\ref{cor:shorth} and~\ref{ThmShorthD}), and hybrid (Theorems~\ref{thm:hybrid} and~\ref{thm:HybridScreening}) estimators for multivariate, radially symmetric distributions. We also show how to relax the symmetry conditions further (Theorem~\ref{ThmRelax}). These estimation error guarantees hold with high probability.
\item Derive upper bounds on the expected error of the estimators (Theorem~\ref{ThmExpBound}). Along the way, we demonstrate the need for additional conditions on the tails of the mixture components in order to derive expected error bounds of the same order as the high-probability results.
\item Derive minimax lower bounds on the error rate of any estimator (Theorem~\ref{ThmLowerBound}), and prove that the hybrid estimator is nearly optimal in various regimes of interest (Theorem~\ref{ThmUpperBound}).
\item Extend the methodology for multivariate mean estimation to linear regression (Theorem~\ref{ThmRegression}).
\item Provide computationally efficient versions of the multivariate mean estimator (Theorem~\ref{ThmModalCompute}) and linear regression estimator (Theorem~\ref{ThmCompReg}) in high dimensions.
\end{itemize}
In order to establish this theory, we derive novel uniform concentration results for independent, non-i.i.d.\ data (Lemmas~\ref{thm:highProb},~\ref{thm:highProbD}, and~\ref{ThmUniformProb}), the proofs of which are interesting in their own right and may be more broadly applicable to other problems involving non-i.i.d.\ data. We also note that while our work vastly generalizes the results of Chierichetti et al.~\cite{ChiEtAl14} for mean estimation in Gaussian mixtures, our derivations also bypass some critical technical gaps in their proofs using a very different approach via empirical process theory. Finally, we comment that preliminary work on this topic appeared in an earlier conference paper~\cite{PenEtal19ISIT}, but was limited to the univariate case (Theorems~\ref{cor:modal},~\ref{ThmLepski},~\ref{cor:shorth}, and~\ref{thm:hybrid}) and did not discuss optimality, regression, or any computational aspects. Furthermore, all examples and counterexamples illustrating various phenomena, including the detailed theoretical derivations (Propositions~\ref{prop:exmple_r_k}--~\ref{prop:HighExp2}), are new to this paper.

The roadmap of the paper is as follows: In Section~\ref{sec:prelim}, we define notation and the basic estimators we will consider in the univariate case, which are subsequently analyzed  in Section~\ref{sec:individual}. In Section~\ref{sec:multivariate}, we present results for the multivariate analog of these estimators. In Section~\ref{sec:hybrid}, we combine the ideas from the previous sections to motivate and analyze a hybrid estimator. In Section~\ref{SecExpect}, we derive expected error bounds on the performance of our estimators, and also present minimax lower bounds on the estimation error of any estimator, thus providing settings in which our proposed estimators are provably optimal. In Section~\ref{SecComputation}, we present computationally feasible variants of our estimators in higher dimensions, and prove that the error rates of these estimators are of the same order as those derived earlier. In Section~\ref{SecAltCond}, we discuss various relaxations of the symmetry assumptions on the mixture components. In Section~\ref{SecRegression}, we describe our results for linear regression. Simulation results reporting the relative performance of different estimators are contained in Section~\ref{SecSims}. All proofs are contained in the supplementary appendix. 

\textbf{Notation:} For two real-valued functions $f(n)$ and $g(n)$, we write $f(n) = \omega(g(n))$ if for every real constant $c > 0$, there exists $n_0 \ge 1$ such that $f(n) > c \cdot g(n)$ for every integer $n \ge n_0$. 
We use $\tilde{\O}\(\cdot\)$, $\tilde{\Omega}\(\cdot\)$, and $\tilde{\omega}\(\cdot\)$ to hide polylogarithmic factors.   
We write w.h.p., or ``with high probability,'' to mean with probability tending to 1 as the sample size increases.
We use $C$, $c$, $C'$, and $c'$ to represent absolute positive constants which may vary from place to place, and their exact values can be found in the proofs. 
Similarly, we use $C_t$ to represent positive numbers that depend only on $t$. For a real-valued random variable $X$, we use $\Var X$ to denote its variance.

We will use $\|\cdot\|_2$ to denote the Euclidean norm. We use $B(x,r)$ to denote the Euclidean ball of radius $r$ centered around $x$, and we also write $B_r$ in place of $B(0,r)$. We denote the $d \times d$ identity matrix by $I_d$.  We use $P(X,\epsilon)$ to denote the $\epsilon$-packing number of a set $X$ with respect to Euclidean distance, and we use $N(X,\epsilon)$ to denote the $\epsilon$-covering number. We write $\Diam(X)$ to denote the diameter of the set with respect to Euclidean distance.

\section{Problem setup}
\label{sec:prelim}

We begin by introducing the univariate mean estimation problem. Suppose we have $n$ independent samples $X_i \sim P_i$, where each $P_i$ is a univariate distribution with a density. Each $P_i$ is assumed to have a density $p_i$ which is symmetric and unimodal around a common mean (and median) $\mustar$\footnote{Or equivalently, the $p_i$'s are symmetric and decreasing around $\mustar$.}. Our goal is to estimate the location parameter $\mustar$ from the $n$ samples, where the $P_i$'s are unknown a priori and may even come from different classes of (non)parametric distributions. For simplicity of presentation, we will assume that $\mustar = 0$, so the error of an estimator $\muhat$ is measured by $|\muhat|$. There is no loss of generality, since the estimators that we consider are translation-invariant.

A natural estimator to use is the empirical mean, which is certainly an unbiased estimator of $\mustar$. However, it is a well-known fact that the mean is not ``robust,'' in the sense that one outlying observation can have a massive impact on the estimation error of the mean. In our setting, one $P_i$ with a very large variance can dramatically inflate the error of the mean, even if the remaining $n-1$ distributions are well-behaved. Due to the symmetry assumption on the $P_i$'s, we could consider the median as a more robust alternative. Our theory in Section~\ref{sec:individual} below shows that using a median estimator can somewhat improve the estimation error so that it depends only on the spread of the $\sqrt{n} \log n$ distributions with the smallest quantiles; however, other more cleverly constructed estimators can reduce this dependence to $\O(\log n)$ distributions, meaning that $n - \O(\log n)$ mixture components may have arbitrarily large (or even infinite) variances, yet have a bounded effect on the behavior of the estimator.

Another potential estimator when the mixing components come from a sufficiently nice parametric family (e.g., Gaussians) is the maximum likelihood estimator. However, since we do not assume knowledge of which observations are drawn from which mixture components, the MLE calculation becomes considerably more complicated. Nonetheless, it is sometimes informative to compare the error rate of the MLE---assuming side information of which observations correspond to which mixture components---to the error rates obtained using various agnostic estimators. In particular, if the former error rate diverges with $n$, we know that a diverging error rate for a proposed estimator is reasonable.

Since we are not given access to the individual $P_i$'s, we will argue about the problem through the lens of the mixture distribution $\overline{P} \coloneqq \frac{1}{n} \sum_{i=1}^n P_i$, which is again unimodal and symmetric. We will write $\overline{P}_n$ to denote the empirical distribution of $X_1,\hdots,X_n$.

\subsection{Estimators}
\label{SecEst}

We now proceed to define the estimators that will serve as building blocks for our algorithms. All of these estimators can be implemented efficiently after sorting the data points.

For $x \in \real$ and $r > 0$, let $f_{x,r}$ denote the indicator function of the interval $[x-r,x+r]$.  Let
\begin{align}
\cH & := \{f_{x,r}: x \in \R, r \in \R, r \ge 0 \}, \\
\cH_{r} & = \{f_{x,r'}: x\in \R , r' \in \R, 0 \le r' \leq r \}.
\label{eq:Hr}
\end{align}
Both $\cH$ and $\cH_r$ have VC dimension $2$~\cite{Ver18}.

For a function $f$, we use $R_n(f) \coloneqq \frac{1}{n} \sum_{i=1}^n f(X_i)$ to denote the expectation of $f$ with respect to the empirical distribution of $X_1,\hdots,X_n$. Let 
\begin{align*}
R(f) \coloneqq  \frac{1}{n} \sum_{i=1}^n\E f(X_i).
\end{align*}
Thus, $R(f)$ is the expectation of $f$ with respect to $\overline{P}$. Let
\begin{align*}
R^*_{r} \coloneqq \sup_{f \in \cH_{r}} R(f)  = R(f_{0,r}),
\end{align*}
where the second equality follows by symmetry and unimodality. It is also equal to the probability of the interval $[-r,r]$ under $\overline{P}$.

\begin{estimator}[$r$-modal interval]
The $r$-modal interval estimator, introduced for the classical i.i.d.\ setting by Chernoff~\cite{Cher64}, outputs the center of the most populated interval of length $r$, with ties broken arbitrarily:
\begin{align}
\label{EqnModalEst}
\widehat{\mu}_{M,r} \in \arg\max_{x} R_n(f_{x,r}).
\end{align}
\end{estimator}

\begin{estimator}[$k$-shortest gap / shorth estimator]
For $k \geq 2$, the $k$-shortest gap ($k$-shorth) estimator, $\widehat{\mu}_{S,k}$, outputs the center of the shortest interval containing at least $k$ points. More precisely, we define
\begin{align}
\label{EqnShorthEst}
\widehat{r}_{k} &\coloneqq \inf\left\{r: \sup_{x}R_n(f_{x,r}) \geq \frac{k}{n}\right\}, \qquad \widehat{\mu}_{S,k} \coloneqq \widehat{\mu}_{M, \widehat{r}_{k}}.
\end{align}
The traditional shorth estimator~\cite{AndEtAl72, KimPol90} corresponds to $k = \frac{n}{2}$, whereas choosing $k=2$ outputs the midpoint of the shortest interval between any two points. As we will see, the choice of $k = C\log n$ will be convenient for our setting, and is more suitable than $k = \frac{n}{2}$ if data are not i.i.d.
\end{estimator}

Note that a type of ``shortest interval'' estimator has also been employed in the work on mean estimation for contaminated i.i.d.\ data~\cite{LaiEtAl16}, but was used as an outlier screening step in that context, rather than a mean estimator. Incidentally, our hybrid estimator to be introduced later will employ a different screening approach based on the median, and then use the shorth estimator to return a more accurate mean estimate.

The $k$-median outputs an element from the centermost $k$ points of the data. Note that the output is a set rather than a point estimator; however, the $k$-median will be useful as a preprocessing step before applying the modal interval or shorth estimators, to obtain better estimation error rates.

\begin{estimator}[$k$-median] 
The $k$-median estimator output an arbitrary element $\widehat{\mu}_{\text{med},k}$ from the subset $S_k$, defined as: $X_i \in S_k$ if and only if $\widehat{\theta}_{\text{med},-k} \leq X_i \leq  \widehat{\theta}_{\text{med},k}$, where%
\begin{align*}
\widehat{\theta}_{\text{med},k} & := \inf\left\{\theta: \psi_n(\theta) \geq \frac{k}{n}\right\},\\
 \widehat{\theta}_{\text{med},-k} & := \sup\left\{\theta: \psi_n(\theta) \leq \frac{-k}{n}\right\},
\end{align*}
and  $\psi_n(\theta) = \frac{1}{n} \sum_{i=1}^n \text{sign}(\theta - X_i)$. The sample median corresponds to taking $k=0$. 
\end{estimator}

\subsection{Population-level properties}

We also define the population-level quantities
\begin{align*}
r_{k} \coloneqq \inf\left\{r: \sup_{x} R(f_{x,r}) \geq \frac{k}{n}\right\} =  \inf \left\{r:  R(f_{0,r}) \geq \frac{k}{n}\right\},
\end{align*}
where the last equality follows from unimodality and symmetry. Note that $r_k$ measures the spread of $\overline{P}$ and $r_{n/2}$ is the interquartile range of $\overline{P}$. Furthermore, since $\overline{P}$ has a density, we have $R^*_{r_k} = \frac{k}{n}$.

Let $q_i$ and $\sigma_i$ denote the interquartile range and standard deviation of $P_i$, respectively. Recall that the interquartile range satisfies $\P(|X_i - \mustar| \le q_i) = \frac{1}{2}$. We use $q_{(i)}$ and $\sigma_{(i)}$ to denote the $i^\text{th}$ smallest interquartile range and standard deviation, respectively. By Lemma~\ref{lemma:properties}\ref{prop:5} below, we have $r_k \le q_{(2k)}$ and $r_k \le 2\sigma_{(2k)}$, although these bounds may be loose (for instance, $r_k$ could be finite even if $\sigma_{(1)}$ is infinite). However, we are guaranteed that $r_k$ will be small if $2k$ points come from ``nice'' (low-variance) distributions.

The following lemma is proved in Appendix~\ref{app:lemmaProp}:

\begin{lemma}
We have the following properties:
\begin{enumerate}[label=(\roman*)]
  \item For any $r>0$ and $x, x' \in \real$, if $|x| < |x'|$, then $R(f_{x,r}) \geq R(f_{x',r})$.
\label{prop:1}
  \item For any $x \in \real$, if $r < r'$, then $R(f_{x,r}) \leq R({f_{x,r'}})$. 
  \label{prop:2}
  \item If $0 < r < r'$, then $ \frac{R^*_r}{r} > \frac{R^*_{r'}}{r'}$. \label{prop:3}
  \item If $0 < r < r'$, then $R(f_{r',r}) < \frac{r}{r'} R^*_{r'}  $. \label{prop:4}
  \item If $1 \leq k \leq n $,  then $\frac{k}{n} < R^*_{q_{(2k)}} $ and $ \frac{k}{n} < R^*_{2 \sigma_{(2k)}}$. \label{prop:5} 
\end{enumerate}
\label{lemma:properties}
\end{lemma}

Lemma~\ref{lemma:properties} shows that we can use $\overline{P}$ as a measure of distance between two intervals. In particular, if two intervals with the same center/radius are close under $R$, the respective radii/centers must also be close.

\subsection{Examples}
\label{sec:behavior_r_k}

Note that the quantity $r_k$ is problem-dependent, since its magnitude depends on the relative dispersion of the mixing components. In particular, we are interested in $r_{ \Theta(\log n)}$.
As the fraction of ``nice'' points increase, $r_k$ gets smaller. However, it doesn't depend too strongly on the high-variance distributions. 
We illustrate this below in several cases for $r_{\log n}$, assuming Gaussian distributions for simplicity:
\begin{example} (i.i.d.\ observations). $P_i = \cN(0, \sigma^2)$, so $\overline{P}$ is again $\cN(0, \sigma^2)$.
\label{exm:IID}
\end{example}
\begin{example} (quadratic variance). $P_i = \cN(0,c^2i^2)$, for some small $c>0$.
\label{exm:QuadVar}
\end{example}
\begin{example} ($\alpha$-mixture distributions).
\begin{align*}
P_i = \begin{cases}\cN(0,1), &\text{if } i \leq c\lceil\log n \rceil, \\
					\cN(0,n^{2\alpha}), & \text{otherwise},  \end{cases}
\end{align*}
for some $\alpha > 0$.
\label{exm:alpha-mix}
\end{example}

These examples will reappear throughout the paper to illustrate the error of our proposed estimators in various regimes of interest. The following proposition, proved in Appendix~\ref{app:prop_r_k}, will be useful in our development.

\begin{proposition}
We have the following bounds for $r_{\log n}$:
\begin{enumerate}
	\item For Example~\ref{exm:IID} (i.i.d.\ observations), we have $r_{\log n} = \Theta\(\frac{\sigma \log n}{n}\)$.
	\item For Example~\ref{exm:QuadVar} (quadratic variance) and sufficiently small $c > 0$,
	we have $r_{\log n} = \Theta(1)$.
	\item For Example~\ref{exm:alpha-mix} ($\alpha$-mixture distributions) and sufficiently large $c > 0$, we have
\begin{align*}
r_{\log n} = \begin{cases}
				\Theta\(\frac{\log n}{n^{1 - \alpha}}\), & \text{if } \alpha < 1, \\
				\Theta(1), & \text{if } \alpha \geq 1.
			\end{cases}
\end{align*}
\end{enumerate}
\label{prop:exmple_r_k}
\end{proposition}

Note that these bounds are tighter than the ones provided by Lemma~\ref{lemma:properties}\ref{prop:5}; the latter states that $r_{k} \leq \sigma_{(2k)}$. This is because Lemma~\ref{lemma:properties}\ref{prop:5} is a worst-case bound which does not account for the contributions of high-variance points.

\section{Performance analysis of individual estimators } %
\label{sec:individual}

We now analyze the behavior of the modal interval, shorth, and median estimators. We begin by deriving a uniform concentration result that is critical to the technical proofs.

\subsection{Concentration inequalities} %
\label{sec:concentration_inequalities}

We use the following concentration bound to prove tight concentration results for the modal and the shorth estimators. 
\begin{lemma}
Let $\cH_r$ be as defined in equation~\eqref{eq:Hr} with $R^*_r = \sup_{f \in \cH_r} R(f)$. For any fixed $t \in (0,1]$ and $n > 1$, we have
\begin{align*}
\P\left\{\sup_{f \in \cH_{r}} |R_n(f) - R(f)| \geq tR^*_{r}\right\}  \leq 2\exp\(- c nR^*_{r}t^2\),
\end{align*}
provided $r$ is large enough so that $R^*_{r} \geq \frac{C_t \log n}{n}$, where $C_t = \left(\frac{144}{t}\right)^2$ and $c = \frac{1}{200}$.
\label{thm:highProb}
\end{lemma}

The proof is given in Appendix~\ref{app:proof_of_thm_high_prob} and modifies the VC-type bounds derived for i.i.d.\ settings. This theorem is useful because the bounds rely on $R_r^*$; i.e., they are adaptive to the problem, compared to the traditional $\O\(\frac{1}{\sqrt{n}}\)$ distribution-independent bound. However, note that Lemma~\ref{thm:highProb} requires the mass $R^*_r$ lying around the true mode to be sufficiently large.

In our later analysis of multidimensional mean estimators and linear regression estimators, we will derive variants of Lemma~\ref{thm:highProb} for different function classes.

\subsection{Modal interval estimator}
\label{SecModal}

The following result provides a high-probability bound on the error of the modal interval estimator by using Lemma~\ref{thm:highProb} on intervals:
\begin{lemma}
Let $r$ be such that $R^*_{r} \geq C_{0.5t}\(\frac{\log n}{n}\)$. Then with probability at least $1 - 2\exp(-c'nR^*_r t^2)$, we have $R(f_{\muhat_{M,r},r}) \geq (1 - t)R^*_{r}$.
\label{thm:modalIntervalMain}
\end{lemma}

If $r$ is small, then $R(f_{x,r})$ behaves like a (scaled) density of the mixture distribution $\overline{P}$. Lemma~\ref{thm:modalIntervalMain} states that the density of $\overline{P}$ at the empirical mode, $\widehat{\mu}_{M,r}$, is within a constant factor of the density at $\mustar$.

We can then derive the following error bound for the modal interval estimator: 
\begin{theorem}
\label{cor:modal}
Let $r$ be such that $R^*_r = \Omega\left(\frac{\log n}{n}\right)$. Suppose $r'$ is such that $R(f_{r', r}) < \frac{R^*_r}{2}$. Then with probability at least $1 - 2\exp(-c'nR^*_r)$, we have
\begin{equation}
\label{EqnMuhatBd1}
|\widehat{\mu}_{M,r}| \le r'.
\end{equation}
In particular, we can always choose $r' = \frac{2r}{R^*_r}$ to obtain the bound
\begin{align}
\label{EqnMuhatBd}
|\widehat{\mu}_{M,r}| \leq  \frac{2r}{R^*_r}.
\end{align}
\end{theorem}

The proofs of Lemma~\ref{thm:modalIntervalMain} and Theorem~\ref{cor:modal} are contained in Appendix~\ref{app:proof_of_theorem_modal}, and proceed by using Lemma~\ref{thm:highProb} to bound the ratio between $R(f_{\widehat{\mu}_{M,r}, r})$ and $R^*_r$, and then using Lemma~\ref{lemma:properties} to turn this into a deviation bound on $|\widehat{\mu}_{M,r}|$. Although the bound~\eqref{EqnMuhatBd} in Theorem~\ref{cor:modal} is simple to state, it may be looser than the bound~\eqref{EqnMuhatBd1}.

\begin{remark*}
\label{RemModal}
Importantly, by Lemma~\ref{lemma:properties}\ref{prop:5}, we know that the choice $r = \sigma_{(C\log n)}$ always guarantees the condition $R^*_r = \Omega\left(\frac{\log n}{n}\right)$. Hence, inequality~\eqref{EqnMuhatBd} implies that
\begin{equation}
\label{EqnRemModal}
|\muhat_{M,r}| \le \frac{2\sigma_{(C\log n)}}{R_r^*} \le \frac{2n\sigma_{(C\log n)}}{\log n},
\end{equation}
with a similar inequality involving $q_{(C'\log n)}$. Note that this bound holds regardless of the magnitude of the standard deviations of the latter $n - C\log n$ mixture components.

At the same time, one might be wary of the fact that the bound in inequality~\eqref{EqnRemModal} could \emph{increase} with $n$ if we fix $q_{(C\log n)}$; for i.i.d.\ data, $R^*_r = \Theta(1)$, so even the first expression in the bound is of constant order. This is rather alarming, compared to the $\O(n^{-1/2})$ error rate of the median. However, it should be noted that if the variances of the mixture components increase sufficiently rapidly with $n$, even the error rate of the MLE in the Gaussian case (which knows the distribution of each sample) will have a diverging error rate. Thus, although the error bounds of the modal interval estimator in Theorem~\ref{cor:modal} may be rather unsatisfactory in the case of i.i.d.\ data, they can lead to more meaningful error bounds when the mixture distribution involves a sizable portion of high-variance points. We will explore the question of optimality in more detail in Section~\ref{SecOptimality} below.
\end{remark*}

We now revisit the examples from Section~\ref{sec:behavior_r_k} and calculate the bounds that follow from Lemma~\ref{thm:modalIntervalMain} by choosing $r = r_{C\log n}$ for a large constant $C > 0$. We also mention the cases where the bound~\eqref{EqnMuhatBd} is weaker than the bound~\eqref{EqnMuhatBd1}. The proof of the following proposition is contained in Appendix~\ref{app:exmPropModal}.

\begin{proposition}
Suppose $r = r_{C \log n}$. We have the following bounds for $|\widehat{\mu}_{M,r}|$:
\begin{enumerate}
	\item For Example~\ref{exm:IID} (i.i.d.\ observations), we have $|\widehat{\mu}_{M,r}| \leq \Theta(\sigma)$, w.h.p.
	\item For Example~\ref{exm:QuadVar} (quadratic variance), we have $|\widehat{\mu}_{M,r}| \leq \O(n^ \epsilon)$, w.h.p., for any $\epsilon > 0$. Inequality~\eqref{EqnMuhatBd}  results in a weaker bound of the form $\O(n)$, w.h.p.
	\item For Example~\ref{exm:alpha-mix} ($\alpha$-mixture distributions), we have
	\begin{align*}
			    |\widehat{\mu}_{M,r}| = \begin{cases}
			    				\O(n^\alpha), & \text{if } \alpha < 1 \\
			    				\O(1), & \text{if } \alpha \ge 1,
 			    				\end{cases}
    \end{align*} w.h.p.
	 For $\alpha \geq  1$, inequality~\eqref{EqnMuhatBd} results in a weaker bound of the form $\O(n)$.
\end{enumerate}
\label{prop:exmModal}	
\end{proposition}

\begin{remark*}
As discussed in Remark~\ref{RemModal} above, the guarantees for the modal interval estimator are somewhat unsatisfactory for i.i.d.\ data, since Proposition~\ref{prop:exmModal}(i) gives an error rate of $\Theta(\sigma)$, rather than the optimal rate $\Theta\left(\frac{\sigma}{\sqrt{n}}\right)$ achievable by the sample mean. On the other hand, Proposition~\ref{prop:exmModal} shows that for other problem settings with more widely varying variances---such as the $\alpha$-mixture with $\alpha \ge 1$---the modal interval estimator results in constant error, whereas the sample mean would have $\Theta(n^{\alpha - 0.5})$ error. These differences are summarized in more detail in Table~\ref{table: ex123} below.
\end{remark*}

The modal interval estimator is a ``local'' estimator that only considers the value of $\overline{P}_n$ in small windows.
As we increase the variance of noisy points, the distribution $\overline{P}$ approaches $0$ around $\mustar$.
The modal interval estimator makes mistakes when $\overline{P}$ is flat after normalization, meaning that the density at $x + \mustar$ is within a $(1 - \epsilon)$-factor of its density at $\mustar$, for $\epsilon = o(1)$. 
If this is the case, $\overline{P}_n$ might assign higher mass at $x+\mustar$ than $\mustar$ due to stochasticity introduced by sampling, so a local method would mistakenly choose $x+\mustar$ over $\mustar$.

More concretely, consider the setting of Example~\ref{exm:alpha-mix}. If an adversary tried to alter the estimator by making the variance of the points very high $(\alpha \gg 1)$, then although $\overline{P}$ would approach $0$, the normalized density would not be flat.
An extreme example of this can be seen when variance of noisy points is ``$\infty$'': Near $\mustar$, the distribution $\overline{P}$ would behave like $\cN(\mustar,1)$ scaled by $\O\(\frac{\log n}{n}\)$, which is \textit{not} flat after normalization although $\overline{P}$ approaches $0$ very rapidly, so that the mean or median would behave poorly. As Proposition~\ref{prop:exmModal} shows, the modal interval estimator would only suffer $\O(1)$ error in this case.

\begin{remark*}
\label{RemPhase}
Examining the bound in Proposition~\ref{prop:exmModal} for Example~\ref{exm:alpha-mix}, we see the possible emergence of a ``phase transition'' phenomenon: For $\alpha < 1$, the modal interval estimator has error growing with $n$, whereas for $\alpha \ge 1$, the modal interval estimator only incurs constant error. This suggests that for $\alpha < 1$, high-variance points are more effectively hidden within the mixture distribution, so the accuracy of the modal interval estimator is more severely compromised than in the case when $\alpha \ge 1$, where the modal interval estimator can distinguish between low-variance and high-variance points. This phase transition phenomenon is established rigorously in Section~\ref{SecPhase} below, where we prove a lower bound of $\Omega(n^\alpha)$ in the case when $\alpha < 1$.
\end{remark*}

Finally, note that the modal interval estimator $\widehat{\mu}_{M,r}$ is an $M$-estimator~\cite{Hub64} of the form $\widehat{\mu} \in \arg\min_\mu \left\{\frac{1}{n} \sum_{i=1}^n g(x_i - \mu)\right\}$, with loss function $g(x) = -f_{0,r}(x)$. Clearly, the loss is nonconvex; however, the modal interval estimator is nonetheless computable, since the overall objective function is piecewise constant, with transitions lying only at the $n$ data points. Whereas a more straightforward analysis of convex $M$-estimators (e.g., Huber) would yield somewhat similar bounds on estimation error, such arguments would generally require tail assumptions on the component distributions $\{P_i\}$, which we avoid altogether in our analysis of the modal interval estimator.

\subsubsection{Adaptive method for choosing $r$}

Note that by Lemma~\ref{lemma:properties}\ref{prop:3}, the bound in Theorem~\ref{cor:modal} is tighter for smaller values of $r$. Thus, the choice of $r$ which optimizes the bound satisfies $R^*_r = C\left(\frac{\log n}{n}\right)$, yielding the bound $|\widehat{\mu}_{M,r}| = \O\( \frac{nr_{C\log n}}{\log n}\)$. However, choosing a suitable $r$ is challenging because we do not know $\overline{P}$:
If $r$ is too small, then $R^*_r$ might not be large enough and the bounds might not hold, whereas if $r$ is too large, then the resulting bound is loose.

Fortunately, an estimator with near-optimal performance may be obtained via Lepski's method~\cite{Lep91}. The basic steps are as follows: 
 Define $r^* = r_{C_{0.25} \log n}$ to be the interval width satisfying $R^*_{r^*} = C_{0.25}\left(\frac{\log n}{n}\right)$, and suppose we have rough initial estimates $r_{\min}$ and $r_{\max}$ such that $r_{\min} \le r^* \le r_{\max}$. Define $r_j := r_{\min} 2^j$, and define
\begin{equation*}
\scriptJ := \left\{j \ge 1: r_{\min} \le r_j < 2 r_{\max}\right\}.
\end{equation*}
We then define the index $j_*$ to be
\begin{equation*}
\min\left\{j \in \scriptJ: \forall i > j \text{ s.t. } i \in \scriptJ, |\muhat_{M, r_i} - \muhat_{M, r_j}| \le \frac{4nr_i}{C_{0.25} \log n}\right\},
\end{equation*}
which may be calculated using pairwise comparisons of the modal interval estimator computed over the gridding of $[r_{\min}, r_{\max}]$. We define $j_* = \infty$ if the set is empty; as proved in the theorem below, we have $j_* < \infty$, w.h.p. We then have the following result, proved in Appendix~\ref{app:prfLepski}:
\begin{theorem}
\label{ThmLepski}
With probability at least $1 - 2\left(1+\log_2\left(\frac{2r_{\max}}{r_{\min}}\right)\right) \exp(-c' \log n)$, we have
\begin{equation}
\label{EqnModalAdapt}
|\muhat_{M, r_{j_*}}| \le \frac{12nr^*}{C_{0.25} \log n}.
\end{equation}
\end{theorem}

Note that the cost of using Lepski's method is a factor of 6 in the estimation error. Of course, the validity of the method requires the availability of the rough bounds $r_{\min}$ and $r_{\max}$. A natural way to obtain rough bounds on $r^*$ from the data is to use the shortest gap estimator, which returns the shortest interval containing at least $C_{0.25} \log n$ points. We can again use Lemma~\ref{thm:highProb} to show that $\widehat{r}_k$ does not fluctuate too wildly from its empirical counterpart:

\begin{lemma}
For $k \geq 2C_{0.5}\log n$, with probability at least $1 - 2\exp(-c k)$, we have $r_{k/2} \leq \widehat{r}_{k} \leq r_{ 2k}$.
\label{lemma:shorthlength}
\end{lemma}

The proof of Lemma~\ref{lemma:shorthlength} is in Appendix~\ref{app:length_of_shortest_gap}. Accordingly, we may use $r_{\min} = \hat{r}_{C_{0.25} \log n/2}$ and $r_{\max} = \hat{r}_{2C_{0.25} \log n}$.

\subsubsection{Phase transition behavior}
\label{SecPhase}

In this subsection, we focus on verifying the statement in Remark~\ref{RemPhase} above, namely the existence of a phase transition for the modal interval estimator depending on whether $\alpha < 1$ or $\alpha \ge 1$. This phenomenon is illustrated via simulations in the plots of Figure~\ref{FigPhase}.

\begin{figure}[!ht]
    \centering
     \begin{minipage}{0.45\textwidth}
        \centering
		\scalebox{0.85}{
\begin{tikzpicture}

\definecolor{color0}{rgb}{0,0.75,0.75}

\begin{axis}[
legend style={cells={align=left}},
log basis x={10},
log basis y={10},
tick align=outside,
tick pos=left,
x grid style={white!69.01960784313725!black},
xlabel={n},
xmin=23987.5800423731, xmax=22918352.4522639,
xmode=log,
xtick style={color=black},
xtick={1000,10000,100000,1000000,10000000,100000000,1000000000},
xticklabels={$\displaystyle {10^{3}}$,$\displaystyle {10^{4}}$,$\displaystyle {10^{5}}$,$\displaystyle {10^{6}}$,$\displaystyle {10^{7}}$,$\displaystyle {10^{8}}$,$\displaystyle {10^{9}}$},
y grid style={white!69.01960784313725!black},
ylabel={Average error},
ymin=24.3682078354879, ymax=1239873.79779726,
ymode=log,
ytick style={color=black},
ytick={1,10,100,1000,10000,100000,1000000,10000000,100000000},
yticklabels={$\displaystyle {10^{0}}$,$\displaystyle {10^{1}}$,$\displaystyle {10^{2}}$,$\displaystyle {10^{3}}$,$\displaystyle {10^{4}}$,$\displaystyle {10^{5}}$,$\displaystyle {10^{6}}$,$\displaystyle {10^{7}}$,$\displaystyle {10^{8}}$}
]
\addplot [thick, color0, opacity=0.55, mark=square*, mark size=4, mark options={solid}]
table {%
32768 51.2980982709956
65536 67.216274420189
131072 84.6022785993172
262144 122.305332313158
524288 175.355832884497
1048576 209.722970395916
2097152 261.598016871146
4194304 377.493804281258
8388608 420.299399706156
16777216 616.773759258315
};
\addplot [thick, black, opacity=0.55, dash pattern=on 1pt off 3pt on 3pt off 3pt, mark=asterisk, mark size=4, mark options={solid}]
table {%
32768 39.8802486933979
65536 57.3435542841997
131072 71.0605855970876
262144 118.944507848551
524288 171.473092644633
1048576 220.165850491761
2097152 311.906335863327
4194304 420.054386978455
8388608 513.651705726343
16777216 808.097778530003
};
\addplot [thick, green!50.0!black, opacity=0.55, dotted, mark=diamond*, mark size=4, mark options={solid}]
table {%
32768 2499.78130755391
65536 5480.25850517543
131072 9443.11581242747
262144 17715.8019524889
524288 33879.924192391
1048576 67693.6455044011
2097152 127073.378496186
4194304 215409.616342662
8388608 432880.37231722
16777216 757605.666574019
};
\end{axis}

\end{tikzpicture}}
        \caption*{(a) $\alpha = 0.9$}
        \label{fig:prob1_6_1_3}
    \end{minipage}%
    \begin{minipage}{.55\textwidth}
        \centering
		\scalebox{0.85}{
\begin{tikzpicture}

\definecolor{color0}{rgb}{0,0.75,0.75}

\begin{axis}[
legend style={cells={align=left}},
legend cell align={left},
legend style={at={(1.04,1)}, anchor=north west, draw=white!80.0!black},
log basis x={10},
log basis y={10},
tick align=outside,
tick pos=left,
x grid style={white!69.01960784313725!black},
xlabel={n},
xmin=23987.5800423731, xmax=22918352.4522639,
xmode=log,
xtick style={color=black},
xtick={1000,10000,100000,1000000,10000000,100000000,1000000000},
xticklabels={$\displaystyle {10^{3}}$,$\displaystyle {10^{4}}$,$\displaystyle {10^{5}}$,$\displaystyle {10^{6}}$,$\displaystyle {10^{7}}$,$\displaystyle {10^{8}}$,$\displaystyle {10^{9}}$},
y grid style={white!69.01960784313725!black},
ylabel={Average error},
ymin=11.7215517054651, ymax=975779.729514267,
ymode=log,
ytick style={color=black},
ytick={1,10,100,1000,10000,100000,1000000,10000000},
yticklabels={$\displaystyle {10^{0}}$,$\displaystyle {10^{1}}$,$\displaystyle {10^{2}}$,$\displaystyle {10^{3}}$,$\displaystyle {10^{4}}$,$\displaystyle {10^{5}}$,$\displaystyle {10^{6}}$,$\displaystyle {10^{7}}$}
]
\addplot [thick, color0, opacity=0.55, mark=square*, mark size=4, mark options={solid}]
table {%
32768 3117.88526486187
65536 5858.55305807642
131072 9757.89909820569
262144 18355.1928127732
524288 27146.5219157093
1048576 50582.7536737193
2097152 80468.9743262579
4194304 155594.194263234
8388608 242510.136089129
16777216 467077.341691056
};
\addlegendentry{mean           }
\addplot [thick, black, opacity=0.55, dash pattern=on 1pt off 3pt on 3pt off 3pt, mark=asterisk, mark size=4, mark options={solid}]
table {%
32768 1899.62477958624
65536 3837.83715862428
131072 7040.30643104309
262144 17861.3193141773
524288 29288.0397077553
1048576 52996.7244295321
2097152 97362.3295180963
4194304 173385.381945333
8388608 306903.077152158
16777216 583040.550335228
};
\addlegendentry{median         }
\addplot [thick, green!50.0!black, opacity=0.55, dotted, mark=diamond*, mark size=4, mark options={solid}]
table {%
32768 886.745601249071
65536 5937.05447786775
131072 20.7883023368208
262144 20.7084387635399
524288 20.7696500393682
1048576 22.0661670294508
2097152 19.9185432914994
4194304 20.9814336260215
8388608 19.6172505429853
16777216 21.3627599366926
};
\addlegendentry{modal\\interval}
\end{axis}

\end{tikzpicture}}
        \caption*{(b) $\alpha = 1.3$}
        \label{fig:prob1_6_2_3}
    \end{minipage}%
        \caption{Plots comparing average error of the mean, median, and modal interval estimators on Example~\ref{exm:alpha-mix} ($\alpha$-mixture distributions) for different values of $\alpha$. As shown in Proposition~\ref{prop:exmModal}, the modal interval estimator undergoes a phase transition at $\alpha = 1$, where the error of modal interval estimator drops from the increasing function $\Omega(n^\alpha)$ to the constant function $\Theta(1)$. Moreover, as shown in Proposition~\ref{prop:exmMedian}, the median has better performance than the modal interval estimator for $\alpha < 1$, motivating our hybrid estimator in Section~\ref{sec:hybrid}. The average error, $\frac{1}{T}\sum_{i=1}^T |\muhat - \mustar| $, is calculated using $T = 200$ runs for each $n$. Both of the axes are on the $\log$ scale. More details can be found in Section~\ref{SecSims}.
    }
    \label{FigPhase}
\end{figure}

For ease of analysis, we tweak the setting of Example~\ref{exm:alpha-mix} slightly: Instead of having different distributions for high variance and low variance points, we assume that the points are sampled \iid from a mixture distribution, with weights resembling their original fraction in Example~\ref{exm:alpha-mix}. Moreover, we assume that individual distributions are uniform rather than Gaussian.

\begin{example} (Modified $\alpha$-mixture distributions).
\begin{align*}
Q_n = \frac{c \log n}{n} U[-1,1] + \frac{n - c\log n}{n} U[-n^\alpha,n^{\alpha}]
\end{align*}
where $U[-a,a]$ is the uniform distribution on $[-a,a]$.
\label{exm:alpha-mixModified}
\end{example}

Note that if we sample $X_1 , \hdots, X_n \iidm Q_n$, the number of points with variance $\Theta(1)$ is $\Theta(\log n)$, w.h.p. It is easy to see that the upper bounds for Example~\ref{exm:alpha-mixModified} are the same as that of Example~\ref{exm:alpha-mix} in Proposition~\ref{prop:exmModal}, i.e.,
\begin{align*}
|\widehat{\mu}_{M,r}| = \begin{cases}
\O(n^\alpha), & \text{if } \alpha < 1 \\
\O(1), & \text{if } \alpha \ge 1,
\end{cases}
\end{align*}
w.h.p. The following proposition, proved in Appendix~\ref{app:modal_lower_bound}, establishes a lower bound of $\Omega(n^\alpha)$ on the error:

\begin{proposition}
For $ \frac{1}{3} \leq  \alpha < 1$ in Example~\ref{exm:alpha-mixModified}, the modal interval estimator incurs $\Omega(n^{\alpha})$ error, with constant probability. 
\label{lem:modal_lower_bound}
\end{proposition}

Proposition~\ref{lem:modal_lower_bound} proves rigorously that the apparent phase transition of the modal interval estimator is not simply an artifact of the argument used to prove Proposition~\ref{prop:exmModal}. Indeed, the modal interval estimator experiences a sharp phase transition depending on the relative variance of the mixture component with the higher variance, which is governed by the parameter $\alpha$.
Moreover, this phase transition is not specific to just modal interval estimator. As stated in Theorem~\ref{ThmLowerBound}, all agnostic estimators must have error $\Omega(n^{\alpha-0.5})$ for $\alpha < 1$. Thus Example~\ref{exm:alpha-mix} is indeed a difficult problem for $\alpha < 1$, but a surprisingly easy one for $\alpha > 1$.

As a final remark, note that in Examples~\ref{exm:alpha-mix} and~\ref{exm:alpha-mixModified}, the sample median and even the mean would have an error of $\tilde{\O}(n^{\alpha-0.5})$. When $\alpha < 1$, this rate is much better than the $\O(n^\alpha)$ guarantee of the modal interval estimator. This motivates the hybrid estimator we will propose in Section~\ref{sec:hybrid}, which is able to combine the ``best of both worlds'' for the modal interval and median estimators.

\subsection{Shorth estimator}

Guarantees for the shorth estimator are similar to the modal interval estimator, but computing the shorth does not require an extra step for adapting to the width of the optimal interval. In addition, as the proofs of the results in this section reveal, the technical machinery we have developed to derive guarantees for the error of the modal interval estimator may also be used to derive estimation error bounds for the shorth estimator.

We have the following theorem, proved in Appendix~\ref{app:shorthMassProof}:

\begin{theorem}
\label{cor:shorth}
Suppose $2k \geq C_{0.25}\log n$. With probability at least $1 - 2\exp(-c'k)$, we have
\begin{align*}
|\widehat{\mu}_{S,k}| \leq \frac{2nr_{2k}}{k} < \frac{2n \min\(q_{(4k)}, 2\sigma_{(4k)}\)}{k}.
\end{align*}
\end{theorem}

The performance of the $\Theta(\log n)$-shorth estimator is similar to the modal interval estimator with $r = r_{\Theta\(\log n\)}$ (cf.\ inequality~\eqref{EqnRemModal} in Remark~\ref{RemModal}). Consequently, the error guarantees derived for the running examples in Proposition~\ref{prop:exmModal} also hold for the $\Theta(\log n)$-shorth.

\begin{remark*}
Lemma~\ref{lemma:properties}\ref{prop:3} shows that the upper bound is actually tighter for small $k$:  for $k' > k$, we have $ kr_{2k'} > k'r_{2k} $. The smallest value permissible from our theory would be $k = \Theta(\log n)$. Also note that the upper bound in Theorem~\ref{cor:shorth} for the shorth estimator resembles the bound in Theorem~\ref{ThmLepski}, except for the fact that the bound for the modal interval estimator involves the quantity $r_{C_{0.25} \log n}$ rather than $r_{2C_{0.25} \log n}$, and the latter could be larger depending on the spread of $\overline{P}$. Furthermore, both upper bounds in Theorem~\ref{cor:shorth} may sometimes be loose: In particular, if the $X_i$'s were i.i.d., $r_{2k}$ would be of order $\Theta\left(\frac{k}{n}\right)$ for small $k$, so the bound $\frac{nr_{2k}}{k}$ would be of constant order, whereas it is known~\cite{KimPol90} that the shorth estimator is consistent for $k = 0.5n$.
\label{remark:shorth}
\end{remark*}

\subsection{$k$-median}

Median-based estimators generally have worse estimation error guarantees than the modal interval or shorth. However, we will use the $k$-median estimator as a screening step in the hybrid estimator described in Section~\ref{sec:hybrid} for obtaining better error rates overall. Thus, we develop some preliminary theory concerning the $k$-median estimator that will be invoked later in the paper.

For the median estimator, we have the following result, proved in Appendix~\ref{app:k_median}:
\begin{lemma}
We have that $\widehat{\theta}_{\text{med}, k} \leq r_{k + \delta}$ and $\widehat{\theta}_{\text{med}, -k} \ge -r_{k + \delta}$, with probability at least $1 - 2\exp(-\delta^2/n)$.
\label{lemma:median}
\end{lemma}

We will be interested in the case when both $k = \delta = \Theta(\sqrt{ n }\log n)$.
Lemma~\ref{lemma:median} states that for $k \geq \sqrt{n \log n}$, the error of $k$-median is less than $r_{2k}$ w.h.p. 
Compared to the shorth and modal interval estimators, the $k$-median is a more ``global'' estimator, because it looks at a much bigger interval $f_{0, r_{2k}}$:
When the latter $n-\O(k)$ points have very large variances, the density $\overline{P}$ is much flatter, so the value of $r_{2k}$ is large. On the other hand, recall that this is not an issue for the modal interval and shorth estimators, which only have error governed by the smallest $\O(\log n)$ variances. This behavior can be seen explicitly in comparing the guarantees for Example~\ref{exm:alpha-mix} ($\alpha$-mixture disributions) for $\alpha\geq 1$ in Table~\ref{table: ex123}.

For completeness, we calculate the bounds of the $(\sqrt{n}\log n)$-median estimator on the recurring examples, proved in Appendix~\ref{app:PropExmMedian}:
\begin{proposition}
We have the following bounds on the $(\sqrt{n}\log n)$-median estimator:
\begin{enumerate}
	\item For Example~\ref{exm:IID} (i.i.d.\ observations), $|\widehat{\mu}_{\text{med},\sqrt{n}\log n}| = \O\left(\frac{\sigma \log n}{\sqrt{n}}\right)$, w.h.p.
	\item For Example~\ref{exm:QuadVar} (quadratic variance), $|\widehat{\mu}_{\text{med},\sqrt{n}\log n}| = \O(n^{0.5}\log n)$, w.h.p.
	\item For Example~\ref{exm:alpha-mix} ($\alpha$-mixture distributions),
	$|\widehat{\mu}_{\text{med},\sqrt{n}\log n}| = \O(n^{\alpha - 0.5}\log n)$, w.h.p.
\end{enumerate}
\label{prop:exmMedian}
\end{proposition} 

\begin{table}[H]
\centering
\renewcommand{\arraystretch}{1.5}

\begin{tabular}{l|l|l|l|l|}
\cline{2-5}
                     &Mean  &Median &Modal/Shorth  &Hybrid  \\ \hline
\multicolumn{1}{|l|}{Example 1 (i.i.d.\ samples)} &$\bm{n^{-0.5}}$ &$\bm{n^{-0.5}}$ &$n^{-1/3}$ &$\bm{n^{-0.5}}$  \\ \hline
\multicolumn{1}{|l|}{Example 2 (quadratic variances)} &$  \sqrt{n}$  &$\sqrt n$ &$\bm{n^\epsilon}$  &$\bm{n^\epsilon}$ \\ \hline
\multicolumn{1}{|l|}{Example 3 ($\alpha < 1$-mixture distributions)} &$\bm{n^{\alpha - 0.5}}$ &$\bm{n^{\alpha-0.5}}$  &$n^\alpha$ &$\bm{n^{\alpha - 0.5}}$  \\ \hline
\multicolumn{1}{|l|}{Example 3 ($\alpha \geq 1$-mixture distributions)} &$n^{\alpha - 0.5}$ &$n^{\alpha - 0.5}$ &$\bm{c}$ &$\bm{c}$  \\ \hline
\end{tabular}
\caption{The table above summarizes the performance of various estimators on our three running examples. We have ignored poly-logarithmic factors for simplicity, and we use $n^\epsilon$ to denote $\O(n^\epsilon)$ error for any $\epsilon > 0$, and $c$ to denote an error bounded by a constant. The radius for the modal estimator and the $k$ for the shorth estimator are adjusted to be optimal for each particular example; i.e., the estimators are assumed to know which example data are coming from. Observe that mean and median estimators outperform the modal and shorth estimators when the outliers have relatively small variances. On the other hand, the modal and shorth estimators are better when the outliers have large variances. Simulations in Section~\ref{SecSims} show that the rates provided above are indeed observed in practice. Our hybrid estimator from Section~\ref{sec:hybrid} achieves the best performance in all cases \emph{without knowing which example is under consideration.}}
\label{table: ex123}
\end{table}

\section{Multivariate case} %
\label{sec:multivariate}

In the following sections, we explore generalizations of our theory to $d$ dimensions. The results of this section focus on the simple setting where the overall mixture distribution is radially symmetric, e.g., we have multivariate Gaussian observations $X_i \sim \cN( 0_d, \sigma_i^2 I_d )$. Throughout this section, we focus on the setting where $d$ is $\O(\log n)$. As shown in Chierichetti et al.~\cite{ChiEtAl14}, when $d$ is $\Omega(\log n)$, the problem reduces to the case of known variances, since these can be estimated accurately. We shall discuss how to replace the spherical symmetry assumption by log-concavity in Section~\ref{SecAltCond}.

As the covariance matrix of a radially symmetric distribution is of the form $\sigma^2 I_d$, we denote the covariance matrix of $X_i$ by $\sigma_i^2 I_d$. We use $\sigma_{(i)}$ to denote the $i^\text{th}$ smallest marginal standard deviation. We also use $s_i$ to denote the interquartile range of $X_i$, so that $\mprob(\|X_i - \mustar\|_2 \le s_i) = \frac{1}{2}$, and we define $s_{(i)}$ as the corresponding order statistic.

Let $f_{x,r}(z) = \1_{\|x-z\|_2 \leq r}$ denote the indicator function of the $\ell_2$-ball of radius $r$. For $s \in \real$, we will also use $f_{s,r}(z)$ to denote the indicator of the ball of radius $r$ centered at the vector with first coordinate $s$ and all other coordinates equal to 0. The spherical symmetry assumption readily gives $R(f_{x,r}) = R(f_{s,r})$ for all $x$ such that $\|x\|_2 = s$. With these definitions, the multidimensional modal interval and shorth estimators are defined exactly as in equations~\eqref{EqnModalEst} and~\eqref{EqnShorthEst}. We will first state the following generalization of Lemma~\ref{lemma:properties}, containing some useful properties for radially symmetric distributions. The proof is contained in Appendix~\ref{AppLemProp2}.

\begin{lemma}
Suppose $R$ is radially symmetric and unimodal. We have the following properties:
\begin{enumerate}[label=(\roman*)]
  \item For any $r>0$ and $x, x' \in \R^d$, if $\|x\|_2 < \|x'\|_2$, then $R(f_{x,r}) \geq R(f_{x',r})$.
\label{prop2:1}
  \item For any $x \in \real^d$, if $r < r'$, then $R(f_{x,r}) \leq R({f_{x,r'}})$. 
  \label{prop2:2}
  \item If $0 < r_1 < r_2$, then $ \frac{R^*_{r_1}}{r_1^d} > \frac{R^*_{r_2}}{r_2^d}$. \label{prop2:3}
  \item If $0 < r_1 < r_2$, then
  \begin{equation*}
  R(f_{r_2,r_1}) < \frac{1}{P(B_{r_2 - r_1}, r_1 )} R^*_{r_2}  \le \left(\frac{2r_1}{r_2 - r_1}\right)^d R^*_{r_2},
  \end{equation*}
  where $P(B_{r_2 - r_1}, r_1)$ denotes the packing number of $B_{r_2 - r_1}$ with respect to $B_{r_1}$. In particular, if $r_1 \le \frac{r_2}{2}$, then $R(f_{r_2, r_1}) \le \left(\frac{4r_1}{r_2}\right)^d R^*_{r_2}$. \label{prop2:4}
  \item If $1 \leq k \leq n $, then $\frac{k}{n} < R^*_{s_{(2k)}}$ and $\frac{k}{n} < R^*_{ 2 \sqrt{d}\sigma_{(2k)}}$. \label{prop2:5}
  \end{enumerate}
\label{lemma:prop2}
\end{lemma}

\subsection{Concentration inequalities}

We consider the hypothesis class $\cH_r = \{f_{x,r'}: x \in \R^d, 0 \leq  r' \leq r\}$. Note that the VC dimension of $\cH_r$ is $d+1$~\cite{WenDud81}.

The proof of the following result is in Appendix~\ref{app:proof_of_thm_high_probD}.

\begin{lemma} 
For any fixed $t \in (0,1]$ and $n > 1$, we have
\begin{align*}
\P\left\{\sup_{f \in \cH_{r}} |R_n(f) - R(f)| \geq tR^*_{r}\right\}  \leq 2\exp\(- c nR^*_{r}t^2\),
\end{align*}
provided $r$ is large enough so that $nR^*_{r} \geq C_t \frac{d+1}{2} \log n$, where $C_t = \left(\frac{144}{t}\right)^2$ and $c = \frac{1}{200}$.
\label{thm:highProbD}
\end{lemma}
Note that the theorem requires $R^*_r$ to increase with $d$. However, $d$ is small as $d = \O(\log n)$.

\subsection{Modal interval estimator}

We will now convert the concentration inequalities into an error bound for the modal interval estimator. The proof is in Appendix~\ref{AppThmModalD}.

\begin{theorem}
\label{LemModalD}
Suppose $R^*_r \ge C_{0.5} \left(\frac{(d+1) \log n}{n}\right)$. The multidimensional modal interval estimator satisfies the error bounds
\begin{align}
\label{EqnModalD1}
\| \widehat{\mu}_{M,r} \|_2 & \leq 4r \left(\frac{2}{R^*_r}\right)^{\frac{1}{d}}, \\
\label{EqnModalD}
\| \widehat{\mu}_{M,r} \|_2   &\leq  8\sqrt{d} \sigma_{ (2Cd \log n)}\left(\frac{2}{R^*_r}\right)^{\frac{1}{d}} 
              \leq 8\sqrt{d} \left(\frac{n}{C'd \log n}\right)^{ \frac{1}{d}} \sigma_{ (2Cd \log n)},
\end{align}
with probability at least $1-2\exp(-c' d\log n)$.
\end{theorem}

By Lemma~\ref{lemma:prop2}\ref{prop2:3}, the bound~\eqref{EqnModalD1} is tighter for smaller values of $r$. Furthermore, the radius $r$ could be calibrated in practice using Lepski's method, analogous to the 1-dimensional case. Also note that inequality~\eqref{EqnModalD} could also be stated using $s_{(2k)}$ in place of $2\sqrt{d} \sigma_{(2k)}$, since it is obtained from inequality~\eqref{EqnModalD1} simply by substituting the bounds of Lemma~\ref{lemma:prop2}\ref{prop2:5}.

\begin{remark*}
Our bound~\eqref{EqnModalD} may be compared with Theorem 5.1 in Chierichetti et al.~\cite{ChiEtAl14}: note that we have removed a factor of polylog$(n)$, although their bound depends on $\sigma_{(\log n)}$ rather than $\sigma_{(d \log n)}$. Nonetheless, we emphasize the fact that our results hold for general radially symmetric distributions, whereas the proofs in Chierichetti et al.~\cite{ChiEtAl14} are Gaussian-specific.
\end{remark*}

\subsection{Shorth estimator}

We now derive error bounds for the multidimensional shorth estimator. The proof is contained in Appendix~\ref{AppThmShorthD}.

\begin{theorem}
\label{ThmShorthD}
Suppose $k \ge C(d+1) \log n$. The multidimensional shorth estimator satisfies the error bound
\begin{equation*}
\|\muhat_{S,k}\|_2 \le 4r_{2k} \left(\frac{2n}{k}\right)^{1/d},
\end{equation*}
with probability at least $1-2\exp(-c'd\log n)$.
\end{theorem}

As in the univariate case, the estimation error guarantees for the multidimensional modal interval and shorth estimators are similar. In particular, for the ``optimal'' choice of $r$ such that $R^*_r = \frac{cd\log n}{n}$, inequality~\eqref{EqnModalD1} in Theorem~\ref{LemModalD} gives the bound $\|\muhat_{M,r}\|_2 = \O\left(r_{c'd\log n}\left(\frac{n}{Cd\log n}\right)^{1/d}\right)$, which is of the same form as the guarantee from Theorem~\ref{ThmShorthD} when $k = C(d+1) \log n$.

\subsection{$k$-median estimator} \label{section: DMedian}

We will also utilize a multivariate version of the $k$-median estimator described in Section~\ref{SecEst}. Various multivariate extensions of the median exist, with different robustness properties and computational complexity; for our purposes, it will suffice to consider the simplest version of the multivariate median, which simply operates componentwise on the data points.

Note that since the overall mixture distribution is radially symmetric, all the marginal distributions are identical and symmetric about $0$. Accordingly, we denote the common marginal distribution by $\overline{P}_1$, and define $r_{k,1}$ to be the smallest interval (centered at $0$) that contains $\frac{k}{n}$ mass under $\overline{P}_1$. 

For each dimension $i$, we look at the $k$ median points in that dimension. We denote this set by $S_{k,i}$:
\begin{align*}
    S_{k,i} \coloneqq \{ X_j(i) : X_j(i) \text{ belongs to } k \text{-median of  } (X_j(i))_{j=1}^n \},
\end{align*}
where $X_j(i)$ denotes the $i^{th}$ coordinate of the vector $X_j$.

We now define $S_k^{\infty}$ to be the cuboid based on $S_{k,i}$, for each dimension $i$:
\begin{align*}
    S_k^\infty = \prod_{i=1}^d [ \min(S_{k,i}), \max(S_{k,i}) ].
\end{align*}
Note that the cuboid $S_k^\infty$ might not contain any data points. However, the following lemma, proved in Appendix~\ref{AppLemCubOrigin}, shows that it contains the origin w.h.p.:
\begin{lemma}
\label{LemCubOrigin}
With probability at least $1 - 4d \exp(-ck^2/n)$:
\begin{itemize}
\item[(i)] The cuboid $S_k^\infty$ contains the origin.
\item[(ii)] We have the bound $\Diam(S_{k}^\infty) \le \sqrt{d} r_{2k,1}$.
\end{itemize}
\end{lemma}

Lemma~\ref{LemCubOrigin} will be critical in our analysis of the hybrid estimator proposed below. In particular, the estimator will consist of projecting the modal interval/shorth estimator onto the cuboid $S_k^\infty$, and Lemma~\ref{LemCubOrigin}(i) guarantees that the estimation error of the projected estimator will be no larger than the estimation error of the initial estimator without projection. On the other hand, Lemma~\ref{LemCubOrigin}(ii) bounds the error of an estimator based on the $k$-median alone.

\section{Hybrid estimators}
\label{sec:hybrid}

We now present an algorithm that combines the shorth and $k$-median in order to obtain superior performance for both fast and slow decay of $\overline{P}$. 
Recall from Table~\ref{table: ex123} that the median has superior performance when there is less heterogeneity in the data and $\overline{P}$ decays fast enough. However, the superior performance of the modal interval estimator is apparent in the presence of large number of high variance points. It is then desirable to have an estimator that adapts to the problem and enjoys the best of both worlds without any prior information. Indeed, as outlined in Proposition~\ref{prop:exmHybrid}, the hybrid estimator achieves this rate. The key point is that if the true mean lies inside a convex set (defined with respect to the $k$-median), then projecting any other point (e.g., the shorth) onto the set will only move the point closer to the mean, so the hybrid estimator can leverage the better of the two rates enjoyed by the median and shorth.

\subsection{Hybrid estimator for $d=1$}

Algorithm~\ref{alg:hybrid2} defines the hybrid estimator in the univariate setting. As described earlier, it proceeds by separately computing the $k_1$-shorth estimator and $k_2$-median. If the shorth estimator lies within the median interval, the algorithm outputs the shorth; otherwise, it outputs the closest endpoint of the median interval. Note that this estimator resembles the estimator proposed by Chierichetti et al.~\cite{ChiEtAl14} since it employs the median as a screening step for points with very large variance. However, the shorth estimator is computed separately and then projected onto an interval around the median. In contrast, the estimator proposed by Chierichetti et al.~\cite{ChiEtAl14} first computes the $k_2$-median and then computes the shorth on the remaining points, leading to a delicate conditioning argument in the analysis and creating some technical gaps in the proofs.

\begin{algorithm}[h]  
  \caption{Hybrid mean estimator 
    \label{alg:hybrid2}}  
  \begin{algorithmic}[1]  
    \Statex  
    \Function{hybridMeanEstimator}{$X_{1:n}, k_1 , k_2 $}  
            \State $S_{k_1} \gets $ kMedian$(X_{1:n},k_1)$.      
        \State $\widehat{\mu}_{S,k_2} \gets $ Shorth$(X_{1:n},k_2)$.
    \If {$\widehat{\mu}_{S,k_2} \in  [\min(S_{k_1}), \max(S_{k_1})]$}
      \State $\widehat{\mu}_{k_1,k_2} \gets \widehat{\mu}_{S,k_2}$
    \Else
       \State $\widehat{\mu}_{k_1,k_2} \gets $ closestPoint$(S_{k_1}, \widehat{\mu}_{S,k_2})$ 
    \EndIf
    \State \Return $\widehat{\mu}_{k_1,k_2}$
    \EndFunction  
  \end{algorithmic}  
\end{algorithm}

\begin{theorem}
If $k_1 = \sqrt{n}\log n $ and $k_2 \ge C\log n$, the error of the hybrid estimator in Algorithm~\ref{alg:hybrid2} is bounded by
\label{thm:hybrid}
\begin{align*}
|\widehat{\mu}_{k_1,k_2}|  \leq \min\(\Diam(S_{k_1}), |\widehat{\mu}_{S,k_2}|\) \leq \frac{4\sqrt{n}\log n}{k_2} r_{2k_2},
\end{align*}
with probability at least $1 - 2\exp(-c'k_2) - 2\exp(-c\log^2n)$.
\end{theorem}

The proof of Theorem~\ref{thm:hybrid} is provided in Appendix~\ref{app:proof_hybrid}.
Importantly, the bound in Theorem~\ref{thm:hybrid} is finite even for heavy-tailed distributions with infinite variance.
Finally, note that in Algorithm~\ref{alg:hybrid2}, we could replace the shorth estimator by the modal interval estimator with adaptively chosen interval width (cf.\ Section~\ref{SecModal}) to obtain similar error guarantees.

The following proposition translates the error guarantees of Theorem~\ref{thm:hybrid} into our running examples. These bounds are a direct result of Theorem~\ref{thm:hybrid} and Propositions~\ref{prop:exmModal} and~\ref{prop:exmMedian}.

\begin{proposition} When $k_1$ and $k_2$ are chosen as in Theorem 
~\ref{thm:hybrid}, we have the following bounds:
\begin{enumerate}
	\item For Example~\ref{exm:IID} (i.i.d.\ observations), $|\widehat{\mu}_{k_1, k_2}| = \O\left( \frac{\sigma \log n}{\sqrt{n}} \right)$, w.h.p.
	\item For Example~\ref{exm:QuadVar} (quadratic variance), $|\widehat{\mu}_{k_1, k_2}| = \O(n^{\epsilon})$, w.h.p., for any $\epsilon > 0$.
	\item For Example~\ref{exm:alpha-mix} ($\alpha$-mixture distributions), with high probability,
	\begin{align*}
			    |\widehat{\mu}_{k_1, k_2}| = \begin{cases}
			    				\O(n^{\alpha-0.5}), & \text{if } \alpha < 1, \\
			    				\O(1), & \text{if } \alpha \ge 1.
 			    				\end{cases}
    \end{align*}
\end{enumerate}
\label{prop:exmHybrid}
\end{proposition}

\subsection{Hybrid estimators for $d>1$}

We now describe how to obtain rates of $O(\sqrt{n}^{1/d})$, rather than $O(n^{1/d})$ obtained in Theorems~\ref{LemModalD} and~\ref{ThmShorthD}. This is a multidimensional analog of the hybrid estimator. The challenge is that for $d > 1$, we need to define an appropriate notion of a $k$-median, and we also need the screening method to be efficiently computable.

The $d$-dimensional hybrid algorithm consists of the following steps, summarized in Algorithm~\ref{alg:hybridD}:
\begin{itemize}
\item[(i)] Compute the cuboid $S_{k_1}^\infty$ with $k_1 = \sqrt{n}\log n$.
\item[(ii)] Compute the $k_2$-shorth estimator $\muhat_{S, k_2}$ with $k_2 = Cd \log n$.
\item[(iii)] If $\muhat_{S, k_2} \notin S_{k_1}^\infty $, return the projection of $\muhat_{S,k_2}$ on  $S_{k_1}^\infty $. Otherwise, return $\muhat_{S, k_2}$.
\end{itemize}

\begin{algorithm}[h]  
  \caption{Hybrid mean estimator ($d$-dimensional)
    \label{alg:hybridD}}  
  \begin{algorithmic}[1]  
    \Statex  
    \Function{hybridMultidimensional}{$X_{1:n}, k_1, k_2, d$}  
            \State $S_{k_1}^\infty \gets $ kCuboid$(X_{1:n},k_1)$.      
        \State $\widehat{\mu}_{S,k_2} \gets $ Shorth$(X_{1:n},k_2)$.
    \If {$\widehat{\mu}_{S,k_2} \in  S_{k_1}^\infty$}
      \State $\widehat{\mu}_{k_1,k_2} \gets \widehat{\mu}_{S,k_2}$
    \Else
       \State $\widehat{\mu}_{k_1,k_2} \gets \arg\min_{x \in S_{k_1}^\infty} \|x - \widehat{\mu}_{S,k_2}\|_2$
    \EndIf
    \State \Return $\widehat{\mu}_{k_1,k_2}$
    \EndFunction  
  \end{algorithmic}  
\end{algorithm}

Note that the projection in step (iii) is easy to accomplish, since $\ell_2$-projection onto the cuboid may be done componentwise, hence computed in $O(d)$ time. As in the case of the 1-dimensional hybrid estimator, the modal interval estimator could also be used in place of the shorth.

We then have the following result, proved in Appendix~\ref{AppThmHybridMult}:

\begin{theorem}
Suppose $k_1 = \sqrt{n} \log n$ and $k_2 \geq Cd\log n$. Then the error of the hybrid algorithm is bounded by
\begin{equation*}
\|\muhat_{k_1, k_2}\|_2 \le \min\left\{\Diam(S_{k_1}^\infty), \|\muhat_{S,k_2}\|_2 \right\}  \le \min\left\{\sqrt{d} r_{2k_1,1}, C'\sqrt{n}^{1/d} r_{k_2}\right\},
\end{equation*}
with probability at least $1 - 2\exp(-c'k_2) - 4d \exp(-c\log^2 n)$.
\label{thm:HybridScreening}
\end{theorem}

\begin{remark*}
\label{RemHybrid}
Similar to the univariate case, the multivariate hybrid estimator achieves good error guarantees for both slow and fast decay of $\overline{P}$. In particular, when data are \iid Gaussian with distribution $\cN(0, \sigma^2 I_d)$, as in Example~\ref{exm:IID}, the error of the hybrid estimator is of the order $\O\(\frac{\sigma\sqrt{d}\log n}{\sqrt n}\)$. This is within $\log $ factors of the optimal $\frac{\sqrt{d} \sigma }{\sqrt{n}}$ error rate. At the same time, the worst-case error guarantee is of the form $\O\(\sqrt{d}\sqrt{n}^{1/d} \sigma_{(Cd \log n)} \)$.

We also briefly comment on the error guarantees of the hybrid estimator on the multivariate analog of Example~\ref{exm:alpha-mix}. We can show that $r_{2k_1, 1} = \tilde{\O}\( n^{\alpha - 0.5}\)$ and $r_{k_2}  =  \tilde{\O}\(\sqrt{d}n^{\alpha - \frac{1}{d}}\)$, so Lemma~\ref{LemCubOrigin} implies a bound of $\tilde{\O}(\sqrt{d} n^{\alpha - 0.5})$ for the median estimator. On the other hand, Theorem~\ref{ThmShorthD} leads to a bound of $\tilde{\O}(\sqrt{d} n^\alpha)$ for the shorth estimator. This bound can be improved for $\alpha \geq \frac{1}{d}$: If $\alpha \geq \frac{1}{d}$, we have $\|\muhat_{S, k_2}\|_2 = \O(\sqrt{d})$ (cf. Theorem~\ref{ThmUpperBound}). The second expression in Theorem~\ref{thm:HybridScreening} then implies that the error of the hybrid estimator is $\tilde{\O}\( \sqrt{d}\min( n^{\alpha - 0.5},1)\)$ for $\alpha \geq \frac{1}{d}$ and $\tilde{\O}(\sqrt{d}n^{\alpha - 0.5})$ for $ \alpha \leq  \frac{1}{d}$.
This improves upon the error rates of both the median and shorth estimators.
\label{RemHybridD}
\end{remark*}

\section{Bounds in expectation}
\label{SecExpect}

Thus far, we have focused on high-probability bounds. We now briefly discuss how to convert the upper bounds into bounds on the expected error of the estimator. We then derive lower bounds on the estimation error of any estimator, thus addressing the question of optimality in certain regimes.

\subsection{Imposing additional assumptions}

We first show that unlike high-probability bounds, expected error bounds of a similar order \emph{cannot} be derived for modal interval estimator without any assumptions on the high-variance mixture components. To illustrate this point, we provide an example in which it is possible to derive high-probability bounds of $\O(1)$ for the modal interval estimator without further assumptions, whereas bounds in expectation of a similar order provably require additional tail assumptions, since $\E|\widehat{\mu}_{M,1}| \to \infty$ as $q_n \to 
\infty$.

\begin{example}
\label{exm:HighExp2}
For any $n$, let the densities of the $P_i$'s be defined as follows:
For $ i \leq C\log n$, let
\begin{align*}
p_i(x) = \begin{cases} \frac{1}{6i} , & |x| \leq 3i, \\
					0, & \text{otherwise}.
		\end{cases}
\end{align*}
For  $i > C\log n$ and $\alpha \in (0,1)$, let
\begin{align*}
 p_i(x) = \begin{cases} n^{- \alpha} , & |x| \leq 1, \\
					h_n, & 1 < |x| \leq q_n, \\
					0, & \text{otherwise},
		\end{cases}
\end{align*}
where the $\{h_n\}$ and $\{q_n\}$ are constrained such that the total area is $1$, i.e., $2n^{- \alpha} + 2(q_n - 1)h_n = 1$ and $h_n \leq \frac{n^{-\alpha}}{2}$. In particular, for an $\alpha > 0$, we can still choose $q_n$ arbitrarily large; we will take $q_n = \Omega(n)$.
\end{example}

The proof of the following statement is contained in Appendix~\ref{app:HighExp2}:
\begin{proposition}
For Example~\ref{exm:HighExp2}, we have $\E|\widehat{\mu}_{M,1}| \to \infty$ as $q_n \to \infty$. Moreover, $|\widehat{\mu}_{M,1}| = \O(1)$, w.h.p.
\label{prop:HighExp2}
\end{proposition}

As seen by the example above, additional assumptions need to be imposed to prove the bounds in expectation. Suppose the variances $\{\sigma_i\}$ are all finite. We will consider two types of assumptions: either (i) ``high-noise'' points do not have very large variances, or (ii) ``low-noise'' points have small support.

We state a result for the modal interval estimator in $d$ dimensions; similar proofs hold for the shorth, median, and hybrid estimators. The following result is proved in Appendix~\ref{AppThmExpBound}.

\begin{theorem}
\label{ThmExpBound}
Let $nR^*_r = \Omega\(d\log n\)$. The following upper bounds hold for the expected error of the modal interval estimator:
\begin{itemize}
\item[(i)] Suppose
\begin{equation}
    \label{EqnSigmaGrowth}
\log\(\frac{\sigma_{(n)}}{r}\)  = \O\(nR^*_r\).
\end{equation}
Then the modal interval estimator satisfies the expected error bound
\begin{equation*}
\E\|\muhat_{M,r}\|_2 = \O\(r\left(\frac{c}{R^*_r}\right)^{1/d}\).
\end{equation*}
\item[(ii)] In the case $d = 1$, suppose the support of $\Omega(nR^*_r)$ points lies in $[-r,r]$. Then
\begin{equation*}
\E|\muhat_{M,r}| = \O\(\frac{r}{R^*_r}\).
\end{equation*}
\end{itemize}
\end{theorem}

\begin{remark*}
The condition~\eqref{EqnSigmaGrowth} in Theorem~\ref{ThmExpBound}(i) can be translated into the inequality $\sigma_{(n)} \le r \exp(CnR^*_r)$, and provides an upper bound on the variance of the worst mixture components. If we choose $r = \sigma_{(d \log n)}$, 
 we obtain the requirement that $\sigma_{(n)}$ is at most a factor of $\O(n^{Cd})$ larger than the variance $\sigma_{(d \log n)}$ of the ``good'' points. This can be compared to the assumption $\sigma_{(n)} = \sigma_{(1)}\text{poly}(n)$ imposed by Chierichetti et al.~\cite{ChiEtAl14} when proving upper bounds on expected error in the univariate case. As the proof of Theorem~\ref{ThmExpBound} reveals, we could also convert the tighter version of the estimation error guarantee (cf.\ Theorem~\ref{cor:modal} in the univariate setting) into an expected error bound in a similar manner: If condition~\eqref{EqnSigmaGrowth} holds in Theorem~\ref{ThmExpBound} and we additionally assume that $r' = \Omega(r)$, then $\E|\muhat_{M,r}| = \O(r')$.
\end{remark*}

Note that the condition in Theorem~\ref{ThmExpBound}(ii) imposes no constraints on the behavior of the large-variance mixture components. The proof proceeds by integrating the tail probability of the modal interval estimator, and showing that it must decay sufficiently quickly by considering the mass of intervals lying far from the true mean. An extension to the multivariate case is possible, but would require somewhat more refined technical analysis.

\subsection{Minimax bounds}
\label{SecOptimality}

We are now ready to discuss the optimality of our hybrid estimator, which we will consider in the context of expected error bounds. We state our results in the case of a general dimension $d \ge 1$. The goal of this section is to describe a general setting in which it is possible to show that the hybrid estimator is (nearly) minimax optimal.

We will consider the class of distributions $\scriptP(\sigma_1, \sigma_2, p)$, containing symmetric, unimodal distributions $\{P_i\}_{i=1}^n$ with common mean $\mu$, such that at least $np$ distributions have marginal variance bounded by $\sigma_2^2$ and the remaining distributions have marginal variance bounded by $\sigma_1^2$.  Note that $\sigma_1, \sigma_2$, and $p$ may all be functions of $n$, e.g., $p = \frac{\log n}{n}$.

We have the following minimax lower bound, proved in  Appendix~\ref{AppThmLowerBound}:

\begin{theorem}
\label{ThmLowerBound}
Suppose $p \le \frac{1}{3}$, $\sigma_2 \le \sigma_1$, and $p = \Omega\left(\frac{\log n}{n}\right)$.
\begin{itemize}
\item[(i)] The minimax error of any agnostic algorithm is
\begin{equation}
\label{EqnLB1}
\min_{\muhat} \max_{\{P_i\} \subseteq \scriptP(\sigma _1, \sigma _2, p)} \E\left[\|\muhat - \mu\|_2\right] \ge C_{\ell} \sqrt{d}\min\left\{\frac{\sigma_2}{\sqrt{np}},  \frac{\sigma_1}{\sqrt{n}}\right\}.
\end{equation}
\item[(ii)] In the case $d = 1$, suppose in addition we have
\begin{equation}
\label{EqnSratio}
\frac{\sigma_1}{\sigma_2} = O\left(\frac{1}{np^2}\right).
\end{equation}
Then
\begin{equation}
\label{EqnLB2}
\min_{\muhat} \max_{\{P_i\} \subseteq \scriptP(\sigma _1, \sigma _2, p)} \E\left[\|\muhat - \mu\|_2\right] \ge \frac{C_\ell'\sigma_1}{\sqrt{n}}.
\end{equation}
\end{itemize}
\end{theorem}

\begin{remark*}
In the $d = 1$ case, the lower bound in Theorem~\ref{ThmLowerBound} when condition~\eqref{EqnSratio} is satisfied matches the lower bound derived by Chierichetti et al.~\cite{ChiEtAl14}. On the other hand, our proof technique is somewhat more direct and proceeds via a straightforward (albeit lengthy) calculation.
\end{remark*}

We now state our general upper bound, achieved by the hybrid estimator. Under the specific regimes, we impose mild regularity conditions on the distributions to obtain cleaner expressions:
\begin{itemize}
\item[(i)] Let $q_i(x)$ denote the marginal distribution of $P_i$, where $q_i : \R \to \R$ (since $P_i$ is radially symmetric, all marginals are equal). Let $\nu_i^2$ denote the marginal variance of $P_i$. Then
\begin{align}
\qquad q_i(\nu_i) \geq \frac{c}{\nu_i}.
\label{EqnRegularityMarginal}
\end{align}
\item[(ii)] Let each density be written as $p_i(x) = f_i(\|x\|_2)$, where $f_i: \real \rightarrow \real$ is a decreasing function on the positive reals. Then
\begin{align}
f_i(0) \leq  \(\frac{c'}{\nu_i}\)^d, \quad \text{and} \quad \int_{ B ( K \sqrt{d} \nu_i, 2\sqrt{d} \nu_i)} p_i(y)dy  \leq C_1 \exp\(- C_2K^2\), \quad \forall  K \geq C_3 > 1.  
\label{EqnRegularityJoint}
\end{align}
\end{itemize}

Condition~\eqref{EqnRegularityMarginal} assumes that the marginal densities do not decrease too rapidly around the mean, and implies the accuracy of the median filtering step. Condition~\eqref{EqnRegularityJoint} assumes that the joint densities do not have too much mass concentrated around any single point (e.g., the mean), from which we may derive tighter error bounds on the shorth estimator when we have sufficiently separated variances, i.e., $\frac{\sigma_1}{\sigma_2} = \Omega\left(n^{1/d}\right)$. Note that conditions (i) and (ii) hold for Gaussian distributions; furthermore, condition (ii) holds more broadly when the norm of $p_i(\cdot)$ has right $ c'\nu_i \sqrt{d}$-sub-Gaussian tails around $\sqrt{d} \nu_i$. Then this expression can be upper bounded by $\P\{ \|X\| - \sqrt{d} \nu_i \geq  cK \sqrt{d} \nu_i \} \leq \exp(-c'K^2)$ using the sub-Gaussian assumption.

We also define $\cQ(\sigma_1, \sigma_2, p)$ to be the class of symmetric, unimodal distributions with  $\{P_i\}_{i=1}^n$ with common mean $\mu$,
 such that at least $np$ distributions have marginal variances bounded by $\sigma_2^2$ and remaining distributions have marginal variance at least $\Omega(\sigma_1^2)$ and at most $\sigma_1^2$. Thus, $\cQ(\sigma_1, \sigma_2 , p)$ is the class of distributions with sufficient division between high-variance and low-variance points, and we clearly have $\cQ(\sigma_1, \sigma_2,p ) \subseteq \scriptP(\sigma_1, \sigma_2, p)$.
Finally, in order to derive bounds in expectation, we impose the additional growth condition~\eqref{EqnSigmaGrowth} on the variance of the mixture components.

The following result is proved in Appendix~\ref{AppThmUpperBound}. 
\begin{theorem}
\label{ThmUpperBound}
If $p = \Omega\left(\frac{d\log n}{n}\right)$ and condition~\eqref{EqnRegularityMarginal} holds, then the hybrid estimator satisfies the upper bound
\begin{equation}
\label{EqnUB1}
\max_{\{P_i\} \subseteq \scriptP(\sigma_1, \sigma_2, p)} \E\left[\|\muhat - \mu\|_2\right] \le C_u \sqrt{d} \min\left\{\sqrt{n}^{1/d}\sigma_2, \frac{\log n}{\sqrt{n}} \sigma_1  \right\}.
\end{equation}
We also have the following special cases if we impose additional assumptions:
\begin{itemize}
\item[(a)] If $p = \Omega\left(\frac{\sqrt{n} \log n}{n}\right)$, we have the tighter bound
\begin{equation}
\label{EqnUB2}
\max_{\{P_i\} \subseteq \scriptP(\sigma_1, \sigma_2, p)} \E\left[\|\muhat - \mu\|_2\right] \le C_u'\sqrt{d}  \min\left\{\frac{ \log n}{p\sqrt{n}}\sigma_2, \frac{\log n}{\sqrt{n}} \sigma_1 \right\}.
\end{equation}
\item[(b)] If $
\frac{\sigma_1}{\sigma_2} = \Omega\(n^{\frac{1}{d}}\)$ and condition~\eqref{EqnRegularityJoint} holds, then
\begin{align}
\label{EqnUB3}
\max_{\{P_i\} \subseteq \cQ(\sigma_1, \sigma_2, p)} \E\left[\|\muhat - \mu\|_2\right] \le
C_u'' \sqrt{d} \min\left\{\sigma_2 \sqrt{\log n}, \frac{\log n}{\sqrt{n}} \sigma_1\right\}.
\end{align}
\end{itemize}
\end{theorem}

It is instructive to compare the upper bounds for the hybrid estimator in Theorem~\ref{ThmUpperBound} with the lower bounds derived in Theorem~\ref{ThmLowerBound}. (Note that the same class of distributions used to obtain the minimax lower bounds over $\cP$ falls into the class $\cQ$, so the lower bounds in Theorem~\ref{ThmUpperBound} may be directly compared with the bound~\eqref{EqnUB3}, as well.) In particular, we can see that the hybrid estimator is nearly minimax optimal in three somewhat different regimes of interest, which can be derived directly from the bounds in the theorems. The results are summarized in Table~\ref{TabOpt}:
\begin{enumerate}
\item Large heterogeneity: When $\sigma_1$ is very large compared to $\sigma_2$ and $p$ is very small, a direct application of the median would lead to large error. However, the shorth estimator is able to focus on the low-variance points due to the sufficiently large separation in variances. As $p$ becomes smaller, the gap between the upper and lower bounds reduces, reaching within $\log n$ factors when $p = \Theta\left(\frac{d \log n}{n}\right)$.
\item Mild heterogeneity: Since $\sigma_1$ is relatively small, the median and even mean are minimax optimal. The hybrid estimator is able to achieve these rates. (This includes the i.i.d.\ case.)
\item Large $p$: As $p$ increases, the number of good points increase and we expect to obtain vanishing error for reasonable values of $\sigma_1$ (e.g., under condition~\eqref{EqnSigmaGrowth}). Indeed, the hybrid estimator achieves vanishing error for large $p = \Omega\left(\frac{\sqrt{n} \log n}{n}\right)$ irrespective of magnitude of $\sigma_1$. Also, the gap between the upper bound and the lower bound decreases as either $p \to 1$ or $\sigma_1 \to \sigma_2$.
\end{enumerate}

\begin{table}[H]
\begin{tabular}{@{}cccc@{}}
\toprule
\multirow{2}{*}{} & \multicolumn{1}{c}{\multirow{2}{*}{\begin{tabular}[c]{@{}c@{}}Large heterogeneity\vspace*{0.3em}\\ $\cQ\( \Omega\(\frac{1}{ \sqrt{p}} + n^{1/d}\),1, o(n^{-0.5})\)$\end{tabular}}} & \multicolumn{1}{c}{\multirow{2}{*}{\begin{tabular}[c]{@{}c@{}}Mild heterogeneity\vspace*{0.3em}\\ $\cP\(\O\( \frac{1}{\sqrt{p}}\), 1,p \)$\end{tabular}}} & \multicolumn{1}{c}{\multirow{2}{*}{\begin{tabular}[c]{@{}c@{}}Large p \vspace*{0.3em}\\ $\cP\(\sigma_1, 1 , \Omega\(n^{-0.5}\)\)$\end{tabular}}} \\
                  & \multicolumn{1}{c}{}                                                                                                                                                 & \multicolumn{1}{c}{}                                                                                                                                              & \multicolumn{1}{c}{}                                                                                                \vspace*{1.5em}\\ \midrule
\vspace*{0.5em}Hybrid estimator  &  $\sqrt{d}$      &  $\frac{\sigma_1\sqrt{d} }{\sqrt{n}}$   &  $\sqrt{d} \min\left\{\frac{1}{p\sqrt{n}}, \frac{\sigma_1}{\sqrt{n}}\right\} $  \vspace*{0.5em}\\
\vspace*{0.5em} Lower bound       &  $\frac{\sqrt{d}}{\sqrt{np}}$      &   $\frac{\sigma_1 \sqrt{d}}{\sqrt{n}}$ & $\sqrt{d} \min\left\{\frac{1}{\sqrt{pn}}, \frac{\sigma_1}{\sqrt{n}}\right\}$  \vspace*{0.1em}\\ \bottomrule
\end{tabular}
\caption{Comparison of upper and lower bounds for estimation error, given by Theorems~\ref{ThmLowerBound} and~\ref{ThmUpperBound}, in three regimes of interest. For simplicity, we set $\sigma_2 = 1$ and ignore multiplicative factors which are  logarithmic in $n$.}
\label{TabOpt}
\end{table}

\begin{remark*}
Although we have shown that the hybrid estimator is indeed optimal in several diverse regimes, the preceding discussion leaves open the question of optimality in other settings. In particular, although our general upper bounds (e.g., inequality~\eqref{EqnUB1}) suggests the presence of a $\sqrt{n}^{1/d}$ factor when using the hybrid estimator, our lower bound techniques do not show that such a factor is unavoidable for $d \geq 2$.
As argued by Chierichetti et al.~\cite{ChiEtAl14}, a factor of $\sqrt{n}$ is unavoidable in $d=1$ (cf. Theorem~\ref{ThmLowerBound}). 
\end{remark*}

\section{Computation in high dimensions}
\label{SecComputation}

We now discuss how to make our estimators computationally feasible when $d$ is large. The main idea is that both the modal interval and shorth estimators involve finding optimal balls in $\real^d$. To save on computation, we will show that restricting the search to balls centered at one of the $n$ data points leads to estimators with similar performance guarantees. This is an idea previously introduced in the literature on mode estimation in i.i.d.\ scenarios~\cite{AbrEtal04, DasKpo14, Jia17}.

Concretely, the modal interval and shorth estimators are replaced by:
\begin{estimator}
The computationally efficient modal interval estimator is defined by
\begin{align}
\label{EqnModalEstComp}
\mutil_{M,r} \coloneqq \arg\max_{x \in \{x_1, \dots, x_n\}} R_n(f_{x,r}).
\end{align}
\end{estimator}

\begin{estimator}
The computationally efficient shorth estimator is defined by
\begin{align}
\label{EqnShorthEstComp}
\rtil_{k} &\coloneqq \inf_{r}\sup_{x \in \{x_1, \dots, x_n\}}\left\{R_n(f_{x,r}) \geq \frac{k}{n}\right\}, \qquad \mutil_{S,k} \coloneqq \mutil_{M, \rtil_{k}}.
\end{align}
In other words, we select the data point such that the smallest ball centered around that point containing at least $k$ points has the minimum radius.
\end{estimator}

Note that both estimators~\eqref{EqnModalEstComp} and~\eqref{EqnShorthEstComp} may be computed in $O(n^2d)$ time. In contrast, computing the modal interval or shorth estimators directly would correspond to solving the circle placement problem or smallest enclosing ball problem, for which the best-known exact algorithms are $\Omega(n^d)$~\cite{LeePer84, EppEri94, AgaSha98}.

Using a peeling argument~\cite{Van00}, we can obtain a more refined concentration result that Theorem~\ref{thm:highProbD}. The proof of the following result is contained in Appendix~\ref{AppThmUniformProb}. Note that the proof critically leverages radial symmetry of $R$, whereas the concentration inequalities in Lemmas~\ref{thm:highProb} and~\ref{thm:highProbD} do not require $R$ to be radially symmetric.

\begin{lemma}
\label{ThmUniformProb}
For any $t \in (0,1]$, radii $\bar{r}, r > 0$, and $n > 1$, we have the following inequalities:
\begin{align}
\label{EqnUniformProb1}
\mprob\Big(|R_n(f_{x,r}) - R(f_{x,r})| \le 2tR(f_{x,r}), \quad \forall x \text{ s.t. } \|x\|_2 \le \bar{r} \Big) & \ge 1 - \frac{2\exp(-cnt^2 R(f_{\bar{r},r}))}{1-\exp(-cnt^2 R(f_{\bar{r}, r}))}, \\
\label{EqnLargeX}
\mprob\left(\sup_{\|x\|_2 \ge \bar{r}} |R_n(f_{x,r}) - R(f_{x,r})| \ge t R(f_{\bar{r}, r})\right) & \le 2\exp(-cnt^2R(f_{\bar{r}, r})),
\end{align}
provided $\bar{r}$ and $r$ are such that $R(f_{\bar{r}, r}) \geq  \frac{C_t d \log n}{n}$.
\end{lemma}

Using Lemma~\ref{ThmUniformProb}, we can derive the following results for the computationally efficient modal interval and shorth estimators. The proof is contained in Appendix~\ref{AppThmModalCompute}.

\begin{theorem}
\label{ThmModalCompute}
For the computationally efficient estimators, we have the following error guarantees:
\begin{itemize}
\item[(i)] Suppose $r \ge 2r_{6Cd\log n}$. Then the modal interval estimator satisfies the bound $\|\mutil_{M,r}\|_2 \le 4r\left(\frac{n}{Cd\log n}\right)^{1/d}$, with probability at least $1-6\exp(-c_3 d\log n)$.
\item[(ii)] Suppose $k \ge 2C_{0.5}(d+1)\log n$. Then the shorth estimator satisfies the bound $\|\mutil_{S,k}\|_2 \le 4r_{2k}\left(\frac{2n}{k}\right)^{1/d}$, with probability at least $1-2\exp(-c'k)$.
\end{itemize}
\end{theorem}

\begin{remark*}
Comparing Theorem~\ref{ThmModalCompute}(i) with Theorem~\ref{LemModalD}, we see that the the computationally efficient modal interval essentially incurs an additional factor of 2 in the error bound, since we require $r \ge 2r_{C'd \log n}$. If we take $k = Cd\log n$, the error bound in Theorem~\ref{ThmModalCompute}(ii) is very similar to the error guarantee for the modal interval estimator~\eqref{EqnModalD} derived in Theorem~\ref{ThmShorthD}, except for an extra factor of 2.
\end{remark*}

Of course, the quality of the guarantee in Theorem~\ref{ThmModalCompute}(i) worsens as $r$ increases. Just as in the 1-dimensional case, we can use Lepski's method to calibrate the modal interval radius. Note that we can again use the shorth estimator to obtain rough upper and lower bounds. Using a similar argument as in the proof of Lemma~\ref{lemma:shorthlength}, we are guaranteed that $\frac{1}{2} \rtil_{3C d \log n} \le r_{6Cd\log n} \le \rtil_{6Cd\log n}$, w.h.p. Essentially the same argument as in Theorem~\ref{ThmLepski} then shows that the error of the modal interval estimator with Lepski calibration is guaranteed to be upper-bounded by $12r_{6C d\log n} \left(\frac{n}{C d\log n}\right)^{1/d}$.

As discussed in Section~\ref{sec:hybrid}, the projection step for the hybrid screening procedure can be computed in $O(d)$ time. The construction of the cuboid $S_k^\infty$ itself can clearly be computed in $O(nd)$ time. Thus, one can also easily obtain the $O(\sqrt{n}^{1/d})$ rates using a computationally efficient hybrid estimator, as well.

\section{Relaxing radial symmetry}
\label{SecAltCond}

We now consider the case when the population-level distribution $\overline{P} = \frac{1}{n} \sum_{i=1}^n P_i$ is not symmetric. In the case $d = 1$, we can obtain the same estimation error rates only assuming that density $p_i$ is log-concave with a unique mode at 0. In the case $d > 1$, we can obtain weaker estimation error guarantees of the order $O(\sqrt{n})$ rather than $O(\sqrt{n}^{1/d})$ if we only assume that the mixture components are centrally symmetric. Furthermore, it is possible to obtain $O(n^{1/d})$ rates if we assume that a certain fraction of the components are radially symmetric.

Although radial symmetry is a strict assumption,  it provides us an $\O(\sqrt{n}^{\frac{1}{d}})$ error. 
Whereas if we just assume central asymmetry, a union bound argument gives $\O(\sqrt{dn})$ error. 
This factor of $\O(\sqrt{dn})$ can not be improved in general. To see this, note that there exists a problem instance in single dimension where the lower bound is a factor of $\tilde{\Omega}(\sqrt{n})$. Central symmetry allows for having the same ``hard'' problem on each dimension separately, forcing an $\tilde{\Omega}(\sqrt{n})$ error in each dimension.

We can relax the radial symmetry assumptions slightly. In particular, Theorem~\ref{thm:highProb} only relies on the fact that $R^*_r$, the mass of the interval centered around the true mode 0, is $\Omega\left(\frac{\log n}{n}\right)$ (with no additional symmetry assumptions). We do need $R(f_{x,r})$ to satisfy some additional monotonicity assumptions along rays as $x$ moves away from 0.

\subsection{General theory}

In place of radial symmetry, we impose the following condition (stated with respect to a fixed radius $r$):
\begin{itemize}
\item[(C1)] The population-level quantity $R(f_{x,r})$ is maximized at $x = 0$, and otherwise monotonically decreasing along rays from the origin.
\end{itemize}
Note that condition (C1) is satisfied if the same property holds for all components $p_i$ in the mixture. We now define the function
\begin{equation}
\label{EqnGfunc}
g(a,r) = \sup_{\|x\|_2 = a} R(f_{x,r}),
\end{equation}
for $a,r > 0$. By Lemma~\ref{lemma:prop2}, we can argue that under radial symmetry of $R$, we have $g(a,r) \le \frac{1}{N(B_a, r)} \le \left(\frac{r}{a}\right)^d$, which can then be plugged into the argument of Theorem~\ref{LemModalD}. The proof of the following statement is contained in Appendix~\ref{AppThmRelax}.

\begin{theorem}
\label{ThmRelax}
Suppose condition (C1) holds.
\begin{itemize}
\item[(i)] Suppose $r$ is such that $R^*_r = \Omega\left(\frac{d \log n}{n}\right)$, and $r'$ is chosen sufficiently large such that $g(r', r) < \frac{R^*_r}{2}$. Then the modal interval estimator satisfies $\|\muhat_{M, r}\|_2 \le r'$, w.h.p.
\item[(ii)] Suppose
$r'$ is chosen such that $g(r', r_{8d \log n}) \le \frac{8d\log n}{4n}$. With high probability, the error of the shorth estimator satisfies $\|\muhat_{S, k}\|_2 \le r'$, and the error of the
hybrid algorithm with $k_2 = r_{8d\log n}$ is bounded by $ \min(r',\sqrt{d}r_{4\sqrt{n\log n},1})$.
\end{itemize}
\end{theorem}

\begin{remark*}
For radially symmetric distributions, note that $g(r', r) \le \left(\frac{r}{r'}\right)^d$, so we can take $r' = r \left(\frac{2}{R^*_r}\right)^{1/d}$ and $r' = r_{2k} \left(\frac{4}{R^*_{2k}}\right)^{1/d}$ to obtain the results of Theorems~\ref{LemModalD} and~\ref{ThmShorthD} for the modal interval and shorth estimators, respectively. Furthermore, by Lemma~\ref{lemma:prop2}\ref{prop2:3}, we have $r_{\sqrt{n}\log n} \leq \(\frac{\sqrt{n} }{8d}\)^{1/d}r_{8d\log n}$. Thus, we also recover the analog of Theorem~\ref{thm:HybridScreening} for the hybrid estimator.
\end{remark*}

Finally, note that an analog of Theorem~\ref{ThmRelax} holds when we use the computationally efficient modal interval and shorth estimators described in Section~\ref{SecComputation}, with minor proof modifications.

\subsection{Sufficient conditions}

Condition (C1) may be a bit difficult to interpret. We define two related conditions:
\begin{itemize}
\item[(C2)] Each component density $p_i$ is log-concave with a unique mode at 0.
\item[(C3)] For all $x \in \real^d$ and all $1 \le i \le n$, we have $p_i(x) = p_i(-x)$.
\end{itemize}
Note that condition (C3) only requires symmetry of the density around 0, rather than radial symmetry; in particular, it holds for Gaussian distributions that are not necessarily isotropic.

We have the following result, proved in Appendix~\ref{AppLemRmode}:

\begin{proposition}
\label{LemRmode}
Suppose conditions (C2) and (C3) hold. Then condition (C1) also holds. Furthermore, $g(a,r) \le \frac{1}{\lfloor a/2r \rfloor}$.
\end{proposition}

In fact, we can even derive a result only assuming condition (C2) in the case $d=1$. As argued in the proof of Theorem~\ref{ThmRelax}, we may establish that $R(f_{\muhat_{M,r}}, r) \ge \frac{R_r^*}{2}$, w.h.p. Thus, there exists some $i$ such that $R_i(f_{\muhat_{M,r}}, r) \ge \frac{R_r^*}{2}$. By properties of log-concave convolutions (cf.\ proof of Proposition~\ref{LemRmode}), we know that $R_i(f_{x,r})$ is decreasing along rays originating from some point $x_i^*$, and also $\|\muhat_{M,r} - x_i^*\|_2 \le \frac{4r}{R_r^*}$, since we could otherwise pack too many intervals into the ray between $x_i^*$ and $\muhat_{M,r}$, thus contradicting the inequality $R_i(f_{\muhat_{M,r}}, r) \ge \frac{R_r^*}{2}$. Finally, note that due to the unimodality of $p_i$ at 0, we clearly have $\|x_i^*\|_2 \le r$. Altogether, we obtain the error bound
\begin{equation*}
\|\muhat_{M,r}\|_2 \le \frac{4r}{R_r^*} + r,
\end{equation*}
which is of the same order as the guarantees in Theorem~\ref{cor:modal}. A similar conclusion could be reached if we replaced condition (C2) by the condition that each $p_i$ has a unique median and mode at 0, since $R_i(f_{x,r})$ is decreasing along rays originating from $r$ ($-r$) in the positive (negative) direction.

\subsection{Examples}

We now describe two examples to illustrate concrete use cases of our more general theory.

\begin{example}[Elliptically symmetric distributions]
\label{ExaElliptical}
We now consider the case where the components of the mixture are not spherical, but have the same axes of symmetry. Concretely, suppose that for a fixed matrix $\Sigma \succ 0$, the density of each $X_i$ is of the form $f_i\left((x - \mu)^T \Sigma^{-1} (x-\mu)\right)$, where $f_i: \real \rightarrow \real$ is a decreasing function defined on the positive reals. The goal is to estimate the common parameter $\mu \in \real^d$. As a specific example, we might have a mixture of nonisotropic Gaussian distributions where the covariance matrices are all scalar multiples of $\Sigma$. This strictly generalizes the case of radially symmetric distributions, which corresponds to the case $\Sigma = I$.

Suppose we employ the modal interval, shorth, or hybrid estimators described above. Note that these estimators do not require knowledge of the matrix $\Sigma$. We wish to analyze the behavior of the quantity $g(a,r)$ defined in equation~\eqref{EqnGfunc}, which is relevant for Theorem~\ref{ThmRelax}. Indeed, we can derive an analog of Lemma~\ref{lemma:prop2} that applies in this setting. The main step is to understand bound the quantity $g(r_2, r_1)$ when $r_1 < r_2$. We have the following result, proved in Appendix~\ref{AppPropElliptical}:

\begin{proposition}
\label{PropElliptical}
Let $r_1 < r_2$. For an elliptically symmetric distribution, we have
\begin{equation*}
g(r_2, r_1) \le C \left(\frac{r_1 \lambda_{\max}(\Sigma)}{r_2 \lambda_{\min}(\Sigma)}\right)^d.
\end{equation*}
\end{proposition}

Clearly, taking $C = 1$ and $\Sigma = I$ in Proposition~\ref{PropElliptical} recovers the result for radially symmetric distributions.
\end{example}

\begin{remark*}
Similar arguments as in Example~\ref{ExaElliptical} could be applied in the case when the probability density functions of the distributions are proportional to $\exp(-\|x - \mu\|/\sigma)$, for a different norm $\|\cdot\|$ besides the squared $\ell_2$-norm or the Mahalanobis norm. Also note that if the matrix $\Sigma$ (accordingly, the norm $\|\cdot\|$) were known a priori, it might be possible to obtain better rates by using a modal interval/shorth estimator based on the level sets of the norm rather than spheres of varying radii.
\end{remark*}

\begin{example} [Mixture of radially and centrally symmetric distributions]
For another interesting special case, suppose we have $s$ points drawn from radially symmetric distributions, and $n-s$ points drawn from centrally symmetric distributions.
Suppose we have $f(n)$ points which are well-behaved in the sense that the interquartile range of the corresponding distributions is small. (These distributions need not coincide with the radially symmetric distributions.) We have the following result, proved in Appendix~\ref{AppPropMixture}:

\begin{proposition}
\label{PropMixture}
For $r = q_{(f(n))}$ and $r' = 2r n^{1/d}$, we have
\begin{equation*}
g(r',r) \le \frac{R^*_r}{2},
\end{equation*}
provided $s \ge n - 2n^{1/d} (f(n)-4)$.
\end{proposition}

Thus, as the proportion of well-behaved points increases, the required proportion of radially symmetric distributions required to obtain a specific error guarantee becomes smaller. In particular, if $f(n) = \Omega(n^{1-1/d})$, we do not need any radially symmetric distributions; recall, however, that the coordinatewise median already performs well when on a mixture of centrally symmetric distributions if $f(n) = \Omega(\sqrt{n} \log n)$.
\end{example}

\section{Linear regression}
\label{SecRegression}

We now shift our focus to the problem of linear regression, and demonstrate how the methodology developed thus far may be adapted to parameter estimation in multivariate regression. Suppose we have observations $\{(x_i, y_i)\}_{i=1}^n$ from the linear model
\begin{equation*}
y_i = x_i^T \betastar + \epsilon_i, \qquad \forall 1 \le i \le n,
\end{equation*}
where the pairs $\{(x_i, \epsilon_i)\}_{i=1}^n$ are independent but not necessarily identically distributed, and $x_i$ and $\epsilon_i$ are independent for each $i$.

Following the theme of our paper, we assume that the $\epsilon_i$'s are symmetric and unimodal. We want to study the behavior of the modal interval regression estimator
\begin{equation}
\label{EqnRegModal}
\betahat = \argmax_{\beta \in \real^d} \frac{1}{n} \sum_{i=1}^n 1\left\{|y_ i - x_i^T \beta| \le r\right\},
\end{equation}
for an appropriate choice of $r > 0$.

A natural question is whether the true parameter $\betastar$ is the unique population-level maximizer in the regression setting. As the following proposition shows, this is indeed the case when the densities of the $x_i$'s are absolutely continuous with respect to Lebesgue measure. The proof is contained in Appendix~\ref{AppPropRegMax}.

\begin{proposition}
\label{PropRegMax}
The population-level maximizer is given by
\begin{equation}
\label{EqnBetaMax}
\betastar = \arg\max_\beta \sum_{i=1}^n \E\left[1\left\{|y_i - x_i^T \beta| \le r\right\}\right], \qquad \forall r > 0.
\end{equation}
\end{proposition}

Importantly, Proposition~\ref{PropRegMax}, and the ensuing theory, does not require specific assumptions on the form of the distribution of the $x_i$'s. However, in order to derive easily interpretable error bounds on the modal interval regression estimator, we will assume further distributional assumptions (cf.\ the statement of Theorem~\ref{ThmRegression} below).

\subsection{Estimation error}

In order to obtain error bounds on $\|\betahat - \betastar\|_2$, we need to analyze the behavior of the quantities
\begin{equation*}
R_\beta := \frac{1}{n} \sum_{i=1}^n \mprob\left(|y_i - x_i^T \beta| \le r\right),
\end{equation*}
for a fixed value of $r$, chosen sufficiently large that $R_{\betastar} \ge \frac{Cd\log n}{n}$. In particular, we want to show that for $\|\beta - \betastar\|_2$ larger than a certain value, we will have $R_\beta < \frac{R_{\betastar}}{2} = \frac{1}{2n} \sum_{i=1}^n \mprob(|\epsilon_i| \le r)$.

As before, the key ingredient for deriving error bounds is a uniform concentration result. This is proved in the following lemma:

\begin{lemma}
\label{LemRegConc}
Let $t \in (0,1]$, and suppose $r$ is large enough so that $R_{\betastar} \ge \frac{Cd\log n}{n}$. Then
\begin{multline}
\label{EqnConcLin}
\mprob\left(\sup_{\beta \in \real^d, r' \le r} \left|\frac{1}{n} \sum_{i=1}^n 1\left\{|y_i - x_i^T \beta| \le r'\right\} - \frac{1}{n} \sum_{i=1}^n \E\left[1\left\{|y_i - x_i^T \beta| \le r'\right\}\right]\right| \ge t R_{\betastar}\right) \\
\le 2\exp(-cnR_{\betastar} t^2).
\end{multline}
\end{lemma}

Since the proof is directly analogous to the proof of Theorem~\ref{thm:highProbD}, we only provide a sketch: The key point is to consider the VC dimension of the class of functions $f(x,y) = 1\{|y - x^T \beta| \le r\}$, indexed by the pair $(\beta, r)$. Note that the subset of points in $\real^{d+1}$ associated with the indicator function $f(x,y)$ is an intersection of two halfspaces. Using results on the VC dimension of an intersection of concept classes~\cite{VanWel09}, we see that the VC dimension of the desired hypothesis class is bounded by $C'd$. The concentration result then follows by the same arguments used to derive Theorem~\ref{thm:highProbD}.

It is generally difficult to state general bounds on estimation error that depend only on order statistics of quantiles, since as in the case of mean regression, the error bounds one can derive will be largely problem-dependent. In order to simplify our presentation, we will only discuss the case where the $\epsilon_i$'s and $x_i$'s are Gaussian: $\epsilon_i \sim N(0, \sigma_i^2)$ and $x_i \sim N(\mu_i', \Sigma_i')$. We have the following result, proved in Appendix~\ref{AppThmRegression}:

\begin{theorem}
\label{ThmRegression}
Let $\lambda_{\min} := \min_i \lambda_{\min}(\Sigma_i')$, and suppose $\lambda_{\min} > 0$. Suppose $r > 0$ is chosen such that $R_{\betastar} \ge \frac{Cd \log n}{n}$. Then the regression estimator~\eqref{EqnRegModal} satisfies
\begin{equation*}
\|\betahat - \betastar\|_2 \le \frac{c'n \sigma_{(cd\log n)}}{\sqrt{\lambda_{\min}}},
\end{equation*}
w.h.p.
\end{theorem}

We conjecture that it is possible to decrease this upper bound to $\O(\sqrt{n} \sigma_{(c\log n)})$ by an appropriate hybrid screening procedure, but we leave this to future work. Also note that in order for the bound in Theorem~\ref{ThmRegression} to be useful, the quantity $\lambda_{\min}$ must either be a constant, or else not decrease too rapidly with $n$.

\subsection{Computation}
\label{SecRegComputation}

A natural question is whether the modal interval regression estimator~\eqref{EqnRegModal} is actually computationally feasible. We claim that an estimator may be obtained in $O(n^d)$ time, using Algorithm~\ref{alg:modalReg}. The proof is in Appendix~\ref{AppThmCompReg}.

\begin{algorithm}[h]  
  \caption{Modal interval regression estimator
    \label{alg:modalReg}}  
  \begin{algorithmic}[1]  
    \Statex  
    \Function{modalIntervalRegression}{$X_{1:n}, Y_{1:n}, r, d$}  
        \State Construct the set of hyperplanes
        \begin{equation*}
        \mathcal{S}_r = \{y_i = x_i^T \beta + r\} \bigcup \{y_i = x_i^T \beta - r\}.
        \end{equation*}
    \State Let $\{S_1, \dots, S_N\}$ denote the set of subsets of $\mathcal{S}_r$ of cardinality $d$.
    \For{$j = 1, \dots, N$}
    \State Solve the system of linear equations given by $S_j$. Let $\beta_j$ be a solution (if one exists).
    \EndFor
    \State $j^* \gets \arg\max_{1 \le j \le N} \frac{1}{n} \sum_{i=1}^n 1\left\{|y_i - x_i^T \beta_j| \le r\right\}$.
    \State \Return $\beta_{j^*}$
    \EndFunction  
  \end{algorithmic}  
\end{algorithm}

\begin{theorem}
\label{ThmCompReg}
The output of Algorithm~\ref{alg:modalReg} is a maximizer of equation~\eqref{EqnRegModal}.
\end{theorem}

\begin{remark*}
Correct application of Algorithm~\ref{alg:modalReg} would assume that $r$ is chosen appropriately. It is less clear how this parameter might be calibrated based on the data, perhaps using an appropriate variant of Lepski's method. We leave this important open question to future work.
\end{remark*}

\section{Simulations}
\label{SecSims}

We now present the results of simulations on the recurring examples to validate our theoretical predictions (cf.\ Table~\ref{table: ex123}).
Although our theorem statements involve large constants, we empirically observe that smaller constants suffice to elicit the same behavior predicted by our theory. 
We run the $k$-shorth estimator with $k= 5d\log n$ and $k$-median with $k = \sqrt{n}\log n$. We use these estimators for the hybrid estimator, i.e., the $(\sqrt{n}\log n, 5d \log n)$-hybrid estimator. The mean estimator corresponds to the simple average, whereas the median estimator refers to the (coordinatewise) sample median.

For each $n$, we run $T = 200$ simulations for univariate data and $T= 20$ simulation for multivariate data and report the average error $\frac{1}{T}\sum_{i=1}^T |\muhat - \mustar|$ of various estimators. Both axes in all of the plots are in a $\log$-scale. In particular, the slope of the curves indicates the power of $n$ in the estimation error, and vertical shifts correspond to constant prefactors.

\subsection{Univariate data}

We first present simulation results when $d = 1$. We use $r=1$ for the simulations involving $r$-modal interval estimators, since $R^*_1 = \Omega\left(\frac{\log n}{n}\right)$ in each of the recurring examples, although the constant prefactors do not exactly align with our theory.

In the case of Example~\ref{exm:IID} (i.i.d.\ observations), we generate $x_i \stackrel{i.i.d.}{\sim} \cN(0,1)$. As seen in Figure~\ref{fig:iid}(a), the mean and median estimator perform optimally in this setting, giving an error rate of $\O(n^{-0.5})$. In contrast, the shorth estimator (with $k = 5\log n)$ has a flat trend line indicative of constant error, as suggested by Remark~\ref{remark:shorth} and the phase transition arguments in Section~\ref{SecPhase}.
On the other hand, the error of the hybrid estimator decays at a rate more comparable to the mean and median. As discussed in Remark~\ref{RemHybrid}, the hybrid estimator is indeed optimal up to log factors. We see that the performance of the modal interval estimator is better than the shorth but worse than the hybrid estimator, and exhibits the cube-root asymptotic decay encountered in classical statistics~\cite{KimPol90}. Furthermore, the estimation error of the hybrid estimator behaves like more like the error of the median estimator as $n$ increases. Note that although our bounds for the shorth and modal interval estimators are tighter for smaller values of $k$ and $r$, choosing larger values results in better performance when the data are homogeneous, which is not a valid assumption in our general use case. 

For Example~\ref{exm:QuadVar} (quadratic variance), we generate $x_i \sim \mathcal{N}(0, i^2)$.
In Figure~\ref{fig:quadratic}(a), we see that the both the median and mean have similar slopes: 
Proposition~\ref{prop:exmMedian} predicts that the median would have $\tilde{\O}(\sqrt{n})$ error, compared to the $\Theta\( \sqrt{\frac{1}{n^2}\sum_{i=1}^n i^2}\) = \Theta\(\sqrt{n}\)$ error of the mean; indeed, the curves are roughly parallel.
 However, the error rate of the modal interval, shorth, and hybrid estimators is significantly smaller. As stated in Propositions~\ref{prop:exmModal} and~\ref{prop:exmHybrid}, the error of these estimators is upper-bounded by $\O(n^{\epsilon})$, for $\epsilon > 0$.

 \begin{figure}[!ht]
 \centering
    \begin{minipage}{0.45\textwidth}
        \centering
		\scalebox{0.85}{
\begin{tikzpicture}

\definecolor{color0}{rgb}{0,0.75,0.75}

\begin{axis}[
legend style={cells={align=left}},
log basis x={10},
log basis y={10},
tick align=outside,
tick pos=left,
x grid style={white!69.01960784313725!black},
xlabel={n},
xmin=23987.5800423731, xmax=22918352.4522639,
xmode=log,
xtick style={color=black},
xtick={1000,10000,100000,1000000,10000000,100000000,1000000000},
xticklabels={$\displaystyle {10^{3}}$,$\displaystyle {10^{4}}$,$\displaystyle {10^{5}}$,$\displaystyle {10^{6}}$,$\displaystyle {10^{7}}$,$\displaystyle {10^{8}}$,$\displaystyle {10^{9}}$},
y grid style={white!69.01960784313725!black},
ylabel={Average error},
ymin=0.000141753284800585, ymax=0.254801585554501,
ymode=log,
ytick style={color=black},
ytick={1e-05,0.0001,0.001,0.01,0.1,1,10},
yticklabels={$\displaystyle {10^{-5}}$,$\displaystyle {10^{-4}}$,$\displaystyle {10^{-3}}$,$\displaystyle {10^{-2}}$,$\displaystyle {10^{-1}}$,$\displaystyle {10^{0}}$,$\displaystyle {10^{1}}$}
]
\addplot [thick, color0, opacity=0.55, mark=square*, mark size=4, mark options={solid}]
table {%
32768 0.00431820489616046
65536 0.00294473338316113
131072 0.00210036884509682
262144 0.00154749476799129
524288 0.00114016408609987
1048576 0.000745498301899608
2097152 0.000524138826403824
4194304 0.000401923414218378
8388608 0.000253173850959438
16777216 0.000199284091538182
};
\addplot [thick, black, opacity=0.55, dash pattern=on 1pt off 3pt on 3pt off 3pt, mark=asterisk, mark size=4, mark options={solid}]
table {%
32768 0.0052398405505227
65536 0.00375379089805843
131072 0.00255304930784307
262144 0.001906306188697
524288 0.00141963127835771
1048576 0.000974217027526744
2097152 0.000639875202429221
4194304 0.000478914579066753
8388608 0.000326605229428896
16777216 0.000234490659749969
};
\addplot [thick, green!50.0!black, opacity=0.55, dotted, mark=diamond*, mark size=4, mark options={solid}]
table {%
32768 0.0262841340092873
65536 0.0231013546958301
131072 0.0159961746782273
262144 0.0144889598538098
524288 0.00998830857362418
1048576 0.00795205955049899
2097152 0.00666020991626187
4194304 0.00500547060822513
8388608 0.00388894814489549
16777216 0.00346544309744676
};
\addplot [thick, blue, opacity=0.55, mark=triangle*, mark size=4, mark options={solid}]
table {%
32768 0.181243577678289
65536 0.178145371830435
131072 0.171823536588354
262144 0.14558670958605
524288 0.151740394738791
1048576 0.144006992803002
2097152 0.143993111056508
4194304 0.149295958522974
8388608 0.133297407445264
16777216 0.1348680137637
};
\addplot [thick, red, opacity=0.55, dashed, mark=*, mark size=4, mark options={solid}]
table {%
32768 0.0636308847986039
65536 0.0501564356678451
131072 0.0383318587053127
262144 0.0280383139209902
524288 0.0213797189914105
1048576 0.0165780745664562
2097152 0.0122952006577982
4194304 0.00901785745150496
8388608 0.0067631228307064
16777216 0.00500146406582579
};
\end{axis}

\end{tikzpicture}}
        \caption*{(a) $d = 1$}
        \label{fig:prob1_iid}
    \end{minipage}%
    \begin{minipage}{.55\textwidth}
        \centering
		\scalebox{0.85}{
\begin{tikzpicture}

\definecolor{color0}{rgb}{0,0.75,0.75}

\begin{axis}[
legend style={cells={align=left}},
legend cell align={left},
legend style={at={(1.04,1)}, anchor=north west, draw=white!80.0!black},
log basis x={10},
log basis y={10},
tick align=outside,
tick pos=left,
x grid style={white!69.01960784313725!black},
xlabel={n},
xmin=3444.31171687921, xmax=155871.754977636,
xmode=log,
xtick style={color=black},
xtick={100,1000,10000,100000,1000000,10000000},
xticklabels={$\displaystyle {10^{2}}$,$\displaystyle {10^{3}}$,$\displaystyle {10^{4}}$,$\displaystyle {10^{5}}$,$\displaystyle {10^{6}}$,$\displaystyle {10^{7}}$},
y grid style={white!69.01960784313725!black},
ylabel={Average error},
ymin=0.00356369627736361, ymax=0.362498779152629,
ymode=log,
ytick style={color=black},
ytick={0.0001,0.001,0.01,0.1,1,10},
yticklabels={$\displaystyle {10^{-4}}$,$\displaystyle {10^{-3}}$,$\displaystyle {10^{-2}}$,$\displaystyle {10^{-1}}$,$\displaystyle {10^{0}}$,$\displaystyle {10^{1}}$}
]
\addplot [thick, color0, opacity=0.55, mark=square*, mark size=4, mark options={solid}]
table {%
4096 0.028233404747043
8192 0.0170238568701205
16384 0.0140836494618674
32768 0.00883181508218745
65536 0.00586011597009259
131072 0.00439689813789753
};
\addlegendentry{mean           }
\addplot [thick, black, opacity=0.55, dash pattern=on 1pt off 3pt on 3pt off 3pt, mark=asterisk, mark size=4, mark options={solid}]
table {%
4096 0.0330726928646258
8192 0.0268926572020795
16384 0.0166134192290906
32768 0.0105596906741519
65536 0.00660727174964788
131072 0.00584619778199807
};
\addlegendentry{median         }
\addplot [thick, green!50.0!black, opacity=0.55, dotted, mark=diamond*, mark size=4, mark options={solid}]
table {%
4096 0.165769211343611
8192 0.152528294298537
16384 0.113350993541811
32768 0.0690610411842275
65536 0.0754366314057041
131072 0.0622708017603675
};
\addlegendentry{modal\\interval}
\addplot [thick, blue, opacity=0.55, mark=triangle*, mark size=4, mark options={solid}]
table {%
4096 0.275365610477995
8192 0.2938061126958
16384 0.273744456425507
32768 0.289738397866868
65536 0.267898656153967
131072 0.216281203582755
};
\addlegendentry{shorth         }
\addplot [thick, red, opacity=0.55, dashed, mark=*, mark size=4, mark options={solid}]
table {%
4096 0.20377882032764
8192 0.17971419453546
16384 0.135168690389512
32768 0.11149272048962
65536 0.0872777643802072
131072 0.0636357222605734
};
\addlegendentry{hybrid         }
\end{axis}

\end{tikzpicture}}
        \caption*{(b) $d = 3$}
        \label{fig:prob2_iid}
    \end{minipage}%
\caption{Plot comparing average error of various estimators on Example~\ref{exm:IID}. Both the mean and median exhibit the familiar $\O(n^{-0.5})$ error rate. The modal interval has errors of order $n^{-1/3}$. As suggested by our theoretical bounds, the ($\log n$)-shorth has constant error. The hybrid estimator improves the rate of the shorth estimator, with a similar error decay as the median estimator as $n$ increases. 
}
	\label{fig:iid}
	 \end{figure}
	 
 \begin{figure}[!ht]
 \centering
    \begin{minipage}{0.45\textwidth}
        \centering
		\scalebox{0.85}{
\begin{tikzpicture}

\definecolor{color0}{rgb}{0,0.75,0.75}

\begin{axis}[
legend style={cells={align=left}},
log basis x={10},
log basis y={10},
tick align=outside,
tick pos=left,
x grid style={white!69.01960784313725!black},
xlabel={n},
xmin=23987.5800423731, xmax=22918352.4522639,
xmode=log,
xtick style={color=black},
xtick={1000,10000,100000,1000000,10000000,100000000,1000000000},
xticklabels={$\displaystyle {10^{3}}$,$\displaystyle {10^{4}}$,$\displaystyle {10^{5}}$,$\displaystyle {10^{6}}$,$\displaystyle {10^{7}}$,$\displaystyle {10^{8}}$,$\displaystyle {10^{9}}$},
y grid style={white!69.01960784313725!black},
ylabel={Average error},
ymin=0.755034362166304, ymax=221.2556729081,
ymode=log,
ytick style={color=black},
ytick={0.01,0.1,1,10,100,1000,10000},
yticklabels={$\displaystyle {10^{-2}}$,$\displaystyle {10^{-1}}$,$\displaystyle {10^{0}}$,$\displaystyle {10^{1}}$,$\displaystyle {10^{2}}$,$\displaystyle {10^{3}}$,$\displaystyle {10^{4}}$}
]
\addplot [thick, color0, opacity=0.55, mark=square*, mark size=4, mark options={solid}]
table {%
32768 8.13817138927236
65536 12.5506856130617
131072 17.5924180221759
262144 23.7558672895694
524288 32.9074336372165
1048576 48.2858095256746
2097152 69.926147377908
4194304 103.199911913929
8388608 148.388813997676
16777216 170.907585018099
};
\addplot [thick, black, opacity=0.55, dash pattern=on 1pt off 3pt on 3pt off 3pt, mark=asterisk, mark size=4, mark options={solid}]
table {%
32768 2.12777531913888
65536 2.88896538600179
131072 4.34959910831692
262144 5.83918067016443
524288 7.96418235632985
1048576 10.2051284203469
2097152 13.9523877772845
4194304 18.7488555034661
8388608 27.4100011237681
16777216 32.9654692822574
};
\addplot [thick, green!50.0!black, opacity=0.55, dotted, mark=diamond*, mark size=4, mark options={solid}]
table {%
32768 1.02130487262952
65536 1.40920915163289
131072 1.10142727728965
262144 1.33889225536346
524288 1.15730978508123
1048576 1.66591236964156
2097152 1.6121914199101
4194304 1.46173498703289
8388608 2.28016753167405
16777216 2.22319713530454
};
\addplot [thick, blue, opacity=0.55, mark=triangle*, mark size=4, mark options={solid}]
table {%
32768 0.977461801078943
65536 1.49302477004359
131072 1.48445639501344
262144 1.39658036865092
524288 1.36333088761433
1048576 2.23683123736888
2097152 2.42028177233781
4194304 2.79162800402465
8388608 4.33146158473558
16777216 2.89886740788448
};
\addplot [thick, red, opacity=0.55, dashed, mark=*, mark size=4, mark options={solid}]
table {%
32768 0.977461801078943
65536 1.49302477004359
131072 1.48445639501344
262144 1.39658036865092
524288 1.36333088761433
1048576 2.23683123736888
2097152 2.42028177233781
4194304 2.79162800402465
8388608 4.33146158473558
16777216 2.89886740788448
};
\end{axis}

\end{tikzpicture}}
        \caption*{(a) $d = 1$}
        \label{fig:prob1_quad}
    \end{minipage}%
    \begin{minipage}{.55\textwidth}
        \centering
		\scalebox{0.85}{
\begin{tikzpicture}

\definecolor{color0}{rgb}{0,0.75,0.75}

\begin{axis}[
legend style={cells={align=left}},
legend cell align={left},
legend style={at={(1.04,1)}, anchor=north west, draw=white!80.0!black},
log basis x={10},
log basis y={10},
tick align=outside,
tick pos=left,
x grid style={white!69.01960784313725!black},
xlabel={n},
xmin=3444.31171687921, xmax=155871.754977636,
xmode=log,
xtick style={color=black},
xtick={100,1000,10000,100000,1000000,10000000},
xticklabels={$\displaystyle {10^{2}}$,$\displaystyle {10^{3}}$,$\displaystyle {10^{4}}$,$\displaystyle {10^{5}}$,$\displaystyle {10^{6}}$,$\displaystyle {10^{7}}$},
y grid style={white!69.01960784313725!black},
ylabel={Average error},
ymin=0.206538183659624, ymax=44.7729589798566,
ymode=log,
ytick style={color=black},
ytick={0.01,0.1,1,10,100,1000},
yticklabels={$\displaystyle {10^{-2}}$,$\displaystyle {10^{-1}}$,$\displaystyle {10^{0}}$,$\displaystyle {10^{1}}$,$\displaystyle {10^{2}}$,$\displaystyle {10^{3}}$}
]
\addplot [thick, color0, opacity=0.55, mark=square*, mark size=4, mark options={solid}]
table {%
4096 5.74986091718944
8192 9.4430602143313
16384 12.3547861170069
32768 16.0551043087852
65536 22.8598537365641
131072 35.061724833351
};
\addlegendentry{mean           }
\addplot [thick, black, opacity=0.55, dash pattern=on 1pt off 3pt on 3pt off 3pt, mark=asterisk, mark size=4, mark options={solid}]
table {%
4096 2.05400114720653
8192 2.34073018899189
16384 4.04537657901854
32768 5.2718340482609
65536 6.17355332448955
131072 8.75670667025558
};
\addlegendentry{median         }
\addplot [thick, green!50.0!black, opacity=0.55, dotted, mark=diamond*, mark size=4, mark options={solid}]
table {%
4096 0.398339040923116
8192 0.265515417612959
16384 0.363501769620111
32768 0.399054717753222
65536 0.263744173132359
131072 0.300553876841097
};
\addlegendentry{modal\\interval}
\addplot [thick, blue, opacity=0.55, mark=triangle*, mark size=4, mark options={solid}]
table {%
4096 3.80067921249015
8192 4.56198380502765
16384 5.58704055115752
32768 6.47003995492287
65536 5.19229331547889
131072 5.68624497198428
};
\addlegendentry{shorth         }
\addplot [thick, red, opacity=0.55, dashed, mark=*, mark size=4, mark options={solid}]
table {%
4096 3.80067921249015
8192 4.56198380502765
16384 5.58704055115752
32768 6.47003995492287
65536 5.19229331547889
131072 5.68624497198428
};
\addlegendentry{hybrid         }
\end{axis}

\end{tikzpicture}}
        \caption*{(b) $d = 3$}
        \label{fig:prob2_quad}
    \end{minipage}%
\caption{
Plot comparing average error of various estimators on Example~\ref{exm:QuadVar}.
As mentioned in Table~\ref{table: ex123}, both the mean and median have $\sqrt{n}$ error rate.
The error rates of the modal interval, shorth (with $k = 5d\log n$), and hybrid estimators are superior to the median in the univariate case, and the hybrid estimator is clearly superior when $d = 3$. 
 }
\label{fig:quadratic}
 \end{figure}
 
For Example~\ref{exm:alpha-mix} ($\alpha$-mixture distributions), we generate $\lceil 10\log n \rceil$ samples from a $\mathcal{N}(0,4 \times 10^{-4})$ distribution and the remaining samples from a $\mathcal{N}(0, n^\alpha)$ distribution, with $\alpha = 0.9$ and $1.3$. The plots in Figure~\ref{fig:alpha-mix} add additional curves to the phase transition plots in Figure~\ref{FigPhase}. As suggested by Propositions~\ref{prop:exmModal}~and~\ref{prop:exmHybrid}, the modal, shorth, and hybrid estimators have constant error for $\alpha > 1$, whereas the error increases with $n$ when $\alpha < 1$. Furthermore, the hybrid estimator performs better than the shorth estimator when $\alpha < 1$, with an error rate of $\O(n^{\alpha - 0.5})$ rather than $\O(n^\alpha)$, while the modal interval estimator seems to perform comparably to the hybrid. Finally, note that the behavior of the hybrid estimator is similar to the behavior of the median estimator when $\alpha < 1$ and to the modal interval/shorth estimator when $\alpha > 1$, showing that it indeed enjoys the better of the two rates in different regimes.

 \begin{figure}[!ht]
    \centering
    \begin{minipage}{0.45\textwidth}
		\scalebox{0.85}{
\begin{tikzpicture}

\definecolor{color0}{rgb}{0,0.75,0.75}

\begin{axis}[
legend style={cells={align=left}},
log basis x={10},
log basis y={10},
tick align=outside,
tick pos=left,
x grid style={white!69.01960784313725!black},
xlabel={n},
xmin=23987.5800423731, xmax=22918352.4522639,
xmode=log,
xtick style={color=black},
xtick={1000,10000,100000,1000000,10000000,100000000,1000000000},
xticklabels={$\displaystyle {10^{3}}$,$\displaystyle {10^{4}}$,$\displaystyle {10^{5}}$,$\displaystyle {10^{6}}$,$\displaystyle {10^{7}}$,$\displaystyle {10^{8}}$,$\displaystyle {10^{9}}$},
y grid style={white!69.01960784313725!black},
ylabel={Average error},
ymin=23.7523866593611, ymax=1241385.96644655,
ymode=log,
ytick style={color=black},
ytick={1,10,100,1000,10000,100000,1000000,10000000,100000000},
yticklabels={$\displaystyle {10^{0}}$,$\displaystyle {10^{1}}$,$\displaystyle {10^{2}}$,$\displaystyle {10^{3}}$,$\displaystyle {10^{4}}$,$\displaystyle {10^{5}}$,$\displaystyle {10^{6}}$,$\displaystyle {10^{7}}$,$\displaystyle {10^{8}}$}
]
\addplot [thick, color0, opacity=0.55, mark=square*, mark size=4, mark options={solid}]
table {%
32768 51.2980982709956
65536 67.216274420189
131072 84.6022785993172
262144 122.305332313158
524288 175.355832884497
1048576 209.722970395916
2097152 261.598016871146
4194304 377.493804281258
8388608 420.299399706156
16777216 616.773759258315
};
\addplot [thick, black, opacity=0.55, dash pattern=on 1pt off 3pt on 3pt off 3pt, mark=asterisk, mark size=4, mark options={solid}]
table {%
32768 39.8802486933979
65536 57.3435542841997
131072 71.0605855970876
262144 118.944507848551
524288 171.473092644633
1048576 220.165850491761
2097152 311.906335863327
4194304 420.054386978455
8388608 513.651705726343
16777216 808.097778530003
};
\addplot [thick, green!50.0!black, opacity=0.55, dotted, mark=diamond*, mark size=4, mark options={solid}]
table {%
32768 2499.78130755391
65536 5480.25850517543
131072 9443.11581242747
262144 17715.8019524889
524288 33879.924192391
1048576 67693.6455044011
2097152 127073.378496186
4194304 215409.616342662
8388608 432880.37231722
16777216 757605.666574019
};
\addplot [thick, blue, opacity=0.55, mark=triangle*, mark size=4, mark options={solid}]
table {%
32768 58.5753280913661
65536 108.42613076394
131072 339.129034220184
262144 295.790184594861
524288 764.052034223845
1048576 3181.46981631791
2097152 5674.69604046946
4194304 20998.0578832042
8388608 45899.0356816499
16777216 72563.3295813787
};
\addplot [thick, red, opacity=0.55, dashed, mark=*, mark size=4, mark options={solid}]
table {%
32768 40.4485526123243
65536 38.9198243485712
131072 102.035823494571
262144 113.034227742584
524288 117.766370527991
1048576 316.606700835632
2097152 541.502453208356
4194304 1212.34969673444
8388608 2030.94953471995
16777216 3745.16421281304
};
\end{axis}

\end{tikzpicture}}
        \caption*{(a) $d = 1, \alpha = 0.9$}
        \label{fig:prob1_al_1}
    \end{minipage}%
    \begin{minipage}{.55\textwidth}
		\scalebox{0.85}{
\begin{tikzpicture}

\definecolor{color0}{rgb}{0,0.75,0.75}

\begin{axis}[
legend style={cells={align=left}},
legend cell align={left},
legend style={at={(1.04,1)}, anchor=north west, draw=white!80.0!black},
log basis x={10},
log basis y={10},
tick align=outside,
tick pos=left,
x grid style={white!69.01960784313725!black},
xlabel={n},
xmin=23987.5800423731, xmax=22918352.4522639,
xmode=log,
xtick style={color=black},
xtick={1000,10000,100000,1000000,10000000,100000000,1000000000},
xticklabels={$\displaystyle {10^{3}}$,$\displaystyle {10^{4}}$,$\displaystyle {10^{5}}$,$\displaystyle {10^{6}}$,$\displaystyle {10^{7}}$,$\displaystyle {10^{8}}$,$\displaystyle {10^{9}}$},
y grid style={white!69.01960784313725!black},
ylabel={Average error},
ymin=5.04320240278835, ymax=1015765.91435884,
ymode=log,
ytick style={color=black},
ytick={0.1,1,10,100,1000,10000,100000,1000000,10000000,100000000},
yticklabels={$\displaystyle {10^{-1}}$,$\displaystyle {10^{0}}$,$\displaystyle {10^{1}}$,$\displaystyle {10^{2}}$,$\displaystyle {10^{3}}$,$\displaystyle {10^{4}}$,$\displaystyle {10^{5}}$,$\displaystyle {10^{6}}$,$\displaystyle {10^{7}}$,$\displaystyle {10^{8}}$}
]
\addplot [thick, color0, opacity=0.55, mark=square*, mark size=4, mark options={solid}]
table {%
32768 3117.88526486187
65536 5858.55305807642
131072 9757.89909820569
262144 18355.1928127732
524288 27146.5219157093
1048576 50582.7536737193
2097152 80468.9743262579
4194304 155594.194263234
8388608 242510.136089129
16777216 467077.341691056
};
\addlegendentry{mean           }
\addplot [thick, black, opacity=0.55, dash pattern=on 1pt off 3pt on 3pt off 3pt, mark=asterisk, mark size=4, mark options={solid}]
table {%
32768 1899.62477958624
65536 3837.83715862428
131072 7040.30643104309
262144 17861.3193141773
524288 29288.0397077553
1048576 52996.7244295321
2097152 97362.3295180963
4194304 173385.381945333
8388608 306903.077152158
16777216 583040.550335228
};
\addlegendentry{median         }
\addplot [thick, green!50.0!black, opacity=0.55, dotted, mark=diamond*, mark size=4, mark options={solid}]
table {%
32768 886.745601249071
65536 5937.05447786775
131072 20.7883023368208
262144 20.7084387635399
524288 20.7696500393682
1048576 22.0661670294508
2097152 19.9185432914994
4194304 20.9814336260215
8388608 19.6172505429853
16777216 21.3627599366926
};
\addlegendentry{modal\\interval}
\addplot [thick, blue, opacity=0.55, mark=triangle*, mark size=4, mark options={solid}]
table {%
32768 9.70143818875214
65536 11.0626772337905
131072 10.6438349227722
262144 9.99767483931799
524288 9.52100725177815
1048576 10.9123957491632
2097152 9.85560412103858
4194304 8.87911834807886
8388608 10.2525479050476
16777216 8.78620380181044
};
\addlegendentry{shorth         }
\addplot [thick, red, opacity=0.55, dashed, mark=*, mark size=4, mark options={solid}]
table {%
32768 9.70143818875214
65536 11.0626772337905
131072 10.6438349227722
262144 9.99767483931799
524288 9.52100725177815
1048576 10.9123957491632
2097152 9.85560412103858
4194304 8.87911834807886
8388608 10.2525479050476
16777216 8.78620380181044
};
\addlegendentry{hybrid         }
\end{axis}

\end{tikzpicture}}
        \caption*{(b) $d = 1, \alpha = 1.3$}
        \label{fig:prob1_al_2}
    \end{minipage}
        \begin{minipage}{0.45\textwidth}
		\scalebox{0.85}{
\begin{tikzpicture}

\definecolor{color0}{rgb}{0,0.75,0.75}

\begin{axis}[
legend style={cells={align=left}},
log basis x={10},
log basis y={10},
tick align=outside,
tick pos=left,
x grid style={white!69.01960784313725!black},
xlabel={n},
xmin=3444.31171687921, xmax=155871.754977636,
xmode=log,
xtick style={color=black},
xtick={100,1000,10000,100000,1000000,10000000},
xticklabels={$\displaystyle {10^{2}}$,$\displaystyle {10^{3}}$,$\displaystyle {10^{4}}$,$\displaystyle {10^{5}}$,$\displaystyle {10^{6}}$,$\displaystyle {10^{7}}$},
y grid style={white!69.01960784313725!black},
ylabel={Average error},
ymin=0.024873770482471, ymax=2.79921419036637,
ymode=log,
ytick style={color=black},
ytick={0.001,0.01,0.1,1,10,100},
yticklabels={$\displaystyle {10^{-3}}$,$\displaystyle {10^{-2}}$,$\displaystyle {10^{-1}}$,$\displaystyle {10^{0}}$,$\displaystyle {10^{1}}$,$\displaystyle {10^{2}}$}
]
\addplot [thick, color0, opacity=0.55, mark=square*, mark size=4, mark options={solid}]
table {%
4096 0.346668773097066
8192 0.152651267547512
16384 0.108790322810951
32768 0.067972159231777
65536 0.0456764539036208
131072 0.0308306216825065
};
\addplot [thick, black, opacity=0.55, dash pattern=on 1pt off 3pt on 3pt off 3pt, mark=asterisk, mark size=4, mark options={solid}]
table {%
4096 0.132562883814687
8192 0.120854604384428
16384 0.0789044080688301
32768 0.0693797964107682
65536 0.0541567013127864
131072 0.0381964990648716
};
\addplot [thick, green!50.0!black, opacity=0.55, dotted, mark=diamond*, mark size=4, mark options={solid}]
table {%
4096 2.25837195303641
8192 1.81987546352944
16384 2.01765521869897
32768 1.98934552334916
65536 1.94376583811798
131072 2.25835498550087
};
\addplot [thick, blue, opacity=0.55, mark=triangle*, mark size=4, mark options={solid}]
table {%
4096 1.08926098628803
8192 1.31273111151647
16384 1.31709900623468
32768 1.38625624083514
65536 1.58182117201976
131072 1.80403022961901
};
\addplot [thick, red, opacity=0.55, dashed, mark=*, mark size=4, mark options={solid}]
table {%
4096 0.901471893371002
8192 0.801879782795581
16384 0.739344878816882
32768 0.627233917643399
65536 0.518558635267057
131072 0.477324931362032
};
\end{axis}

\end{tikzpicture}}
        \caption*{(c) $d = 3, \alpha = \frac{1}{2d}$ }
        \label{fig:prob1_al_3}
    \end{minipage}%
        \begin{minipage}{0.55\textwidth}
		\scalebox{0.85}{
\begin{tikzpicture}

\definecolor{color0}{rgb}{0,0.75,0.75}

\begin{axis}[
legend style={cells={align=left}},
log basis x={10},
log basis y={10},
tick align=outside,
tick pos=left,
x grid style={white!69.01960784313725!black},
xlabel={n},
xmin=3444.31171687921, xmax=155871.754977636,
xmode=log,
xtick style={color=black},
xtick={100,1000,10000,100000,1000000,10000000},
xticklabels={$\displaystyle {10^{2}}$,$\displaystyle {10^{3}}$,$\displaystyle {10^{4}}$,$\displaystyle {10^{5}}$,$\displaystyle {10^{6}}$,$\displaystyle {10^{7}}$},
y grid style={white!69.01960784313725!black},
ylabel={Average error},
ymin=13.7570115349747, ymax=28092.2776752561,
ymode=log,
ytick style={color=black},
ytick={1,10,100,1000,10000,100000,1000000},
yticklabels={$\displaystyle {10^{0}}$,$\displaystyle {10^{1}}$,$\displaystyle {10^{2}}$,$\displaystyle {10^{3}}$,$\displaystyle {10^{4}}$,$\displaystyle {10^{5}}$,$\displaystyle {10^{6}}$}
]
\addplot [thick, color0, opacity=0.55, mark=square*, mark size=4, mark options={solid}]
table {%
4096 1183.54969103913
8192 1974.23303528112
16384 3774.65653653201
32768 6415.36999594318
65536 10975.4356324738
131072 19866.8745848152
};
\addplot [thick, black, opacity=0.55, dash pattern=on 1pt off 3pt on 3pt off 3pt, mark=asterisk, mark size=4, mark options={solid}]
table {%
4096 24.8941353451604
8192 37.3872357326194
16384 118.891047728785
32768 118.010851890456
65536 2083.21878260737
131072 9729.21184191984
};
\addplot [thick, green!50.0!black, opacity=0.55, dotted, mark=diamond*, mark size=4, mark options={solid}]
table {%
4096 71.0934956134424
8192 89.7211671330872
16384 82.6687415041501
32768 68.5954934251291
65536 84.0160058731876
131072 67.4209326120722
};
\addplot [thick, blue, opacity=0.55, mark=triangle*, mark size=4, mark options={solid}]
table {%
4096 24.9472722719392
8192 20.8672003036699
16384 26.7020435457721
32768 22.3310707145506
65536 20.0447548147601
131072 19.452772320695
};
\addplot [thick, red, opacity=0.55, dashed, mark=*, mark size=4, mark options={solid}]
table {%
4096 24.9472722719392
8192 20.8672003036699
16384 26.7020435457721
32768 22.3310707145506
65536 20.0447548147601
131072 19.452772320695
};
\end{axis}

\end{tikzpicture}}
        \caption*{(d) $d = 3, \alpha = 1.3$}
        \label{fig:prob1_al_4}
    \end{minipage}
    \caption{Plots comparing average error of various estimators on Example~\ref{exm:alpha-mix} for different values of $\alpha$. As suggested by Proposition~\ref{prop:exmMedian}, the median and mean have superior performance to the modal interval and shorth estimators for $\alpha < 1$. Moreover, the hybrid estimator exhibits similar behavior to the median when $\alpha < 1$ and to the shorth when $\alpha > 1$.     }
\label{fig:alpha-mix}
\end{figure}

\subsection{Multivariate}

We now present simulation results for multivariate data, using $d = 3$. The data for all three recurring examples are generated with the same parameters as in the univariate case, except with isotropic distributions. We run the computationally efficient versions of the shorth and modal interval estimators described in Section~\ref{SecComputation}, with $k = 5d \log n$ and $r = \sqrt{d}$.

The trends for i.i.d.\ data, shown in Figure~\ref{fig:iid}(b), are analogous to the univariate case.  Similarly, the plots in Figure~\ref{fig:quadratic}(b) for the quadratic variance example resemble the plots in Figure~\ref{fig:quadratic}(a), with the hybrid, shorth, and modal interval estimators performing noticeably better than the mean or median. Note that for these experiments, the modal interval estimator appears to behave better than either the shorth or hybrid estimators by a constant factor. For the multivariate version of the $\alpha$-mixture distribution, we run simulations with $\alpha = \frac{1}{2d} < 1$ and $\alpha = 1.3$, where we have chosen the first value of $\alpha$ so that the upper bound in Theorem~\ref{ThmUpperBound} gives $\O\left(n^{\alpha - \frac{1}{2}}\right) = \O(n^{\frac{1}{2d} - \frac{1}{2}})$ error for the hybrid estimator, whereas the derived bounds for the modal interval and shorth are $\O(n^{\alpha}) = \O(n^{\frac{1}{2d}})$ (cf. Remark~\ref{RemHybridD}). Indeed, we see in Figure~\ref{fig:alpha-mix}(c) that the estimation error of the hybrid estimator decreases with $n$, like the mean and median estimators, whereas the shorth estimator has an increasing trend line. The curve for the modal interval estimator appears to be roughly constant (or possibly slightly increasing). The curves in Figure~\ref{fig:alpha-mix}(d) are very similar to the curves in Figure~\ref{fig:alpha-mix}(b), suggesting the existence of a phase transition for $\alpha \in \left(\frac{1}{2d}, 1\right]$ in the multivariate case, as well.

\section{Conclusion}

We have studied the problem of mean estimation of a heterogeneous mixture when the fraction of clean points tends to $0$. 
We have shown that the modal interval and shorth estimator, which perform suboptimally in i.i.d.\ settings, are superior to the sample mean in such settings.
We have also shown that these estimators and the $k$-median have complementary strengths that may be combined into a single hybrid estimator, which adapts to the given problem and is nearly optimal in certain settings. An important question for further study is whether the proposed hybrid estimator is always near-optimal, or optimal, for more general collections of variances.

Our discussion of linear regression estimators has been fairly brief. Some issues that we have not addressed include derivations for non-Gaussian error distributions and regression estimators in the case of a fixed design matrix. We leave these questions, and a derivation of optimal error rates in the linear regression setting, for future work.

\section*{Acknowledgments}

AP and PL were partially supported by NSF grant DMS-1749857, and VJ acknowledges partial support from the NSF grant CCF-1841190.  PL thanks Gabor Lugosi for introducing her to the entangled mean estimation problem at the 2017 probability and combinatorics workshop in Barbados.

\bibliography{ref}

\newpage

\appendix

\section{Properties of symmetric distributions}

In this Appendix, we derive the lemmas concerning properties of symmetric distributions (when $d = 1$) and radially symmetric distributions (when $d > 1$). We also discuss the behavior of quantities related to the shapes of the distributions in the running examples described in the paper.

\subsection{Proof of Lemma~\ref{lemma:properties}}
\label{app:lemmaProp}

The proofs proceed using simple calculus and algebraic manipulations, relying only on the properties of symmetry and unimodality.

\begin{enumerate}[label=(\roman*)]
  \item Property~\ref{prop:1} follows directly by unimodality and symmetry of $\overline{P}$.
  \item Property~\ref{prop:2} is true by the non-negativity of density.

\item Let $p(x)$ be the density of $\overline{P}$. Then $R^*_x = 2\int_0^x p(y)dy$. Define $g(x) \coloneqq \frac{R^*_x}{x}$ for $x>0$. Property~\ref{prop:3} is equivalent to showing that $\frac{d}{dx} g(x) < 0$. By unimodality of $p(\cdot)$, we have $g(x) > 2p(x)$ for $x>0$. By differentiation, we have
\begin{align*}
\frac{d}{dx} g(x) = \frac{2xp(x) - 2\int_0^xp(y) dy}{x^2} = \frac{2p(x) - g(x)}{x} < 0,
\end{align*}
as wanted.
\item Note that $r'$ can be written as $r' = (K+\alpha)r$, where $K \in \mathbb{N}$ and $\alpha\in[0,1)$. As $r' >r$, $K\geq 1$. We need to show that $R^*_{r'} > (K + \alpha) R(f_{r',r})$. We may write
\begin{align*}
R^*_{r'} &= 2\int_{0}^{r'}p(x)dx \\
& = 2\int_{0}^{\alpha r} p(x)dx + \sum_{k=1}^{K}2\int_{r'-kr}^{r'-(k-1)r}p(x)dx
\end{align*}
By~\ref{prop:3} above, we have $R^*_{\alpha r} > \alpha R^*_{r}$. Therefore,
\begin{align*}
R^*_{r'} &> 2\alpha \int_{0}^{r} p(x)dx + \sum_{k=1}^{K}2\int_{r'-kr}^{r'-(k-1)r}p(x)dx \\
&> \alpha \int_{r'-r}^{r'+r} p(x)dx + \sum_{k=1}^{K}\int_{r'-r}^{r'+ r}p(x)dx \\
&= (\alpha + K)R(f_{r',r}),
\end{align*}
where the last inequality again uses unimodality of $\overline{P}$.

\item 
Note that
\begin{equation*}
R^*_{q_{(2k)}} = \frac{1}{n} \sum_{i=1}^n \P(|X_i| \le q_{(2k)}) > \frac{1}{2} \cdot \frac{2k}{n} = \frac{k}{n}.
\end{equation*}
Let $\tilde{R}_i(f)$ be the expectation of $f$ under $P_i$, i.e., $\tilde{R}_i(f) = \E f(X_i)$.
For the second inequality, note that by Chebyshev's inequality,
\begin{align*}
\tilde{R}_i (f_{0,2\sigma_i}) = \P(|X_i - \mu| \leq 2 \sigma_i)  \geq \frac{3}{4},
\end{align*}
for all $i$. Therefore, an interval of length $4\sigma_{(2k)}$ covers at least $\frac{3}{4}$ mass of at least $2k$ distributions, implying that
\begin{align*}
R^*_{2 \sigma_{(2k)}} = R(f_{0, 2 \sigma_{(2k)}}) &= \frac{1}{n} \sum_{i=1}^n \tilde{R}_i(f)  
          \geq \frac{1}{n} \cdot \frac{3 \times 2k}{4}   > \frac{k}{n}. 
\end{align*}
\end{enumerate}

\subsection{Proof of Lemma~\ref{lemma:prop2}}
\label{AppLemProp2}

\begin{enumerate}
  \item Note that $R(f_{x,r})$ can be written as convolution of $\overline{P}$ with indicator function of $B_r$, both of which are unimodal and radially symmetric. The desired result then follows by Proposition 8 in Li et al.~\cite{LiEtal19}, which implies that $R(f_{x,r})$ is also unimodal and radially symmetric.  
  \item This follows from the nonnegativity of the density.
  \item 
  As $\overline{P}$ is radially symmetric, let the density of $\overline{P}$ at $x$ be given by $p(\|x\|)$. $R^*_r$ can be written as $R^*_r = C\int_0^rp(s)s^{d-1}ds$ where $C$ is a constant for a fixed dimension. Define $g(r) \coloneqq \frac{R^*_r}{Cr^d} = \frac{\int_0^r p(s)s^{d-1}ds}{r^d}$ for $r > 0$. Property~\ref{prop2:3} is equivalent to showing that $ \frac{d}{dr}g(r) < 0$. By unimodality of $p(\cdot)$, it follows that $g(r) > \frac{p(r)}{d}$. Differentiating $g(\cdot)$, we get
	  \begin{align*}
	  \frac{d}{dr}g(r) = \frac{p(r)r^{d-1}r^{d} - dr^{d-1}\int_{0}^rp(s)s^{d-1}ds }{r^{2d}} = \frac{p(r) - d g(r)}{r} < 0.
	  \end{align*}
	  
  \item Note that any $r_1$-packing of $B(0, r_2-r_1)$ has the property that all balls in the packing must be entirely contained within the larger ball $B_{r_2}$.
   Furthermore, by Lemma~\ref{lemma:prop2}\ref{prop2:1} above, we know that $R(f_{x,r_1}) \ge R(f_{r_2, r_1})$ when $\|x\|_2 \le r_2$.
    Hence, by summing up the densities of all balls in the packing, we obtain
\begin{equation*}
R(f_{0, r_2}) \ge P(B_{r_2 - r_1}, r_1) R(f_{r_2, r_1}),
\end{equation*}
from which the first inequality follows.

To obtain the second inequality, we use the sphere-packing lower bound
\begin{equation*}
P(B_{r_2 - r_1}, r_1) \ge N(B_{r_2 - r_1}, 2r_1) \ge \left(\frac{r_2 - r_1}{2r_1}\right)^d,
\end{equation*}
where $N(\cdot, \cdot)$ denotes the covering number (cf.\ Proposition 4.2.12 of Vershynin~\cite{Ver18}).
  \item The proof of the first inequality is the same as the proof of the corresponding statement in Lemma~\ref{lemma:properties}. The second inequality follows by noting that $\E\|X_i - \mu\|_2^2  = \text{Tr}(\Sigma_i) = d \sigma_i^2$. By Chebyshev's inequality, we have
\begin{align*}
\tilde{R}_i (f_{0,2\sigma_i \sqrt{d}}) = \P(\|X_i - \mu\|_2 \leq 2 \sqrt{d}\sigma_i)  \geq \frac{3}{4},
\end{align*}
for each $i$. Thus, $B_{2 \sigma_{(2k) \sqrt{d}}}$ covers at least $\frac{3}{4}$ of the mass of at least $2k$ distributions, implying the desired result.
 \end{enumerate}

\subsection{Proof of Proposition~\ref{prop:exmple_r_k}}
\label{app:prop_r_k}

\begin{enumerate}
	\item The lower bound follows by noting that the density at $0$ is $\frac{1}{\sqrt{2 \pi}\sigma}$. The upper bound follows by noting that density at $x = |\sigma|$ is within constant factor of the density at $0$.
	\item The lower bound follows by noting that the density at $x = 0$ is
	\begin{equation*}
	\overline{p}(0)= \(\sum_{i=1}^n \frac{1}{\sqrt{2 \pi}c i n}\) = \Theta\(\frac{\log n}{cn}\).
	\end{equation*}
	The upper bound follows by noting that the density at $x = 1$ is
	\begin{equation*}
	\overline{p}(1) =  \(\sum_{i=1}^n \frac{e^{-\frac{1}{i^2c^2}}}{ \sqrt{2 \pi}c i n}\) \geq  \overline{p}(0) - \O\(\frac{1}{n}\).
\end{equation*}
	\item For $\alpha \geq 1$, the upper bound follows from the fact that at least $c\log n$ distributions have small variance $1$. Thus the interval $[-1,1]$ contains more than $0.6$ probability of at least $c \log n$ distributions. The lower bound follows by noting that the density at 0 is
	\begin{equation*}
	\frac{c \log n}{n} \frac{1}{\sqrt{2 \pi}} + \frac{n - c\log n}{n} \frac{1}{n^{\alpha} } = \Theta\(\frac{\log n}{n}\).
	\end{equation*}

	For $\alpha < 1$, the density at $0$ is
	\begin{equation*}
	\frac{c \log n}{n} \frac{1}{\sqrt{2 \pi}} + \frac{n - c\log n}{n} \frac{1}{n^{\alpha} } = \Theta\(\frac{1}{n^\alpha}\).
	\end{equation*}
	The lower bound follows by noting that the density at $x =1 $ is also  $\Theta\(\frac{1}{n^\alpha}\) $.
\end{enumerate}

\section{Proofs of concentration inequalities}

In this appendix, we provide the proofs of the main technical results underlying the success of our estimators.

\subsection{Proof of Lemma~\ref{thm:highProb}} %
\label{app:proof_of_thm_high_prob}

This proof is a special case of the proof of Lemma~\ref{thm:highProbD} in Appendix~\ref{app:proof_of_thm_high_probD}, with $V = 2$.

\subsection{Proof of Lemma~\ref{thm:highProbD}}
\label{app:proof_of_thm_high_probD}

Recall that $\tilde{R}_i(f) = \E f(X_i)$. We define the random variables
\begin{align*}
Y_{f,i} &\coloneqq f(X_i) - \tilde{R}_i(f).
\end{align*}
Note that $\E_i[Y_{f,i}] =0$ and $|Y_{f,i}| \leq 1$. Furthermore, the variables $(Y_{f,i})_{i=1}^n$  are independent for each fixed $f$.
Let 
\begin{align*}
Z &\coloneqq \sup_{f \in \cH_{r}} \left(R_n(f) - R(f)\right) =  \sup_{f \in \cH_{r}} \frac{1}{n}\sum_{i=1}^nY_{f,i}.
\end{align*}

We will apply Lemma~\ref{thm:Bou12.9} to obtain a high-probability upper bound on $Z$. Here $V = d+1$, the VC dimension of balls.

Since its application requires a bound on the expectation, we first derive the following lemma:

\begin{lemma}
If $nR^*_{r} \geq 1300 V \log n$ with both $ n > 1$ and $d \geq 1$, then 
\begin{align*}
\E Z \leq 72 \sqrt{V\frac{R^*_{r} \log n}{2n}}.
\end{align*}
\label{lemma:ConvergenceInExpectation}
\end{lemma}

\begin{proof}
We will use Theorem~\ref{thm:vcConvExpTech} from Appendix~\ref{AppAux},
with $\sigma^2 = \sup_{x, r' \leq r}R(f_{x,r'}) = R^*_{r}$.
In particular, note that since $n\sigma^2 \geq 1300 V \log n $, we have
\begin{align*}
\log\( \frac{4e^2}{\sigma}\) & = \frac{1}{2}\log\( \frac{16e^4}{\sigma^2}\) \leq  \frac{1}{2}\log\( \frac{16e^4n}{1300 V \log n}\) \\
& \leq \frac{\log n}{2},
\end{align*}
so
\begin{align*}
 \(24 \sqrt{ \frac{V}{5n} \log \(\frac{4e^2}{\sigma}\)}\)^2 &= \frac{576V}{5n}\log\(\frac{4e^2}{\sigma}\)\\
  &\leq \frac{576V}{5n} \cdot \frac{\log n}{2} = 57.6 V \frac{\log n}{n} \le \sigma^2.
\end{align*}
Thus, Theorem~\ref{thm:vcConvExpTech} is applicable and leads to the following bound:\footnote{Note that the definition of $Z$ in Theorem~\ref{thm:vcConvExpTech} has a factor of $1/\sqrt n$ as opposed to the factor of $1/n$ here.}
\begin{align*}
\E Z \leq 72 \frac{\sqrt{R^*_r}}{\sqrt{n}} \sqrt{V \log\(\frac{4e^2}{\sigma}\)} \leq 72 \sqrt{ \frac{VR^*_r \log n}{2n} }.  
\end{align*}
\end{proof}

We now apply Theorem 12.9 from Boucheron et al.~\cite{BouEtAl16} (stated in Lemma~\ref{BouEtal12.9} in Appendix~\ref{AppAux}) with $W_{i,s} = Y_{i,f}$ and
\begin{align*}
\rho^2 &= \sup_{f \in \cH_r} \sum_{i=1}^n \E Y_{i,f}^2 
                        = \sup_{f \in \cH_r} \sum_{i=1}^n \Var[ f(X_i)] \\
                        &\leq \sup_{f \in \cH_r} \sum_{i=1}^n \E [f(X_i)]
                        = \sup_{f \in \cH_r} nR(f) 
                        = nR^*_r,
\end{align*}
where the inequality holds because the variance of a Bernoulli random variable is bounded by its expectation. Hence, using Lemma~\ref{lemma:ConvergenceInExpectation} and the assumption $nR^*_r \geq 1300 V\log n$, we have
\begin{align*}
v & = 2n \E Z + \rho^2 \leq 2n \E Z + nR^*_r 
\\ &\leq 144\sqrt{ 0.5V  n R^*_r \log n} + nR^*_r 
    \leq nR^*_r\(144 \sqrt{\frac{0.5V\log n}{nR^*_r}} + 1\) \\
    &\leq nR^*_r\(144 \sqrt{\frac{0.5V\log n}{1300V \log n} } + 1\) < 6nR^*_r.
\end{align*}
Thus, $\frac{ntR^*_r}{2v} > \frac{t}{12}$, so 
\begin{align}
\log\(1 + 2\log \(1 + \frac{ntR^*_r}{2v}\)\) & \ge \log \(1 + 2 \log\left(1 + \frac{t}{12}\right)\) \ge \frac{t}{50},
\label{eq:logApprox}
\end{align}
using the fact that $t \le 1$.

Now suppose $nR^*_r \geq C_t \frac{V}{2} \log n$ for the constant $C_t = \(\frac{144}{t}\)^2 $. Note that for $t \leq 1$, we have $nR^*_r \geq 1300V\log n$, so all the previous results are also valid. Moreover, we have
\begin{align*}
\frac{\E Z}{0.5tR^*_r} &= \frac{n\E Z}{0.5tnR^*_r} \leq \frac{72 \sqrt{0.5VnR^*_r\log n} }{0.5tnR^*_r} = \frac{144 \sqrt{0.5V\log n} }{t \sqrt{nR^*_r}}  \\&\leq  \frac{144 \sqrt{0.5V\log n} }{t \sqrt{0.5C_tV\log n}} =  \frac{144 }{t \sqrt{C_t}} < 1.
\end{align*}

Now we have all the ingredients required for the application of Theorem 12.9
:\begin{align*}
\P\{ Z & \geq tR^*_r \} 
\leq    
 \P\{Z \geq \E Z + 0.5tR^*_r \} \\
 &\leq \exp\(-\frac{ntR^*_r}{4}\log\(1 + 2\log\(1 + \frac{ntR^*_r}{2v}\)\)\) \\
 &\leq \exp\(-\frac{1}{200}nt^2R^*_r\),
\end{align*}
where the last inequality follows by inequality~\eqref{eq:logApprox}.

An identical argument can be used to upper-bound the quantity
\begin{equation*}
\sup_{f \in \cH_{r}} \left(R(f) - R_n(f)\right),
\end{equation*}
concluding the proof.

\subsection{Proof of Lemma~\ref{ThmUniformProb}}
\label{AppThmUniformProb}

We begin by proving inequality~\eqref{EqnUniformProb1}. First consider the following peeling lemma, an adaptation of Lemma 3 in Raskutti et al.~\cite{RasEtal10}:
\begin{lemma}
\label{LemPeel}
Let $A \subseteq \real^p$, and suppose $\{Y_x\}_{x \in A}$ is a collection of random variables indexed by $x$. Also suppose $g: \real \rightarrow \real_+$ is a strictly increasing
function such that $\inf_{x \in A} g(h(\|x\|_2)) \ge \mu$, for some $\mu > 0$, and $h: \real_+ \rightarrow \real_+$ is a constraint function, and the tail bound
\begin{equation*}
\mprob\left(\sup_{x \in A: h(\|x\|_2) \le s} Y_x \ge g(s)\right) \le 2 \exp\big(-c g(s)\big)
\end{equation*}
holds for all $s \in \mathrm{range}(h)$. Then
\begin{equation}
\label{EqnUniform}
\mprob\Big(Y_x \le 2g(h(\|x\|_2)), \quad \forall x \in A\Big) \ge 1 - \frac{2\exp(-c\mu)}{1-\exp(-c\mu)}.
\end{equation}
\end{lemma}

\begin{proof}
We define the sets
\begin{equation*}
A_m := \left\{x \in A: 2^{m-1} \mu \le g(h(\|x\|_2)) \le 2^m \mu \right\},
\end{equation*}
for $m \ge 1$. By a union bound, we have
\begin{equation*}
\mprob\Big(\exists x \in A \text{ s.t. } Y_x > 2g(h(\|x\|_2)) \Big) \le \sum_{m=1}^M \mprob\Big(\exists x \in A_m \text{ s.t. } Y_x > 2g(h(\|x\|_2)) \Big),
\end{equation*}
where $M = \sup_{m \ge 1} g^{-1}(2^{m-1} \mu) \in \mathrm{range}(h)$.

Further note that if $x \in A_m$ satisfies $Y_x > 2g(h(\|x\|_2))$, then $g(h(\|x\|_2)) \ge 2^{m-1} \mu$, so
\begin{align*}
\mprob\left(\sup_{x \in A_m} Y_x > 2g(h(\|x\|_2))\right) & \le \mprob\left(\sup_{x \in A_m} Y_x > 2 \cdot 2^{m-1} \mu\right) \\
& \le \mprob\left(\sup_{x \in A: g(h(\|x\|_2)) \le 2^m \mu} Y_x > 2^{m} \mu \right) \\
& = \mprob\left(\sup_{x \in A: h( \|x\|_2) \le g^{-1} (2^m \mu)} Y_x > 2^m \mu \right) \\
& \le 2\exp\left(-c \cdot 2^m \mu\right),
\end{align*}
if $m < M$. If $m = M$, the same logic shows that
\begin{equation*}
\mprob\left(\sup_{x \in A_m} Y_x > 2g(h(\|x\|_2))\right) \le \mprob\left(\sup_{x \in A: h(\|x\|_2) \le \nu} Y_x > 2^m\mu \right),
\end{equation*}
where $\nu = \sup_{x \in A} h(\|x\|_2)$. Furthermore, $2^{m-1} \mu \le g(\nu) \le 2^m \mu$, so the last probability is upper-bounded by
\begin{equation*}
\mprob\left(\sup_{x \in A: h(\|x\|_2) \le \nu} Y_x \ge g(\nu) \right) \le 2\exp(-c g(\nu)) \le 2\exp(-c \cdot 2^{m-1} \mu).
\end{equation*}
It follows that
\begin{equation*}
\mprob\left(\sup_{x \in A_m} Y_x > 2g(h(\|x\|_2))\right) \le 2\exp(-c\cdot 2^{m-1} \mu),
\end{equation*}
for all $m \ge 1$, so summing up over $m$ then gives
\begin{align*}
\mprob\Big(\exists x \in A \text{ s.t. } Y_x > 2g(h(\|x\|_2)) \Big) & \le \sum_{m=1}^\infty 2\exp(-c \cdot 2^{m-1} \mu) \le \frac{2\exp(-c\mu)}{1-\exp(-c\mu)},
\end{align*}
implying inequality~\eqref{EqnUniform}.
\end{proof}

We apply Lemma~\ref{LemPeel} with $A = \{x: \|x\|_2 \le \bar{r}\}$, and
\begin{equation*}
Y_x = |R_n(f_{x,r}) - R(f_{x,r})|, \qquad h(\|x\|_2) = R(f_{x,r}), \qquad g(s) = ts,
\end{equation*}
for fixed values of $\bar{r}, r > 0$ and $t \in (0,1]$. Clearly, $g$ is monotonically increasing and satisfies $\inf_{x \in A} g(h(\|x\|_2)) \ge tR(f_{\bar{r}, r})$. Note that for any $s \in \mathrm{range}(h)$, we have $s = R(f_{x_s, r})$ for some $x_s$, and
\begin{align*}
\mprob\left(\sup_{x \in A: h(\|x\|_2) \le s} |R_n(f_{x,r}) - R(f_{x,r})| \ge g(s)\right) & = \mprob\left(\sup_{\|x_s\|_2 \le \|x\|_2 \le \bar{r}} |R_n(f_{x,r}) - R(f_{x,r})| \ge tR(f_{x_s, r})\right) \\
& \le 2\exp(-cn R(f_{x_s, r}) t^2) \\
& = 2\exp(-cnt g(s)),
\end{align*}
assuming $nR(f_{\bar{r}, r}) \geq C_t d \log n$, where we use a slight modification of Lemma~\ref{thm:highProbD} where $\cH_r$ is the set of balls centered around points in $\{\|x\|_2 \ge \|x_s\|_2\}$. Lemma~\ref{LemPeel} then implies the desired concentration inequality.

To establish inequality~\eqref{EqnLargeX}, note that we can simply use a modification of Theorem~\ref{thm:highProbD}, where $\cH_r$ is now the set of balls centered around points in $\{\|x\|_2 > \bar{r}\}$.

\section{Modal interval estimator}

In this appendix, we provide proofs of the various theorems and lemmas related to the modal interval estimator.

\subsection{Proof of Lemma~\ref{thm:modalIntervalMain}}
\label{app:proof_of_theorem_modal}

This will follow from Lemma~\ref{thm:highProb} by choosing $0.5t$ instead of $t$. 
If $R^*_r \geq C_{0.5 t}\frac{\log n}{n}$, then with probability $1 - 2\exp(-cnR^*_rt^2/4)$, we have
\begin{align*}
 |R_n(f) - R(f)| \le \frac{tR^*_r}{2},
 \end{align*}
uniformly over $f \in \cH_r$. Assume that this event happens. Note that $R(f_{0,r})=R^*_r$ and $R_n(f_{\muhat_{M,r},r}) \geq R_n(f_{0,r})$ by maximality of the modal interval estimator. Since $f_{\muhat_{M,r},r}, f_{0,r} \in \cH_r$, we have
  \begin{align*}
  R(f_{\muhat_{M,r},r}) &\geq  R_n(f_{\muhat_{M,r},r}) - \frac{tR^*_r}{2} \geq R_n(f_{0,r}) - \frac{tR^*_r}{2} \\
              &\geq R(f_{0,r}) - tR^*_r = R^*_r - tR^*_r,
\end{align*}
as wanted.

\subsection{Proof of Theorem~\ref{cor:modal}}

By Lemma~\ref{lemma:properties}\ref{prop:1}, we know that if $ R(f_{r', r}) < R(f_{\muhat_{M,r}, r})$, then $|\widehat{\mu}_{M,r}| \leq r'$.
Furthermore, taking $t = \frac{1}{2}$ in Lemma~\ref{thm:modalIntervalMain}, we have $R(f_{\muhat_{M,r}, r}) \ge \frac{R^*_r}{2}$, with probability at least $1-2\exp(-c'nR^*_r/4)$. Thus, inequality~\eqref{EqnMuhatBd1} holds provided $R(f_{r',r}) < \frac{R^*_r}{2}$.

Now suppose Let $r' = \frac{2r}{R^*_r}$. By Lemma~\ref{lemma:properties}\ref{prop:4} and noting that $R^*_{r'} \leq 1$, we have
\begin{align*}
R(f_{r',r}) &< \frac{r}{r'} R^*_{r'}
      \leq \frac{r}{r'} = \frac{r}{\frac{2r}{R^*_r}} =  \frac{R^*_r}{2}.
\end{align*}
This establishes inequality~\eqref{EqnMuhatBd}.

\subsection{Proof of Proposition~\ref{prop:exmModal}}
\label{app:exmPropModal}

Since $r = r_{C\log n}$, we have $R^*_{r} =  \frac{C \log n}{n}$.
By inequality~\eqref{EqnMuhatBd} of Theorem~\ref{cor:modal}, we have 
\begin{align}
|\muhat_{M,r}| \leq \frac{2nr_{C \log n}}{C \log n},
\label{eq:modalCorollary}
\end{align}
w.h.p.
\begin{enumerate}
	\item Analogously to Proposition~\ref{prop:exmple_r_k}, we have $r_{C \log n} = \Theta\( \frac{C \sigma\log n}{n}\)$. Inequality~\eqref{eq:modalCorollary} then gives the result.
	\item The bound of $\tilde\O(n)$ follows by inequality~\eqref{eq:modalCorollary} and noting that $r_{C \log n} = \O(1)$ for a fixed $C$ and sufficiently small $c>0$.
	We now focus on how to obtain the tighter bound of $\O(n^{\epsilon})$ for an $\epsilon > 0$, using inequality~\eqref{EqnMuhatBd1}.

	Let $\tilde{R}_i(f)$ be the expectation of $f$ under $P_i$, i.e., $\tilde{R}_i(f) = \E f(X_i)$.
	Fix an $\epsilon > 0$. Let $r' = n^\epsilon$ and $r=1$. Then it suffices to show that $R^*_r - R(f_{r',r}) \geq C'R^*_r$ where $C' > 0$ might depend on $\epsilon$ but not on $n$. 

	We will show that 
	\begin{enumerate}
		\item $R^*_r - R(f_{r',r}) \geq c_1 \sum_{i \leq \frac{r'}{10c}}\tilde{R_i}(f_{0,r})$,
		\item $\sum_{i \leq \frac{r'}{5c}}\tilde{R}_i(f_{0,r}) \geq c_2n R^*_r$.
	\end{enumerate}
	To derive the first inequality, note that 
	\begin{align*}
	nR^*_r - nR(f_{r',r}) &\geq \sum_{i \leq \frac{r'}{10c}} R(f_{0,1}) - R(f_{r',1}) \\
						  &\geq \sum_{i \leq \frac{r'}{10c}} 2\int_{0}^{1} \frac{1}{\sqrt{2 \pi}ci}\(e^{- \frac{x^2}{2c^2i^2}} - e^{- \frac{(0.5r'+ x)^2}{2c^2i^2}}\)dx \\
						  &\geq \sum_{i \leq \frac{r'}{10c}} 2\int_{0}^{1} \frac{(1 - e^{-\frac{0.25r'^2}{2c^2i^2}})}{\sqrt{2 \pi}ci}e^{- \frac{x^2}{2c^2i^2}}dx \\
						  &\geq (1 - e^{-10}) \sum_{i \leq \frac{r'}{10c}}\tilde{R}_i(f_{0,r}).
	\end{align*}
	Now it remains to show that $\sum_{i \leq \frac{r'}{10c}}\tilde{R}_i(f_{0,r}) \geq c_2 R^*_r$. First note that $nR^*_r \leq \frac{\log n}{c}$. Hence,
	\begin{align*}
	\sum_{i \leq \frac{r'}{10c}}\tilde{R}_i(f_{0,1})  &\geq \sum_{i: \frac{1}{c} < i \leq \frac{r'}{10c}}\tilde{R}_i(f_{0,1}) \geq  \sum_{i: \frac{1}{c} < i \leq \frac{r'}{10c}}  \frac{2 e^{-0.5}}{\sqrt{2 \pi} ci} \geq c_3 \log\(\frac{r'}{10e}\) \geq c_4 \log n^ \epsilon  \geq c_5 \epsilon nR^*_r.
	\end{align*}
	
	\item For $\alpha < 1$, let $r' = \Theta\(n^{\alpha}\)$. Then it is easy that $R(f_{r',r}) \leq \frac{R^*_r}{2}$. This follows by observing that the density of a Gaussian distribution decreases by more than half at a distance of $\sigma$ from the mean.

	For $\alpha \geq 1$, let $r' = 10$. Then $R^*_r \geq 0.5  \frac{C \log n}{n}$, as a Gaussian distribution contains about $0.68$ mass within 1 standard deviation of the mean.
	Moreover,
	\begin{equation*}
	R(f_{r',r}) \leq 0.1 \frac{C \log n}{n} + \frac{n}{\sqrt{2 \pi}n^\alpha} \leq 0.2 \frac{C \log n}{n} \leq \frac{R^*_r}{2}.
	\end{equation*}
	Inequality~\eqref{EqnMuhatBd1} then implies the result.
\end{enumerate}

\subsection{Proof of Theorem~\ref{ThmLepski}}
\label{app:prfLepski}
\begin{proof}
Let $j' := \min\{j \in \scriptJ: r_j \ge r^*\}$. Then
\begin{align*}
\P(j_* > j') & = \P\left(\bigcup_{i \in \scriptJ: i > j'} \left\{|\muhat_{M, r_i} - \muhat_{M, r_{j'}}| > \frac{4nr_i}{C_{0.25} \log n}\right\}\right) \\
& \le \P\left(|\muhat_{M, r_{j'}}| > \frac{2nr_{j'}}{C_{0.25} \log n}\right)  \\
 &\qquad + \sum_{i \in \scriptJ: i > j'} \P\left(|\muhat_{M, r_i}| > \frac{2nr_i}{C_{0.25} \log n}\right),
\end{align*}
using a union bound and the triangle inequality. We may use Theorem~\ref{cor:modal} to bound each individual term, so that the probability of the bad event
\begin{align*}
E := \bigcup_{i \in \scriptJ: i > j'} \left\{|\muhat_{M, r_i}| > \frac{2nr_i}{C_{0.25}\log n}\right\} \cup \left\{|\muhat_{M, r_{j'}}| > \frac{2nr_{j'}}{C_{0.25} \log n}\right\}
\end{align*}
is bounded by
\begin{align*}
\P(E) &\le (1+|\scriptJ|) \cdot 2\exp(-cC_{0.25}\log n/16) \\
&\le 2\left(1+\log_2\left(\frac{2r_{\max}}{r_{\min}}\right)\right) \exp(-c C_{0.25}\log n/16).
\end{align*}
Finally, note that on the event $E^c$, we have $j_* \le j'$ (establishing that $j_*$ is finite), so
\begin{equation*}
|\muhat_{M, r_{j_*}} - \muhat_{M, r_{j'}}| \le \frac{4nr_{j'}}{C_{0.25} \log n}.
\end{equation*}
Combined with the inequality $|\muhat_{M, r_{j'}}| < \frac{2nr_{j'}}{C_{0.25} \log n}$, we conclude that
\begin{align*}
|\muhat_{M, r_{j_*}}| & \le \frac{4 nr_{j'}}{C_{0.25} \log n} + \frac{2nr_{j'}}{C_{0.25} \log n} \le \frac{6nr_{j'}}{C_{0.25} \log n} \\
& \le \frac{12r^*}{C_{0.25} \log n},
\end{align*}
using the fact that $r_{j'} < 2 r^*$.
\end{proof}

\subsection{Proof of Lemma~\ref{lemma:shorthlength}} %
\label{app:length_of_shortest_gap}

We first prove the upper bound. Consider the fixed interval $f_{0, r_{2k}}$. Note that $R(f_{0, r_{2k}}) = R^*_{r_{2k}}= \frac{2k}{n}$. It suffices to show that this interval contains at least $k$ points, with high probability.
By the multiplicative form of the Chernoff bound (Lemma~\ref{LemChernoff} in Appendix~\ref{AppAux}), 
\begin{align*}
\P \left( R_n(f_{0, r_{2k}}) \leq \frac{k}{n}\right) & = \P \left( R_n(f_{0, r_{2k}}) \leq \frac{1}{2} R(f_{0,r_{2k}})\right) \\
&\leq \exp\left(- n \cdot \frac{k}{n} \cdot \frac{1}{8}\right) 
          =  \exp(-k/8).
\end{align*}
Therefore, with probability at least $1 - \exp(-k/8)$, an interval of size $2r_{2k}$ contains at least $k$ points, implying that the shortest gap, $\widehat{r}_{k} \le r_{2k}$.

We now turn to verifying the lower bound. We will prove that with high probability, no interval of size $2r_{k/2}$ contains at least $k$ points, so that $\widehat{r}_k > r_{k/2}$. By definition, $nR^*_{r_{k/2}} = \frac{k}{2}$. Thus, assuming $k \geq 2C_{0.5}\log n$, we may apply Lemma~\ref{thm:highProb} to conclude that
\begin{align*}
\sup_{f \in \cH_{r_{k/2}}} R_n(f) - R(f)  &\leq \frac{R^*_{ r_{k/2}}}{2},
\end{align*}
with probability at least
\begin{equation*}
1 - \exp\left(-\frac{c n}{4}R^*_{r_{k/2}} \right) = 1 - \exp(-ck/8).
\end{equation*}
This implies that
\begin{align*}
\sup_{f \in \cH_{r_{k/2}}} R_n(f) &\leq \frac{3}{2} \cdot R^*_{r_{k/2}} 
                        = \frac{3}{2} \cdot \frac{k}{2n} < \frac{k}{n},
\end{align*}
which is exactly what we want.

\subsection{Proof of Proposition~\ref{lem:modal_lower_bound}}
\label{app:modal_lower_bound}

We first provide the main steps of the proof. Proofs of supporting lemmas are contained in further sub-sections.

\subsubsection{Main argument}

Let $A = [-2,2]$. Consider two disjoint set of hypothesis classes $\cK$ and $\cJ$, with $\cK = \{f_{x,1}: x \in A\}$ and $\cJ = \{f_{x,1}: x \not \in A \}$. 
Note that $R^*_1 = \sup_{f \in \cK} R(f) = \Theta\(n^{-\alpha}\) $. Define $R^*_{\cJ} \deff \sup_{f \in \cJ} R(f)$. Note that supremum is achieved in both the cases and $R^*_{\cJ} < R^*_1$. Moreover, we have the following straightforward relations:
\begin{enumerate}
	\item $  2R^*_1 \geq \P(A) \geq R^*_1$.
	\item $nR^*_{\cJ} =  \Theta\(n^{1 - \alpha}\)$.

	\item $\P(A) \sqrt{nR^*_\cJ} = \O(1)$.

	\item For every constant $C'$, there exists another constant $C > 0$ such that
	\begin{align*}
		 R^*_{\cJ} + C\(\sqrt{\frac{R^*_{\cJ}}{n}}	\) \geq R^*_1 + C'\(\sqrt{\frac{R^*_{1}}{n}}\).
	\end{align*}
\end{enumerate}
Define the following random variables:
\begin{align*}
Z_1 = \sup_{f \in \cK} R_n(f) , \qquad
Z_2 = \sup_{f \in \cJ} R_n(f) 
\end{align*}
These relations suffice for showing that $Z_1 < Z_2$ with constant probability. To this end, we would show that with constant probability both (1) $Z_1 = R^*_1 + \O\(\sqrt{\frac{R^*_{1}}{n}}\)$, and (2) $Z_2 \geq R^*_{\cJ} + C\(\sqrt{\frac{R^*_{\cJ}}{n}}	\)$, for any $C > 0$. 
Note that these events are dependent and thus we'd use the following lemma, which shows that conditioned on the inclusion of points in each of two disjoint intervals, the distributions of the histograms on each of the intervals behave independently:

\begin{lemma}
\label{LemIndep}
Let $\{x_1, \dots, x_n\}$ be i.i.d.\ draws from a distribution with density $p_i$. Consider two disjoint intervals $A$ and $B$. For any two disjoint subsets $S, T \subseteq \{1, \dots, n\}$, we use $x_S$ to denote the vector $(x_i: i \in S)$, and we define $x_T$ similarly. Let $E$ denote the event that $x_i \in A$ for all $i \in S$, and $x_i \in B$ for all $ \in T$. Then for $x_S \subseteq A$ and $x_T \subseteq B$, we have
\begin{equation*}
p_{S,T}(x_S, x_T \mid E) = p_S(x_S \mid E) p_T(x_T \mid E).
\end{equation*}
Furthermore,
\begin{align*}
p_S(x_S \mid E) & = \prod_{i \in S} \frac{p_i(x_i)}{\mprob(X_i \in A)}, \quad \text{and} \\
p_T(x_T \mid E) & = \prod_{i \in T} \frac{p_i(x_i)}{\mprob(X_i \in B)}
\end{align*}
are the joint densities of independent draws from the renormalized distributions of the points lying in each interval.
\end{lemma}

Let $S \subset \{1, \hdots, n\}$ be an index set.
For a fixed index set $S$, let the event $E_S$ be $E_S = \{X_S \subset A, X_{S^c} \subset A^c\}$, 
where $X_S$ is the vector $(X_i : i \in S )$ and $A$ is defined above.

Conditioned on $E_S$, Lemma~\ref{LemIndep} states that $X_i$'s are independent. 
Thus conditioned on $E_S$, the random variables $Z_1$ and $Z_2$ are independent.

\begin{lemma}
Consider the setting in Proposition~\ref{lem:modal_lower_bound}. Let $S \subset [n]$ be such that $|S| \leq  n\P(A)$.
Then for some $C' > 0$,
\begin{align*}
 Z_1  \leq R^* + C'\sqrt{ \frac{R^*}{n}} 
\end{align*}
with a constant nonzero probability, conditioned on the event $E_S$.
\label{LemLowZ_1}
\end{lemma}

\begin{lemma}
Consider the setting in Proposition~\ref{lem:modal_lower_bound}. Let $S \subset [n]$ be such that $|S^c| \geq  n\P(A^c)$. Conditioned on the event $E_S$, 
we have that for all $C > 0$,
\begin{align*}
 Z_2 \geq R^*_{\cJ_n} + C \(\sqrt{\frac{R^*_{\cJ_n}}{n}}	\)
 \end{align*}
 with a constant, nonzero probability depending on the constant $C$. 
\label{LemLowZ_2}
\end{lemma}

\begin{lemma}
Let $X_1, \hdots, X_n \iidm P$, where $P$ is a uniform distribution over $[-b,-a] \bigcup [a,b]$ for some $0 \leq a < b$.
Let $Z = \sup_{f \in \cH_r} R_n(f)$ and $k \in \mathbb{N}$ such that $E = \{Z=k\}$ is an event of nonzero probability. If $ \frac{b-a}{r} > C $, then
\begin{enumerate}
	\item $\P(|\muhat_{M,r}| \geq \frac{b-a}{2}) \geq c > 0$.
	\item $\P(|\muhat_{M,r}| \geq \frac{b-a}{2} | Z \geq k) \geq c > 0$.
\end{enumerate}
\label{LemUnifSym}
\end{lemma}

Lemmas~\ref{LemLowZ_1},~\ref{LemLowZ_2}, and~\ref{LemUnifSym} give us the required lower bound on the probability of error. Let $\muhat_{M,1,\cJ} := \arg\max_{f \in \cJ}R_n(f)$. Clearly, we can write
\begin{align*}
& \P\left\{ |\muhat_{M,r}| \geq \frac{n^{\alpha}}{2}\right\} = \P\left\{Z_1 < Z_2 , |\muhat_{M,1,\cJ}| \geq \frac{n^{\alpha}}{2}\right\}  \\
 & \qquad = \sum_{S \subset [n]} \P(E_S) \P\(Z_1 \leq Z_2, |\muhat_{M,1,\cJ}| \geq \frac{n^{\alpha}}{2} \bigg| E_S\)  \\
& \qquad \geq \sum_{S \subset [n]: |S| \leq n\P(A)} \P(E_S)\P\(Z_1 \leq  n R^* + C\sqrt{nR^*}, Z_2 \geq  n R^* + C\sqrt{nR^*} , |\muhat_{M,1,\cJ}| \geq \frac{n^\alpha}{2} \bigg|  E_S\).
\end{align*}
Furthermore, note that since $Z_1$ is computed over the points lying in $A$ and $Z_2$ and $\muhat_{M,1,\cJ}$ is computed over the points lying in $A^c$, Lemma~\ref{LemIndep} implies that
\begin{align*}
& \P\(Z_1 \leq  n R^* + C\sqrt{nR^*}, Z_2 \geq  n R^* + C\sqrt{nR^*} , |\muhat_{M,1,\cJ}| \geq \frac{n^\alpha}{2} \bigg|  E_S\) \\
& \qquad = \mprob\left(Z_1 \leq  n R^* + C\sqrt{nR^*} \bigg| E_S\right) \mprob\left(Z_2 \geq  n R^* + C\sqrt{nR^*} , |\muhat_{M,1,\cJ}| \geq \frac{n^\alpha}{2} \bigg|  E_S\right) \\
& = \mprob\left(Z_1 \leq  n R^* + C\sqrt{nR^*} \bigg| E_S\right) \mprob\left(Z_2 \geq  n R^* + C\sqrt{nR^*} \bigg| E_S\right) \\
& \qquad \qquad \cdot \mprob\left(|\muhat_{M,1,\cJ}| \geq \frac{n^\alpha}{2} \bigg|  Z_2 \geq  n R^* + C\sqrt{nR^*}, E_S\right).
\end{align*}
Finally, note that conditioned on $E_S$, the points in $A^c$ are certainly still uniformly distributed by the construction. Hence, we can apply Lemmas~\ref{LemLowZ_1},~\ref{LemLowZ_2}, and~\ref{LemUnifSym} to lower-bound each of the three factors by a constant. We conclude that
\begin{align*}
 \P\left\{ |\muhat_{M,r}| \geq \frac{n^{\alpha}}{2}\right\} 
&\geq \sum_{S \subset [n]: |S| \leq n\P(A)} \P(E_S) \Theta(1)  = \Theta(1),
\end{align*}
where the final equality uses the fact that for $X \sim \text{Bin}(n,p)$, we have $\P(X \leq \E X) = \Theta(1)$.

\subsubsection{Proof of Lemma~\ref{LemIndep}}

Clearly, we have
\begin{equation*}
p_{S,T}(x_S, x_T \mid E) = \frac{p_{S,T}(x_S, x_T)}{\mprob(E)} = \frac{\prod_{i \in S} p_i(x_i) \prod_{j \in T} p_i(x_j)}{\mprob(E)}.
\end{equation*}
Similarly, we may write
\begin{align*}
p_S(x_S \mid E) & = \frac{p_i(x_S) \prod_{j \in T} \mprob(X_j \in B)}{\mprob(E)}, \\
p_T(x_T \mid E) & = \frac{p_i(x_T) \prod_{i \in S} \mprob(X_i \in A)}{\mprob(E)}.
\end{align*}
Using the fact that
\begin{equation*}
\mprob(E) = \prod_{i \in S} \mprob(X_i \in A) \prod_{j \in T} \mprob(X_j \in B)
\end{equation*}
implies the desired statements.

\subsubsection{Proof of Lemma~\ref{LemLowZ_1}}

 Conditioned on $E_S$, Lemma~\ref{LemIndep} states that $X_S$ is a vector of $|S|$ \iid  points with distribution, say, $Q_{n|A}$.
 Under $Q_{n|A}$, $\sup_{f \in \cH_n} R(f) = \frac{R^*_1}{\P(A)} \geq \frac{1}{2}$ .

Using Theorem~\ref{ThmVCAbsDeviation} (Theorem 8.3.23 in Vershynin~\cite{Ver18}), we get that
\begin{align*}
\E\left[\left|\sup_{f \in \cK} \sum_{i \in S} f(X_i) - \E[f(X_i)|E_S]\right|\right]  \leq C\sqrt{|S|}  \leq C \sqrt{2|S| \frac{R^*_1}{\P(A)}}.
\end{align*}
Thus, with constant positive probability,
\begin{align*}
Z_1 = \sup_{f \in \cK} \sum_{i} f(X_i) &\leq |S| \frac{R^*_r}{\P(A)} + C'\sqrt{|S| \frac{R^*_1}{\P(A)}} \\
	&\leq nR^*_r + C' \sqrt{nR^*_1}, 
\end{align*}
where we use Markov's inequality and the assumption that $|S| \leq n\P(A)$.

\subsubsection{Proof of Lemma~\ref{LemLowZ_2}}

Consider a fixed function $f \in \cJ_n$. As the distribution is uniform, $R(f) = R^*_{\cJ}$.
Once we have conditioned on $E_S$, there are $|S^c|$ points distributed over $A^c$ according to Lemma~\ref{LemIndep}, i.e., \iid with a uniform distribution, say, $Q_{n|A^c}$.

For each $i \in S^c$, let  $Y_i = f(X_i) - \frac{R^*_{\cJ}}{\P(A^c)}$. $Y_i$'s are centered \iid Bernoulli random variables. 
 We calculate the following quantities required for the Berry-Esseen Theorem,
\begin{align*}
\E[Y_i]  &= 0 \\
\Var[Y_i] &= \frac{R^*_{\cJ}}{\P(A^c)}\(1 - \frac{R^*_{\cJ}}{\P(A^c)}\)  \geq \frac{R^*_{\cJ}}{2\P(A^c)} \\
\E|Y_i|^3		 &= \frac{R^*_{\cJ}}{\P(A^c)}\left|1 - \frac{R^*_{\cJ}}{\P(A^c)}\right|^3 + \(1 - \frac{R^*_{\cJ}}{\P(A^c)}\) \left|\frac{R^*_{\cJ}}{\P(A^c)}\right|^3\\
				&\leq  \frac{R^*_{\cJ}}{\P(A^c)} + \(\frac{R^*_{\cJ}}{\P(A^c)}\)^3 \leq 2\frac{R^*_{\cJ}}{\P(A^c)} 
\end{align*}
By the Berry-Esseen Theorem~\cite{Ver18}, we have
\begin{align*}
\P\left\{ \frac{\sum_{i \in S^c}  Y_i}{\sqrt{|S^c| \Var[Y_i] }}   \geq t \right\} &\geq \phi(t) - \frac{ \E |Y_i|^3  }{ \sqrt{\Var[Y_i]^3 |S^c|} } 
					\geq \phi(t) - \frac{ \frac{2R^*_{\cJ}}{\P(A^c)} }{ \sqrt{ \frac{\(R^*_{\cJ}\)^3}{8\P(A^c)^3}   n\P(A^c) } } \\
					&\geq \phi(t) - \frac{c' }{ \sqrt{n R^*_{\cJ}} } = \phi(t) - o_n(1),
\end{align*}
where $\phi(t) \deff \P(g \leq t)$ and $g \sim \cN(0,1)$.
Therefore,
\begin{align*}
\P\left\{Z_2 \geq R^*_{\cJ} + C \(\sqrt{\frac{R^*_{\cJ}}{n}}	\) \right\} & \geq 
\P\left\{ \sum_{i \in S^c}f(X_i) \geq  nR^*_{\cJ} + C  \sqrt{nR^*_{\cJ}}\right\} \\
				&= \P\left\{ \sum_{i \in S^c}Y_i \geq  |S|R^*_{\cJ} + C  \sqrt{nR^*_{\cJ}}\right\} \\
				&= \P\left\{  \frac{1}{\sqrt{|S^c|\Var[Y_i]}}\sum_{i \in S^c}Y_i \geq \frac{ |S|R^*_{\cJ} + C  \sqrt{nR^*_{\cJ}} }{\sqrt{|S^c|\Var[Y_i]} }\right\} \\
				&\geq  \phi\(\frac{ |S|R^*_{\cJ} + C  \sqrt{nR^*_{\cJ}} }{\sqrt{|S^c|\Var[Y_i]} }\) - o_n(1) \\
				&\geq \phi\(\frac{ n\P(A)R^*_{\cJ} + C  \sqrt{nR^*_{\cJ}} }{ \sqrt{n\P(A^c) \frac{R^*_{\cJ}}{2\P(A^c)}} }\) - o_n(1) \\
				&\geq \phi\( \P(A)\sqrt{nR^*_{\cJ}} + \sqrt{2}C\) - o_n(1) \geq c\phi\( C'+  \sqrt{2}C\)
\end{align*}
where we use that for $\alpha \geq \frac{1}{3}$, $\P(A)\sqrt{nR^*_\cJ} = \Theta\(n^{-\alpha + \frac{1- \alpha}{2}}\) = \O(1)$.

\subsubsection{Proof of Lemma~\ref{LemUnifSym}}

Let $\cH$ be the set of intervals of width equal to $2r$.
Currently the intervals near the end points have less probability mass. We will replace such intervals with bigger intervals to make the process symmetric. 
First consider the intervals near $\pm a$ which have less probability mass: we can instead focus on bigger intervals to include the middle interval $[-a,a]$. Let $\cJ \coloneqq \{\1_{[x, y]}: |x-y| = 2r + 2(b-a), |x+a| \leq 2r \}$. 
Next we can consider warping the number line and ``joining'' the two endpoints, i.e., let  $\cK \coloneqq \{\1_{[-\infty, x] \cup [y, \infty] }: 0\leq b-y \leq 2r, 0 \leq x + b \leq 2r, y -x = 2b -2r \}$.

Let $\cH' \coloneqq \cJ \cup \cK \cup \cH \setminus \{f \in \cH: R(f) < \frac{2r}{2(b-a)}\}$ and $\muhat'_{M,r} = \arg\max_{f \in \cH'}R_n(f)$.
Note that every function in $\cH'$ contains equal mass and the distribution is uniform.
Moreover, for $|x| \in [\frac{b-a}{2}, \frac{3(b-a)}{4}] $, $f_{x,r} \in \cH' \cap \cH$ because $b-a \geq C r$. Thus we have not removed a lot of functions from $\cH$. 

 The problem of the location of $\muhat'_{M,r}$ is equivalent to a uniform distribution on a circle of circumference $2(b-a)$, where we form the circle by joining $-a$ and $a$ at a single point, and join $-b$ to $b$.
By symmetry, we obtain that  $|\muhat'_{M,r}|$ is uniform on $[a, b]$. Thus 
$\P\( |\muhat'_{M,r}| \in  [\frac{b-a}{2},  \frac{3(b-a)}{4} ]\) = \frac{1}{4}  $.
\begin{align*}
\P\( |\muhat_{M,r}| \geq \frac{b-a}{2}\) &\geq \P\( |\muhat_{M,r}| \in  \left[\frac{b-a}{2},  \frac{3(b-a)}{4} \right]\) \\
			&\geq \P\( |\muhat'_{M,r}| \in  \left[\frac{b-a}{2},  \frac{3(b-a)}{4} \right]\) = \frac{1}{4}. 
\end{align*}
This proves the first statement.
Now, we consider the case when we condition on the value of $Z$. Note that if $|\muhat'_{M,r}| \in  \left[\frac{b-a}{2},  \frac{3(b-a)}{4} \right]$, then $Z = Z'$.
\begin{align*}
\P\( |\muhat_{M,r}| \geq \frac{b-a}{2} \bigg| Z \geq k\) &\geq \P\( |\muhat_{M,r}| \in  \left[\frac{b-a}{2},  \frac{3(b-a)}{4}  \right] \bigg| Z \geq k\) \\
			&\geq \P\( |\muhat'_{M,r}| \in  \left[\frac{b-a}{2},  \frac{3(b-a)}{4} \right] \bigg| Z \geq k\) \\
			&= \frac{\P\( |\muhat'_{M,r}| \in  \left[\frac{b-a}{2},  \frac{3(b-a)}{4} \right] ,  Z \geq k\)}{\P\( Z \geq k\)} \\
			&\geq \frac{\P\( |\muhat'_{M,r}| \in  \left[\frac{b-a}{2},  \frac{3(b-a)}{4} \right] ,  Z \geq k\)}{\P\( Z' \geq k\)} \\
			&= \frac{\P\( |\muhat'_{M,r}| \in  \left[\frac{b-a}{2},  \frac{3(b-a)}{4} \right] ,  Z' \geq k\)}{\P\( Z' \geq k\)} \\
			&= \P\( |\muhat'_{M,r}| \in  \left[\frac{b-a}{2},  \frac{3(b-a)}{4} \right] \bigg|  Z' \geq k\) 
			= \frac{1}{4}
\end{align*}
where we use the following Lemma~\ref{LemRotInd} for independence of $\muhat'_{M,r}$ and $Z'$.
\begin{lemma}
Suppose $X_1, \dots, X_n$ are i.i.d.\ uniform points on a circle. Let $E$ be the event that the maximum number of points contained in an arc of a certain length is equal to $k$. Then the joint distribution $p(x_1, \dots, x_n)$ is rotationally invariant.
\label{LemRotInd}
\end{lemma}

\begin{proof}
Suppose without loss of generality that the circle has circumference 1. Note that the law of $(X_1, \dots, X_n)$ can be equivalently generated as follows: First generate $Y_1, \dots, Y_n \stackrel{i.i.d.}{\sim} Unif[0,1]$. Next, generate $R \sim Unif[0,1]$, and define $X_i = Y_i + R$ for all $1 \le i \le n$, where the addition is taken modulo 1. We want to show that
\begin{equation}
\label{EqnRotate}
p(x_1, \dots, x_n \mid E) = p(x_1 + r, \dots, x_n + r \mid E)
\end{equation}
for any $r \in [0,1]$, where addition is again taken modulo 1. Clearly, it suffices to consider configurations $(x_1, \dots, x_n)$ that are consistent with $E$.

We can calculate
\begin{align*}
p(x_1, \dots,x_n \mid E) = \frac{\int_{E'} p(x_1, \dots, x_n, y_1, \dots, y_n)dy}{P(E)},
\end{align*}
where the integral is taken over the region of $[0,1]^n$ containing points $(y_1, \dots, y_n)$ that can be obtained from $(x_1, \dots, x_n)$ via some rotation. Importantly, note that
\begin{equation*}
p(x_1, \dots, x_n, y_1, \dots, y_n) = p(x_1, \dots, x_n \mid y_1, \dots, y_n) p(y_1, \dots, y_n) = p(y_1, \dots, y_n),
\end{equation*}
since $R$ is uniform, so we have
\begin{equation*}
p(x_1, \dots, x_n \mid E) = \frac{\int_{E'} p(y_1, \dots, y_n) dy}{P(E)}.
\end{equation*}

Similarly, we can write
\begin{equation*}
p(x_1 + r, \dots, x_n + r \mid E) = \frac{\int_{E'} p(x_1+r, \dots, x_n+r, y_1, \dots, y_n) dy}{P(E)} = \frac{\int_{E'} p(y_1, \dots, y_n) dy}{P(E)}.
\end{equation*}
This establishes the desired equality~\eqref{EqnRotate} and completes the proof.
\end{proof}

\subsection{Proof of Theorem~\ref{LemModalD}}
\label{AppThmModalD}

The initial steps in the proof parallel the proof of Theorem~\ref{cor:modal}, where Lemma~\ref{thm:modalIntervalMain} is proved using the concentration inequality in Lemma~\ref{thm:highProbD} rather than Lemma~\ref{thm:highProb}. It then follows that if we choose $r$ such that $R^*_r \ge C_{0.5} \left(\frac{(d+1) \log n}{n}\right)$, we have $R(f_{\muhat_{M,r},r}) \ge \frac{R^*_r}{2}$, w.h.p.

Now let $r_2 = 4r \left(\frac{2}{R^*_r}\right)^{\frac{1}{d}}$. By Lemma~\ref{lemma:prop2}(i), the desired result will follow if we can show that $R(f_{r_2, r}) \leq \frac{R^*_r}{2}$. By Lemma~\ref{lemma:prop2}(iv), we have
\begin{align*}
R(f_{r_2,r}) & \leq \frac{R^*_r}{2} \cdot R^*_{r_2} \le \frac{R^*_r}{2}.
\end{align*}

To obtain inequality~\eqref{EqnModalD}, note that using Lemma~\ref{lemma:prop2}\ref{prop2:5}, we know that $r = 2 \sqrt{d}\sigma_{(2Cd\log n)}$ satisfies the assumption on $R^*_r$. Plugging into inequality~\eqref{EqnModalD1} then produces the desired bound.

\subsection{Proof of Theorem~\ref{ThmModalCompute}}
\label{AppThmModalCompute}

We begin by deriving the proof for the modal interval estimator.
Let $s_1 = \frac{r}{2}$, and define $s_2$ such that $R(f_{s_2, r}) = \frac{1}{3} R(f_{s_1, r})$. Note that
\begin{equation*}
R(f_{s_1, r}) \ge R(f_{0,r/2}) \ge \frac{3C_{1/6} d\log n}{n},
\end{equation*}
so $R(f_{s_2, r}) \ge \frac{C_{1/6} d\log n}{n}$. Applying Lemma~\ref{ThmUniformProb} with $\bar{r} = s_1$ and $t = \frac{1}{6}$, we conclude that
\begin{equation}
\label{EqnFirst}
R_n(f_{x,r}) \ge \frac{2}{3} R(f_{x,r}) \ge \frac{2}{3} R(f_{s_1,r}),
\end{equation}
uniformly over $\|x\|_2 \le s_1$, with probability at least $1 - \frac{2\exp(-cn R(f_{s_1,r})/36)}{1-\exp(-cn R(f_{s_1,r})/36)}$, which is in turn lower-bounded by $1-4\exp(-c_1 d\log n)$.

Furthermore, inequality~\eqref{EqnLargeX} implies that
\begin{equation}
\label{EqnSecond}
R_n(f_{x,r}) \le R(f_{x,r}) + \frac{1}{3} R(f_{s_2, r}) \le \frac{4}{3} R(f_{s_2,r}) = \frac{4}{9} R(f_{s_1,r}),
\end{equation}
uniformly over $\|x\|_2 > s_2$, with probability at least $1-2\exp(-cn R(f_{s_2, r})/9) \ge 1 - 2\exp(-c_2 d\log n)$. Thus, combining inequalities~\eqref{EqnFirst} and~\eqref{EqnSecond}, we conclude that
\begin{equation}
\label{EqnThird}
\sup_{\|x\|_2 > s_2} R_n(f_{x,r}) < \inf_{\|x\|_2 \le s_1} R_n(f_{x,r}),
\end{equation}
with probability at least $1-6\exp(-c_3 d\log n)$.

Now note that by inequality~\eqref{EqnFirst}, we also have $R_n(f_{0, s_1}) \ge \frac{2}{3} R(f_{0, s_1}) > 0$, implying that $\{x_1, \dots, x_n\} \cap B(0, s_1) \neq \emptyset$. In particular,
\begin{equation*}
\sup_{x \in \{x_1, \dots, x_n\}} R_n(f_{x,r}) \ge \inf_{\|x\|_2 \le s_1} R_n(f_{x,r}).
\end{equation*}
Together with inequality~\eqref{EqnThird}, we conclude that $\|\mutil_{M,r}\|_2 < s_2$.

Finally, we claim that $s_2 \le 4r\left(\frac{n}{C_{1/6}d \log n}\right)^{1/d}$.
To see this, let $\tilde{s}_2 := 4r\left(\frac{n}{C_{1/6}d\log n}\right)^{1/d}$, and note that by Lemma~\ref{lemma:prop2}(iv), we have
\begin{align*}
R(f_{\tilde{s}_2, r}) & \le \frac{C_{1/6}d \log n}{n} \cdot R^*_{\tilde{s}_2} \le \frac{C_{1/6}d \log n}{n}.
\end{align*}
Since the last quantity is upper-bounded by $R(f_{s_2, r})$, we conclude that $s_2 \le \tilde{s}_2$, as claimed.

Turning to the analysis of the computationally efficient shorth estimator, we adapt the argument in the proof of Theorem~\ref{cor:shorth}. By Lemma~\ref{thm:highProbD}, if $R^*_{2r_{2k}} \ge \frac{C_{0.5} (d+1) \log n}{n}$, we have
\begin{equation*}
    \sup_x \sup_{r \le 2r_{2k}} \left(R_n(f_{x, r}) - R(f_{x,r})\right) < \frac{t}{2} R^*_{2r_{2k}},
\end{equation*}
with probability at least $1-2\exp(-cnR^*_{2r_{2k}}t^2) \ge 1-2\exp\left(-cnt^2 \cdot \frac{2k}{n}\right)$.

We know that $\frac{k}{n} = R_n(f_{\mutil_{S,k}, \rtil_k}) \le R_n(f_{\mutil_{S,k}, 2r_{2k}})$. Let $s$ be defined such that $R(f_{s, 2r_{2k}}) = \frac{k}{2n}$. By inequality~\eqref{EqnLargeX}, we know that
\begin{equation*}
\sup_{\|x\|_2 \ge s} \left|R_n(f_{x,2r_{2k}}) - R(f_{x,2r_{2k}})\right| \le \frac{1}{2} R(f_{s,2r_{2k}}),
\end{equation*}
with probability at least $1-2\exp(-ck)$, implying that for $\|x\|_2 \geq s$, we 
have
\begin{equation*}
R_n(f_{x,2r_{2k}}) \le R(f_{x,2r_{2k}}) + \frac{1}{2} R(f_{s,2r_{2k}}) \le \frac{3}{2} R(f_{s, 2r_{2k}}) = \frac{3k}{4n}.
\end{equation*}
Since this is strictly smaller than $R_n(f_{\mutil_{S,k},2r_{2k}})$, we conclude that $\|\mutil_{S,k}\|_2 \le s$, w.h.p.
, which also implies that $R(f_{\mutil_{S,k}, 2r_{2k}}) \ge \frac{k}{2n}$.

Finally, let $r' = 4r_{2k} \left(\frac{2n}{k}\right)^{1/d}$. By Lemma~\ref{lemma:prop2}(iv), we have
\begin{equation*}
    R(f_{r', 2r_{2k}}) < \frac{k}{2n} \cdot R^*_{r'} \le \frac{k}{2n} < R(f_{\mutil_{S,k}, 2r_{2k}}).
\end{equation*}
Applying Lemma~\ref{lemma:prop2}(i), we conclude that $\|\mutil_{S, k}\|_2 \le r'$.

\section{Shorth estimator}

In this appendix, we provide proofs of the various theorems and lemmas related to the shorth estimator.

\subsection{Proof of Theorem~\ref{cor:shorth}}
\label{app:shorthMassProof}

The proof of Theorem~\ref{cor:shorth} is similar in spirit to the proof of Theorem~\ref{cor:modal}. We begin by proving a lemma, which replaces Lemma~\ref{thm:modalIntervalMain}:

\begin{lemma}
For $2k \geq C_{0.5t}\log n$ and $t \in (0,1]$, with probability at least $1 - 2\exp(-c'kt^2)$, we have
\begin{align*}
R(f_{\widehat{\mu}_{S,k}, r_{2k}}) \geq (1-t)R^*_{r_{k}} = (1-t) \frac{k}{n}.
\end{align*}
\label{thm:shorthMass}
\end{lemma}

\begin{proof}
By assumption, we have $nR^*_{r_{2k}} = 2k \geq C_{0.5t}\log n$. Applying Lemma~\ref{thm:highProb} with $t = 0.5t$ and $r = r_{2k}$, we know that with probability at least $1 - \exp(-c2kt^2/4)$, we have
\begin{equation*}
\sup_{x,r \leq r_{2k}} R_n(f_{x,r}) - R(f_{x,r}) < \frac{t}{2}R^*_{r_{2k}}.
\end{equation*}
Combined with the guarantee of Lemma~\ref{lemma:shorthlength}, we conclude that
\begin{equation*}
R_n(f_{ \widehat{\mu}_{S,k} , \widehat{r}_k }) - R(f_{ \widehat{\mu}_{S,k} , \widehat{r}_k }) < \frac{t}{2} R^*_{r_{2k}},
\end{equation*}
with probability at least $1 - \exp(-ckt^2/2) - \exp(-k/8)$.

Furthermore, since all the distributions have densities, all the $X_i$'s are distinct with probability $1$, so $R_n(f_{ \widehat{\mu}_{S,k} , \widehat{r}_k }) = \frac{k}{n}$. We thus conclude that
\begin{align*}
\frac{k}{n} - R(f_{ \widehat{\mu}_{S,k} , \widehat{r}_k }) &<  \frac{t}{2} \cdot \frac{2k}{n},
\end{align*}
so $R(f_{ \widehat{\mu}_{S,k} , \widehat{r}_k }) > (1 - t)\frac{k}{n} = (1 - t)R^*_{r_{k}}$.
Again using the fact that $\widehat{r}_k \leq r_{2k}$, we can use Lemma~\ref{lemma:properties}\ref{prop:2} to conclude that
$R(f_{ \widehat{\mu}_{S,k} , \widehat{r}_k }) \leq R(f_{ \widehat{\mu}_{S,k} , r_{2k} })$, so the required statement holds.
\end{proof}

Let $r' = \frac{2nr_{2k}}{k}$. Taking $t = \frac{1}{2}$ in Lemma~\ref{thm:shorthMass} and using Lemma~\ref{lemma:properties}\ref{prop:1}, it suffices to show that $R(f_{r' , r_{2k}}) < \frac{k}{2n}$, which follows by Lemma~\ref{lemma:properties}\ref{prop:4} and the fact that $R^*_{r'} \leq 1$.

\subsection{Proof of Theorem~\ref{ThmShorthD}}
\label{AppThmShorthD}

We parallel the proof of Theorem~\ref{cor:shorth}. Note that the guarantees of Lemma~\ref{lemma:shorthlength} and Lemma~\ref{thm:shorthMass} continue to hold in $d$ dimensions, except that we have the lower bound $k \ge 2C_{0.5}(d+1) \log n$ instead, since the concentration inequality in Lemma~\ref{thm:highProb} will be replaced by the concentration inequality in Lemma~\ref{thm:highProbD}. We then conclude that $R(f_{\muhat_{S,k}, r_{2k}}) \ge \frac{k}{2n}$, with probability at least $1-2\exp(-c'd\log n)$.

Setting $r' = 4r_{2k} \left(\frac{2n}{k}\right)^{1/d}$, it thus suffices to show that $R(f_{r', r_{2k}}) \le \frac{k}{2n}$. By Lemma~\ref{lemma:prop2}(iv), we have
\begin{equation*}
R(f_{r', r_{2k}}) \le \frac{k}{2n} \cdot R^*_{r'} \le \frac{k}{2n},
\end{equation*}
as wanted.

\section{Proofs for the hybrid estimator} %

In this appendix, we provide proofs of the results related to the single- and multi-dimensional hybrid estimator.

\subsection{Proof of Lemma~\ref{lemma:median}}
\label{app:k_median}

The $k$-median was defined using $\psi_n$. It is therefore instructive to study the properties of the population-level quantity $\psi(\theta) \coloneqq \E \psi_n(\theta)$.  For $\theta > 0$, we have
\begin{align*}
\psi(\theta) &\coloneqq \E \psi_n(\theta) = \frac{1}{n} \sum_{i=1}^n \E[\text{sign}(\theta - X_i)] 
\\
 &
 = \frac{1}{n} \sum_{i=1}^n\P( -\theta \leq X_i < \theta) = R(f_{0,\theta})  = R^*_{\theta}.
\end{align*}
In particular $\psi(r_{k}) = R^*_{r_k} = \frac{k}{n}$.
Similarly, for $\theta < 0$, we have $\psi(\theta) = -R^*_\theta$.

We will focus only on the error on the positive side, i.e., $\widehat{\theta}_{\text{med},k} > r_{k + \delta}$. The analysis for $\widehat{\theta}_{\text{med},-k} < - r_{k + \delta}$ is similar by symmetry.  
Recall that $\psi_n(\widehat{\theta}_{\text{med},k}) = \frac{k}{n}$ a.s., so by monotonicity of $\psi_n$, it follows that
\begin{align*}
\P\(  \widehat{\theta}_{\text{med},k} > r_{k+\delta} \) &\leq \P\(  \psi_n(r_{k + \delta}) \le \frac{k}{n} \)  \\
&= \P\(  \psi_n(\epsilon) - \psi(\epsilon) \le -\(\psi(\epsilon) - \frac{k}{n}\) \).
  \end{align*}
Since $\psi_n(\cdot) - \psi(\cdot)$ is a centered sum of independent bounded random variables, we may apply Hoeffding's inequality on its negative tail. By assumption,  $\psi(\epsilon)  - \frac{k}{n} = \delta \ge 0$. Therefore,
\begin{align*}
\P\(  \widehat{\theta}_{\text{med},k} > \epsilon \)  \leq \exp\(-cn\(\psi(\epsilon) - \frac{k}{n}\)^2\) \le \exp(-cn\delta^2).
\end{align*}

\subsection{Proof of Proposition~\ref{prop:exmMedian}}
\label{app:PropExmMedian}

In the following, we will show the bounds on $r_{2\sqrt{n}\log n}$, which gives us the result:
\begin{enumerate}
	\item As in the proof of Proposition~\ref{prop:exmple_r_k}, we have $r_k = \Theta\(\frac{\sigma k}{n}\)$ for small $k$.
	\item By Lemma~\ref{lemma:properties}\ref{prop:1}, we have $r_{2\sqrt{n}\log n}\leq 2\sigma_{(4\sqrt{n}\log n )} = \O\(\sqrt{n}\log n\)$.
	\item Note that for any fixed $k$, the value of $r_{k}$ for Example~\ref{exm:alpha-mix} is smaller than the value of $r_{k}$ for Example~\ref{exm:IID} with $\sigma = n^\alpha$.
	Thus, we have $r_{2 \sqrt{n}\log n} = \O\(\frac{n^\alpha \sqrt{n}\log n }{n}\) = \O\( n^{\alpha-0.5}\log n\) $. 
\end{enumerate}

\subsection{Proof of Lemma~\ref{LemCubOrigin}}
\label{AppLemCubOrigin}

We first show (i). Note that by Lemma~\ref{lemma:kmedianContains0}, we know that for a fixed $i$, we have $0 \in [\min(S_{k,i}, \max(S_{k,i}))]$ with probability at least $1 -  2\exp(-k^2/n)$. Taking a union bound gives the desired result.

Note that (ii) follows from Lemma~\ref{lemma:median} and a union bound.

\subsection{Proof of Theorem~\ref{thm:hybrid}}
\label{app:proof_hybrid}

\begin{proof}
Let $r' = \frac{4\sqrt{n}\log n}{k_2} r_{2k_2}$.
We break down the analysis in two cases:

\paragraph{\textbf{Case 1:}} Suppose $R^*_{r'} \geq \frac{2\log n}{\sqrt{n}} $. Note that $r' \geq r_{2k_1}$, so
  by Lemma~\ref{lemma:median}, we have $\max(S_{k_1}) \leq r'$ and $\min(S_{k_1}) \geq -r'$, with probability at least $1 - 2\exp(-c\log^2n)$.
  Since the final prediction is always within the set spanned by $S_{k_1}$, we must have $ |\widehat{\mu}_{k_1,k_2}| \leq  r'$ with probability at least $1 - 2\exp(-c\log^2n)$.

\paragraph{\textbf{Case 2:}} Suppose $R^*_{r'} < \frac{2 \log n}{\sqrt{n}}$. We will first show that $ |\widehat{\mu}_{S,k_2}| \leq r'$. Similar to the proof of Theorem~\ref{cor:shorth}, it suffices to show that $R(f_{r',r_{2k_2}}) < \frac{k_2}{2n}$. Indeed, we have
\begin{align*}
R(f_{r',r_{2k_2}}) < \frac{r_{2k_2} }{ r'  }R^*_{r'} <  \frac{1}{\frac{4\sqrt{n}\log n}{k_2}}\frac{2 \log n}{ \sqrt{n}} = \frac{k_2}{2n},
\end{align*}
with probability at least $1 - 2\exp(-c'k_2)$.

However, we are still not completely done. We still have to show that the final prediction, $\widehat{\mu}_{k_1,k_2}$, is not any worse. Based on Algorithm~\ref{alg:hybrid2}, we have two cases to consider:
\begin{enumerate}
  \item $\widehat{\mu}_{S,k_2} \in [\min(S_{k_1}),\max(S_{k_1})]$.

    Then the output is the shortest gap estimator itself, so $|\widehat{\mu}_{k_1,k_2}| =  |\widehat{\mu}_{S,k_2}|$.
    
  \item $\widehat{\mu}_{S,k_2} \not\in [\min(S_{k_1}),\max(S_{k_1})]$.
  
  We will use the following lemma, a slight generalization of Lemma $4.1$ from Chierichetti et al.~\cite{ChiEtAl14}. The lemma states that for sufficiently large values of $k$, the $k$-median contains the true mean, $\mustar = 0$, with high probability.

\begin{lemma}
\label{lemma:kmedianContains0}
Let $S_k$ be the output of the $k$-median algorithm. Then with probability at least $1 - 2\exp\(-c\frac{k^2}{n}\)$, we have $0 \in [\min(S_k),\max(S_k)]$. 
\label{lemma}
\end{lemma}
\begin{proof}
We bound the probability that $\max(S_k) < 0$; the bound for $\min(S_k) > 0$ is analogous. If $\max(S_k) < 0$, then $\psi_n(0) \ge \frac{k}{n}$ by monontonicity of $\psi_n$ and the fact that $\max(S_k) = \widehat{\theta}_{\text{med},k}$ and $\psi_n(\widehat{\theta}_{\text{med},k}) = \frac{k}{n}$.
By Hoeffding's inequality, we then have
\begin{align*}
\P\(\max(S_k) < 0\) & \leq \P\( \psi_n(0) \ge \frac{k}{n} \) = \P\left(\psi_n(0) - \psi(0) \ge \frac{k}{n}\right) 
\\
& \leq \exp\left(-cn \cdot \frac{k^2}{n^2}\right) = \exp\(-c\frac{k^2}{n}\).
\end{align*}
\end{proof}

    By Lemma~\ref{lemma:kmedianContains0}, we have $0 \in [\min(S_{k_1}), \max(S_{k_1})]$ with probability at least $1 - 2\exp(-c\log^2n)$. If this event holds, then the output $\widehat{\mu}_{k_1, k_2} =$ closestPoint$(S_{k_1}, \widehat{\mu}_{S,k_2})$ must lie between $\widehat{\mu}_{S,k_2}$ and 0. It follows that $|\widehat{\mu}_{k_1,k_2}| \le |\widehat{\mu}_{S,k_2}|$.

Altogether, we conclude that $|\widehat{\mu}_{k_1,k_2}| \leq r'$, with probability at least $1 - 2\exp(-c'k_2) - 4\exp(-c\log^2n)$. 
\end{enumerate}
\end{proof}

\subsection{Proof of Theorem~\ref{thm:HybridScreening}}
\label{AppThmHybridMult}

We first derive an upper bound of $\sqrt{d} r_{ 2\sqrt{n}\log n,1}$.
We begin by deriving the following lemma, relating the statistics of marginal distributions to the statistics of the overall distribution:

\begin{lemma}
\label{LemRkMarginal}
We have that $r_{\frac{k}{2},1} \leq \frac{C}{\sqrt{d} } r_{k}$, for some absolute constant $C > 0$ and any $k \leq n$.
\end{lemma}

\begin{proof}
Consider a uniform distribution on a sphere (or shell) of radius $r$ in $\real^d$. Theorem 3.4.6 in Vershynin~\cite{Ver18} provides a concentration result which states that most of the probability of such a distribution lies close to the equator; i.e., the set $\left[-\frac{Cr}{\sqrt{d}}, \frac{Cr}{\sqrt{d}}\right] \times \real^{d-1}$  contains at least half the probability for some absolute constant $C > 0$. Notice that a radially symmetric distribution is simply a weighted sum of uniform distributions on spheres. Thus, given a radially symmetric distribution restricted to the ball of radius $r$, the set $\left[-\frac{Cr}{\sqrt{d}}, \frac{Cr}{\sqrt{d}}\right] \times \real^{d-1}$ will contain at least half the total probability assigned to the ball.

By our definition of $r_k$, the ball of radius $r_k$ centered at origin, $\B_{r_k}$, contains $\frac{k}{n}$ probability mass. The above argument implies that the set $\left[-\frac{Cr_k}{\sqrt{d}}, \frac{Cr_k}{\sqrt{d}}\right] \times \real^{d-1}$ will contain at least half the probability of the total probability contained in $\B_{r_k}$. Equivalently, $r_{\frac{k}{2},1} \leq \frac{C}{\sqrt{d} } r_{k}$. 
\end{proof}

Since the output of the hybrid algorithm must lie within the cuboid $S_{\sqrt{n} \log n}^\infty$, it is clear that we have the error bound
\begin{equation*}
\|\muhat_{k_1, k_2}\|_2 \le \sqrt{n}^{1/d} \cdot \sqrt{d} r_{ 2\sqrt{n}\log n,1}.
\end{equation*}

To obtain the second upper bound expression, we parallel the proof of Theorem~\ref{thm:hybrid}, by splitting into two cases:

\paragraph{\textbf{Case 1:}} $r_{4\sqrt{n} \log n} \le \sqrt{n}^{1/d} r_{8d \log n}$.
By Lemma~\ref{LemRkMarginal}, we therefore have
\begin{equation*}
r_{ 2\sqrt{n} \log n,1} \le \frac{C}{\sqrt{d}} r_{4\sqrt{n} \log n} \le \frac{C}{\sqrt{d}} \cdot \sqrt{n}^{1/d} r_{8d\log n}.
\end{equation*}
By Lemma~\ref{LemCubOrigin}, w.h.p., the cuboid $S_{\sqrt{n} \log n}^\infty$ is entirely contained in the $\ell_2$-ball of radius $\sqrt{d} r_{ 2\sqrt{n} \log n,1}$ around the origin. This ball in turn lies inside the $\ell_2$-ball of radius $C \sqrt{n}^{1/d} r_{8d\log n}$ around the origin. Since the output of the hybrid algorithm must also lie within this ball, the desired result follows.

\paragraph{\textbf{Case 2:}} $r_{4\sqrt{n} \log n} > \sqrt{n}^{1/d} r_{8d \log n}$. Denoting $r' = \sqrt{n}^{1/d} r_{8d \log n}$, we therefore have the relation $R^*_{r'} < \frac{4\sqrt{n} \log n}{2n}$. In particular, since
\begin{equation*}
R(f_{\muhat_{S, 8d\log n}, r_{8d\log n}}) \ge R_n(f_{\muhat_{S, 8d\log n}, r_{8d\log n}}) - \frac{1}{2} R^*_{r_{8d\log n}} = \frac{8d \log n}{4n},
\end{equation*}
w.h.p., by Lemma~\ref{thm:highProbD}, we have
\begin{equation*}
R(f_{r', r_{8d \log n}}) \le \left(\frac{1}{\sqrt{n}^{1/d}}\right)^d R^*_{r'} < \frac{1}{\sqrt{n}} \cdot \frac{2\log n}{\sqrt{n}} = \frac{8d\log n}{4n} \le R(f_{\muhat_{S, 8d\log n}, r_{8d \log n}}).
\end{equation*}
This implies that $\muhat_{S, 8d \log n}$ is within $r'$ of the origin.

Finally, we need to show that projecting the shorth estimator on the cuboid does not increase its distance from the origin. Note that $\ell_2$-projection onto a cuboid is simply a componentwise operation of projection on each interval defining an edge of the cuboid. Furthermore, Lemma~\ref{LemCubOrigin} guarantees that the origin lies within the cuboid, w.h.p., in which case each interval contains 0. As argued in the proof of Theorem~\ref{thm:hybrid}, the distance from the shorth estimator to the origin computed along any dimension will not increase after the projection. Therefore, the $\ell_2$-norm of the projected estimator is also upper-bounded by $r'$.

Hence, if we take $C' = \max\{C, 1\}$, we have the desired bound in both cases. This concludes the proof.

\section{Proofs for expected error bounds}
\label{app:ExpBounds}

In this appendix, we prove the results stated in Section~\ref{SecExpect}.

\subsection{Proof of Proposition~\ref{prop:HighExp2}}
\label{app:HighExp2}

The proof sketch is that we will show that with finite probability, no interval contains more than one low-variance point, and all the high-variance points lie far from origin. Conditioned on this event, the modal interval estimator incurs a high error. 

Let $E = A \cap B$, where we define the events
\begin{align*}
 A  & =\{ R_n(f_{x,1}) \leq 1, \quad \forall x: |x| \leq 3C\log n\}, \\
B & = \{ X_i \not \in [-4C\log n, 4C\log n], \quad \forall i > C\log n\}.
 \end{align*}
Hence, on the event $E$, no interval overlapping with $[-3C\log n, 3C\log n]$ contains two low-variance points or a single high-variance point. Then $\P(E)$ is lower-bounded by
\begin{align*}
\P(E) &\geq  \(\prod_{i=1}^{C\log n}  \P\{X_i \in  [3i -3 , 3i - 2]\}\)\(\prod_{i > C\log n} \P\{X_i \not \in [-4C\log n, 4C\log n]\} \) \\
	 &= \(\prod_{i=1}^{C\log n} \frac{1}{6i} \)\(\prod_{i > C\log n} (1 - n^{- \alpha} - h_n (8C\log n -2)) \) \\
	 &\geq \frac{1}{6^{C \log n}\Gamma(3C \log n)} e^{-c n^{1 - \alpha}} \\
	 &\geq  \exp\( -cn^{1 - \alpha} - \O\(\log^2n \)  \),
\end{align*}
assuming $h_n \log n \ll n^{-\alpha}$, which happens for $q_n = \Omega(n)$.

However, conditioned on $E$, the points $\{X_i\}_{i > C \log n}$ are \iid with the following distribution:
\begin{align*}
 p_{i,E}(x) = \begin{cases} 0 , & |x| \leq  4C \log n, \\
					\frac{h_n}{(1 - n^{- \alpha} - h_n (8C\log n -2)) }, &4C \log n <  |x| \leq q_n, \\
					0, & \text{otherwise}.
		\end{cases}
\end{align*}
We can now apply the symmetry arguments of Lemma~\ref{LemUnifSym}.
Note that no interval lying inside $[-3C\log n, 3C\log n]$ can contain more than one point.
Thus unless a tie occurs, the mode will be located outside the interval $[-3C\log n, 3C \log n]$, and hence a distance of $\Theta(q_n)$ away from the mean in expectation.
Even if we were to break ties randomly, a large error would occur with probability at least $\frac{1}{n}$, since at most $n$ ties can occur.
Thus,
\begin{equation*}
\E[|\widehat{\mu}_{M,1}||E] \geq \P(E) \E[|\widehat{\mu}_{M,1}||E] 
				\geq \exp( -cn^{1 - \alpha})\Theta(q_n).
\end{equation*}
The bounds in high probability follow from Theorem~\ref{thm:modalIntervalMain} by noting that $nR^*_r = \Omega\(n^{-\alpha}\) = \Omega(\log n)$. Moreover, the density drops by at least half at $x > 1$.

\subsection{Proof of Theorem~\ref{ThmExpBound}}
\label{AppThmExpBound}

We begin by proving (i). By Theorem~\ref{LemModalD}, we have
\begin{align*}
\|\muhat_{M,r}\|_2 = \O\(r\left(\frac{c}{R^*_r}\right)^{1/d}\),
\end{align*}
with probability at least most $1 - \O(\exp\(-c' nR^*_r\))$. In the worst case, the modal interval estimator returns the point which is furthest from the origin, which has expected value bounded as
\begin{equation*}
\E\left[\max_{i} \|X_i\|_2\right] \leq \E\left[\sqrt{\sum_{i=1}^n \|X_i\|_2^2}\right] \le \sqrt{\sum_{i=1}^n \E[\|X_i\|_2^2]} \le \sqrt{n \cdot d \sigma_{(n)}^2}.
\end{equation*}

Using the assumption that $\sigma_{n} \leq r \exp(CnR^*_r)$, for some constant $C > 0$, we then have
\begin{align*}
\E\|\muhat_{M,r}\|_2 & \leq \O\(r\left(\frac{c}{R^*_r}\right)^{1/d}\) + \O(\exp\(-c' nR^*_r\)) \sqrt{nd} \sigma_{(n)} 
\\
& \leq \O\(r\left(\frac{c}{R^*_r}\right)^{1/d}\) +  \O\(\exp\(-c'nR^*_r\) r \sqrt{nd} \exp(CnR^*_r)\) \\
& = \O \(r\left(\frac{c}{R^*_r}\right)^{1/d}\),
\end{align*}
where in the last inequality, we use the facts that
\begin{equation*}
\exp(-c' nR_r^*) \sqrt{nd} = O(\exp(-c'' n R_r^*))
\end{equation*}
and $nR^*_r = \Omega\(d\log n\)$, and choose $C < c''$.

Turning to (ii), we first prove the following concentration result, which may be viewed as a refinement of Lemma~\ref{thm:highProb} that is suitable for our settings. For example, note that if $R^*_{\cJ} =\O\left(\frac{1}{n}\right)$, the derivations from Lemma~\ref{thm:highProb} would not be meaningful since $R^*_{\cJ} = o\left(\frac{\log n}{n}\right)$. On the other hand, if $KR^*_{\cJ} = \Theta\left(\frac{\log n}{n}\right)$, Lemma~\ref{LemBadIntErr} gives a vanishing upper bound. 

\begin{lemma}
Let $\cJ$ be a set of intervals and define $R^*_{\cJ} := \sup_{f \in \cJ} R(f)$. If $R^*_{\cJ} \le \frac{1}{3}$, then for any $K \geq 8$, we have
\begin{align*}
\P\left\{\sup_{f \in \cJ } R_n(f) \geq K R^*_{\cJ}\right\}   &\leq \frac{2}{R^*_\cJ} \exp\(- cn R^*_{\cJ}  K \log K  \).
\end{align*}
\label{LemBadIntErr}
\end{lemma}

\begin{proof}
For a given $f \in \cJ$, the desired bound follows from Chernoff's inequality. 
We want to upper-bound the probability that any one interval in $\cJ$ has too many points. 
In general, the set $\cJ$ may be infinite, so a direct union bound is not feasible.
We thus create a new finite set of intervals $\cF$, not necessarily a subset of $\cJ$, satisfying the following properties: 
\begin{enumerate}
 	\item For each $f \in \cF$, we have $\frac{R^*_{\cJ}}{2} \le R(f) \le R^*_{\cJ}$.
 	\item $|\cF| \leq \frac{2}{R^*_{\cJ}}$.
 	\item $\cF$ covers $\cJ$ in the sense that $\forall f \in \cJ, \exists f_1,f_2 \in \cF: f(x) \leq f_1(x) + f_2(x)$.
 \end{enumerate} 
It follows that if any interval in $\cJ$ contains at least $k$ points, then at least one interval in $\cF$ contains at least $\frac{k}{2}$ points. We construct $\cF$ of cardinality $|\cF| = \lceil \frac{1}{R^*_{\cJ}} \rceil \leq \frac{2}{R^*_\cJ}$, as follows: To create the first interval ($i=1$),  define $x_1 \in \real$ such that $R(\1_{(-\infty,x_1])}) = \frac{1}{|\cF|}$. (Such an $x_1$ exists because $\overline{P}$ is assumed to have a density.) Then iteratively, for each $i \geq 1$, define $x_i$ such that $R(\1_{(x_{i-1},x_i]}) = \frac{1}{|\cF|}$. For the final interval, add $\1_{[x_{i-1}, \infty)}$ to $\cF$ and terminate the construction. Note that for each $f \in \cF$, we have $R(f) = \frac{1}{\lceil 1/R^*_{\cJ} \rceil}$, which clearly lies in $\left[\frac{R^*_{\cJ}}{2}, R^*_{\cJ}\right]$ under the assumptions.

 We are now ready to apply the union bound on $\cF$ using Lemma~\ref{LemChernoff}(ii):
 \begin{align*}
\P\left\{\sup_{f \in \cJ} R_n(f) \geq KR^*_{\cJ}\right\} 
&\leq \P\left\{\sup_{f \in \cF} R_n(f) \geq \frac{KR^*_{\cJ}}{2}\right\}\\
 &\leq |\cF| \P\left\{ R_n(f) \geq \frac{KR^*_{\cJ}}{2} \text{ for a fixed } f \text{ with } R(f) \leq R^*_{\cJ}\right\}\\
 &\leq \frac{2}{R^*_{\cJ}}\exp\(-cnR^*_{\cJ} K \log K \).
 \end{align*} 
\end{proof}

For an $s \geq 0$, let $\cJ_s = \{f_{x,r} : \|x\|_2 \geq s\}$, i.e., the set of intervals which incur large error.
By assumption, the support of at least $CnR^*_r$ points is contained in $[-r,r]$, implying that $R_n(f_{0,r}) \geq CR^*_r$, a.s.
If $\|\muhat_{M,r}\|_2 \geq s$, then $ \sup_{f \in \cJ_s} R_n(f) \geq CR^*_r$.
However as $s$ increases, the quantity $R^*_{\cJ_s} := \sup_{f \in \cJ_s}  R(f) = R(f_{s,r})$ decreases.
We can then use Lemma~\ref{LemBadIntErr} to control this probability of error.

For $s \geq \frac{K r}{CR^*_r}$, it follows from Lemma~\ref{lemma:prop2}\ref{prop2:4} that $R^*_{\cJ_s} = R(f_{s,r}) \leq \frac{CR^*_{r}}{K}$. Taking $K \ge C'$, we then have
\begin{align*}
\P\{|\widehat{\mu}_{M,r}| \geq s\} &\leq \P\left(\sup_{f \in \cJ_s} R_n(f) \geq CR^*_r\right) \\
& = \P\left(\sup_{f \in \cJ_s} R_n(f) \geq \frac{CR^*_r}{R^*_{\cJ_s}} R^*_{\cJ_s}\right) \\
		&\leq  \frac{2}{R^*_{\cJ_s}} \exp\(- cnR^*_{\cJ_s} \frac{CR^*_r}{R^*_{\cJ_s}} \log\(\frac{CR^*_r}{R^*_{\cJ_s}}\)  \) \\
		&= \frac{2}{R^*_r}\exp\(- cCnR^*_r \log\(\frac{CR^*_r}{R^*_{\cJ_s}}\)  + \log\(\frac{R^*_r}{R^*_{\cJ_s}}\)\) \\
		&\leq \frac{2}{R^*_r} \exp\(- c'nR^*_r \log\(\frac{R^*_r}{R^*_{\cJ_s}}\) \),
\end{align*}
where we have applied Lemma~\ref{LemBadIntErr} in the second inequality. Thus,
\begin{align*}
\E|\muhat_{M,r}| &\leq \frac{4r}{CR^*_r} + \int_{\frac{4r}{CR^*_r} }^ \infty  \P\{|\widehat{\mu}_{M,r}| \geq s\} ds \\
				&\leq \O\(\frac{r}{R^*_r}\) + \frac{2}{R^*_r} \int_{\frac{4r}{CR^*_r} }^ \infty \exp\(- c'nR^*_r \log\(\frac{R^*_r}{R^*_{\cJ_s}}\) \)  ds \\
				&\leq \O\(\frac{r}{R^*_r}\) + \frac{2}{R^*_r} \int_{\frac{4r}{CR^*_r} }^ \infty \exp\(- c'nR^*_r \log\(\frac{sR^*_r}{r}\)  \)  ds \\
				&\leq \O\(\frac{r}{R^*_r}\) + \frac{r}{R^*_r} \frac{2}{R^*_r} \int_{4/C}^ \infty \exp\(-c'nR^*_r \log s_1\)  ds_1 \\
				&= \O\(\frac{r}{R^*_r}\) + \frac{r}{R^*_r} \frac{2}{R^*_r} \int_{4/C}^ \infty s_1^{-c'nR^*_r}  ds_1 \\
				&\leq \O\(\frac{r}{R^*_r}\) + \frac{r}{R^*_r} \frac{2}{R^*_r} \cdot \frac{1}{c'nR_r^* - 1} (4/C)^{1-c'nR_r^*} \\
				&= \O\(\frac{r}{R^*_r}\),
\end{align*}
where the third inequality uses the fact that $R^*_{\cJ_s} = R(f_{s,r}) \le \frac{r}{s}$, and the last equality follows from an appropriately small choice of $C$.

\subsection{Proof of Theorem~\ref{ThmLowerBound}}
\label{AppThmLowerBound}

Note that for any $s > 0$, Markov's inequality gives
\begin{align*}
\min_{\muhat} \max_{\{P_i\} \subseteq \scriptP(\sigma_1, \sigma_2, p)} \E[\|\muhat - \mu\|_2] & \ge \min_{\muhat} \max_{\{P_i\} \subseteq \scriptP(\sigma_1, \sigma_2, p)} s \cdot \mprob(\|\muhat - \mu\|_2 \ge s).
\end{align*}
Clearly, the right-hand expression is lower-bounded by the maximum over any specific collection of distributions in the class $\scriptP(\sigma_1, \sigma_2, p)$. In particular, let $\scriptP_m^\mu$ be the collection of multivariate distributions where each distribution is either $N(\mu, \sigma_1^2 I)$ or $N(\mu, \sigma_2^2 I)$, with $m$ distributions of the latter type. We then have
\begin{align*}
\min_{\muhat} \max_{\{P_i\} \subseteq \scriptP(\sigma_1, \sigma_2, p)} \mprob(\|\muhat - \mu\|_2 \ge s) & \ge \min_{\muhat} \max_\mu \max_{np \le m \le 2np} \mprob(\|\muhat - \mu\|_2 \ge s \mid \{P_i\} = \scriptP_m^\mu) \\
& \ge \min_{\muhat} \max_\mu \sum_{np \le m \le 2np} \mprob\left(\|\muhat - \mu\|_2 \ge s \mid \{P_i\} = \scriptP_m^\mu \right) p_m,
\end{align*}
where $\{p_m\}$ is any allocation of probabilities defined over $\{\scriptP_{np}^\mu, \dots, \scriptP_{2np}^\mu\}$, such that $0 \le p_m \le 1$ for all $m$ and $\sum_m p_m \le 1$. In particular, consider the probability mass function $\{q_m\}_{m=1}^n$ over $\{\scriptP_1^\mu, \dots, \scriptP_n^\mu\}$ corresponding to the Binomial$(n, p)$ distribution, and define $p_m = q_m$ for all $np \le m \le 2np$.

Now let $\mprob_{\text{Bin}}^\mu$ denote the probability distribution when the $P_i$'s are chosen i.i.d.\ in the following manner: with probability $p' := 1.5p$, the distribution is $N(\mu, \sigma_2^2 I)$, and with probability $1-1.5p$, the distribution is $N(\mu, \sigma_1^2 I)$. Then
\begin{equation*}
\mprob_{\text{Bin}}^\mu(\|\muhat - \mu\|_2 \ge s) = \sum_{m =1}^n \mprob\left(\|\muhat - \mu\|_2 \ge s \mid \{P_i\} = \scriptP_m^\mu\right) q_m.
\end{equation*}
Hence,
\begin{align*}
\left|\sum_{np \le m \le 2np} \mprob\left(\|\muhat - \mu\|_2 \ge s \mid \{P_i\} = \scriptP_m^\mu \right) p_m - \mprob_{\text{Bin}}^\mu (\|\muhat - \mu\|_2 \ge s)\right| & \le \sum_{m < np} q_m + \sum_{m > 2np} q_m \\
& \le 2\exp(-cnp) \\
& \le 2\exp(-c'\log n),
\end{align*}
where second inequality follows from the multiplicative Chernoff bound (Lemma~\ref{LemChernoff}) and the last inequality follows by the assumption $p = \Omega\left(\frac{\log n}{n}\right)$. Combining the inequalities, we conclude that
\begin{equation*}
\min_{\muhat} \max_{\{P_i\} \subseteq \scriptP(s_1, s_2, p)} \E[\|\muhat - \mu\|_2] \ge s \left( \min_{\muhat} \max_\mu \mprob_{\text{Bin}}^\mu(\|\muhat - \mu\|_2 \ge s) - 2 \exp(-c'\log n)\right).
\end{equation*}
Thus, it suffices to find $s$ such that the expression $\min_{\muhat} \max_\mu \mprob_{\text{Bin}}^\mu(\|\muhat - \mu\|_2 \ge s)$ can be lower-bounded by a constant.

For part (i), using standard techniques~\cite{Tsy08, Wai19}, we may obtain such a lower bound via Fano's inequality. In particular, if we can construct a set $\{\mu_1, \dots, \mu_M\} \subseteq \real^d$ such that $\|\mu_j - \mu_k\|_2 \ge 2s$ and $KL(\mprob_{\text{Bin}}^{\mu_j}, \mprob_{\text{Bin}}^{\mu_k}) \le \alpha$ for all $j \neq k$, then
\begin{equation*}
\min_{\muhat} \max_\mu \mprob_{\text{Bin}}^\mu(\|\muhat - \mu\|_2 \ge s) \ge s\left(1 - \frac{\alpha + \log 2}{\log M}\right).
\end{equation*}
Note that by tensorization and convexity of the KL divergence, we have the upper bound
\begin{equation}
\label{EqnKLExp}
KL(\mprob_{\text{Bin}}^{\mu_j}, \mprob_{\text{Bin}}^{\mu_k}) \le n(1-p') KL\left(N(\mu_j, \sigma_1^2 I), N(\mu_k, \sigma_1^2 I)\right) + np' KL\left(N(\mu_j, \sigma_2^2 I), N(\mu_k, \sigma_2^2 I)\right),
\end{equation}
where the KL divergences in the right-hand expression are computed with respect to single samples from the respective multivariate normal distributions. Furthermore, the right-hand side of inequality~\eqref{EqnKLExp} is easily calculated to be
\begin{equation*}
n(1-p') \cdot \frac{\|\mu_j - \mu_k\|_2^2}{2\sigma_1^2} + np' \cdot \frac{\|\mu_j - \mu_k\|_2^2}{2\sigma_2^2} = n\|\mu_j - \mu_k\|_2^2\left(\frac{1-p'}{2\sigma_1^2} + \frac{p'}{2\sigma_2^2}\right).
\end{equation*}

In particular, suppose $\{\mu_1, \dots, \mu_M\}$ is a $2s$-packing of the ball of radius $4s$ in $\ell_2$-norm, with $s = C \sqrt{d} \min\left\{\frac{\sigma_1}{\sqrt{n}}, \frac{\sigma_2}{\sqrt{np'}}\right\}$. Then $\log M \ge cd$ and
\begin{equation*}
KL(\mprob_{\text{Bin}}^{\mu_j}, \mprob_{\text{Bin}}^{\mu_k}) \le 4ns^2\left(\frac{1-p'}{2\sigma_1^2} + \frac{p'}{2\sigma_2^2}\right) \le 4C^2 d := \alpha.
\end{equation*}
For a sufficiently small choice of $C$, we conclude that $\min_{\muhat} \max_\mu \mprob_{\text{Bin}}^\mu(\|\muhat - \mu\|_2 \ge s) \ge \frac{1}{2}$. Hence, we arrive at the desired bound~\eqref{EqnLB1}.

We now turn to part (ii). We derive the tighter lower bound~\eqref{EqnLB2} for the case $d = 1$ by evaluating $KL(\mprob_{\text{Bin}}^{\mu_1}, \mprob_{\text{Bin}}^{\mu_2})$ more directly. By Theorem 2.2 in Tsybakov~\cite{Tsy08}, we know that if we have a pair $\mu_1, \mu_2 \in \real^d$ such that $\|\mu_1 - \mu_2\|_2 \ge 2s$ and
\begin{equation}
\label{EqnKLUpper}
KL(\mprob_{\text{Bin}}^{\mu_1}, \mprob_{\text{Bin}}^{\mu_2}) \le \alpha < \infty,
\end{equation}
then
\begin{equation*}
\min_{\muhat} \max_\mu \mprob_{\text{Bin}}^\mu(\|\muhat - \mu\|_2 \ge s) \ge \max\left\{\frac{\exp(-\alpha)}{4}, \frac{1-\sqrt{\alpha/2}}{2}\right\}.
\end{equation*}
Again, since the KL divergence tensorizes, it suffices to compute the KL divergence between a single sample from the distributions $\mprob_{\text{Bin}}^{\mu_1}$ and $\mprob_{\text{Bin}}^{\mu_2}$, which we denote by $\mprob_1$ and $\mprob_2$, respectively.

We provide the details of the calculation for general $d$, with the assumption~\eqref{EqnSratio} replaced by the condition
\begin{equation}
\label{EqnSratioD}
\left(\frac{\sigma_1}{\sigma_2}\right)^d = O\left(\frac{1}{np^2}\right).
\end{equation}
By a straightforward calculation, we have
\begin{align*}
\log\left(\frac{d\mprob_1(x)}{d\mprob_2(x)}\right) & = \log\left(\frac{(1-p') \frac{1}{(\sqrt{2\pi} \sigma_1)^d} \exp\left(\frac{-\|x - \mu_1\|_2^2}{2\sigma_1^2}\right) + p' \frac{1}{(\sqrt{2\pi} \sigma_2)^d} \exp\left(\frac{-\|x - \mu_1\|_2^2}{2\sigma_2^2}\right)}{(1-p') \frac{1}{(\sqrt{2\pi} \sigma_1)^d} \exp\left(\frac{-\|x - \mu_2\|_2^2}{2\sigma_1^2}\right) + p' \frac{1}{(\sqrt{2\pi} \sigma_2)^d} \exp\left(\frac{-\|x - \mu_2\|_2^2}{2\sigma_2^2}\right)}\right) \\
& = \left(\frac{-\|x - \mu_1\|_2^2}{2\sigma_1^2} + \frac{\|x - \mu_2\|_2^2}{2\sigma_1^2}\right) + \log\left(\frac{1+y}{1+z}\right),
\end{align*}
where
\begin{align*}
y & := \frac{p'}{1-p'} \left(\frac{\sigma_1}{\sigma_2}\right)^d \exp\left(\frac{-\|x - \mu_1\|_2^2}{2\sigma_2^2} + \frac{\|x - \mu_1\|_2^2}{2\sigma_1^2}\right), \\
z & := \frac{p'}{1-p'} \left(\frac{\sigma_1}{\sigma_2}\right)^d \exp\left(\frac{-\|x - \mu_2\|_2^2}{2\sigma_2^2} + \frac{\|x - \mu_2\|_2^2}{2\sigma_1^2}\right).
\end{align*}
Hence,
\begin{align*}
KL(\mprob_1, \mprob_2) & = \E_{x \sim \mprob_1}\left[\frac{-\|x - \mu_1\|_2^2}{2\sigma_1^2} + \frac{\|x - \mu_2\|_2^2}{2\sigma_1^2}\right] + \E_{x \sim \mprob_1}\left[\log\left(\frac{1+y}{1+z}\right)\right] \\
& \le \frac{\|\mu_1 - \mu_2\|^2}{2\sigma_1^2} + \E_{x \sim \mprob_1} [y] - \E_{x \sim \mprob_1} [z] + \E_{x \sim \mprob_1} [z^2],
\end{align*}
using the fact that
\begin{equation*}
\log\left(\frac{1+y}{1+z}\right) = \log\left(1 + \frac{y-z}{1+z}\right) \le \frac{y-z}{1+z} \le y - z + z^2,
\end{equation*}
since $y, z > 0$. We now write
\begin{align*}
\E_{x \sim \mprob_1} [y] & = \frac{p'}{1-p'} \left(\frac{\sigma_1}{\sigma_2}\right)^d \Bigg((1-p') \int \exp\left(\frac{-\|x - \mu_1\|_2^2}{2\sigma_2^2} + \frac{\|x - \mu_1\|_2^2}{2\sigma_1^2}\right) \frac{1}{(\sqrt{2\pi} \sigma_1)^d} \exp\left(\frac{-\|x - \mu_1\|_2^2}{2\sigma_1^2}\right) dx \\
& \qquad + p' \int \exp\left(\frac{-\|x - \mu_1\|_2^2}{2\sigma_2^2} + \frac{\|x - \mu_1\|_2^2}{2\sigma_1^2}\right) \frac{1}{(\sqrt{2\pi} \sigma_2)^d} \exp\left(\frac{-\|x - \mu_1\|_2^2}{2\sigma_2^2}\right) dx \Bigg) \\
& := A_y + B_y,
\end{align*}
and
\begin{align*}
\E_{x \sim \mprob_1} [z] & = \frac{p'}{1-p'} \left(\frac{\sigma_1}{\sigma_2}\right)^d \Bigg((1-p') \int \exp\left(\frac{-\|x - \mu_2\|_2^2}{2\sigma_2^2} + \frac{\|x - \mu_2\|_2^2}{2\sigma_1^2}\right) \frac{1}{(\sqrt{2\pi} \sigma_1)^d} \exp\left(\frac{-\|x - \mu_1\|_2^2}{2\sigma_1^2}\right) dx \\
& \qquad + p' \int \exp\left(\frac{-\|x - \mu_2\|_2^2}{2\sigma_2^2} + \frac{\|x - \mu_2\|_2^2}{2\sigma_1^2}\right) \frac{1}{(\sqrt{2\pi} \sigma_2)^d} \exp\left(\frac{-\|x - \mu_1\|_2^2}{2\sigma_2^2}\right) dx \Bigg) \\
& := A_z + B_z.
\end{align*}
Now, we may calculate
\begin{equation*}
A_y = p' \left(\frac{1}{\sqrt{2\pi} \sigma_2}\right)^d \int \exp\left(\frac{-\|x - \mu_1\|_2^2}{2\sigma_2^2} \right) dx = p',
\end{equation*}
and
\begin{align*}
B_y & = \frac{(p')^2}{1-p'} \left(\frac{\sigma_1}{\sqrt{2\pi} \sigma_2^2}\right)^d \int \exp\left(\frac{-\|x - \mu_1\|_2^2}{\sigma_2^2} + \frac{\|x - \mu_1\|_2^2}{2\sigma_1^2}\right) dx \\
& = \frac{(p')^2}{1-p'} \left(\frac{\sigma_1}{\sqrt{2\pi} \sigma_2^2}\right)^d \left(\frac{\pi}{\frac{1}{\sigma_2^2} - \frac{1}{2\sigma_1^2}}\right)^{d/2} \le \frac{(p')^2}{1-p'} \left(\frac{\sigma_1}{\sigma_2}\right)^d,
\end{align*}
using the fact that $\frac{1}{2\sigma_1^2} \le \frac{1}{2\sigma_2^2}$.

For ease of calculation, we now set
\begin{align}
\label{EqnMus}
\mu_1^T & = (\mu, 0, \dots, 0), \notag \\
\mu_2^T & = (-\mu, 0, \dots, 0).
\end{align}
Using the formula
\begin{equation}
\label{EqnGaussInteg}
\int \exp\left(-x^T A x + b^T x + c\right) dx = \sqrt{\frac{\pi^d}{\det(A)}} \exp\left(\frac{1}{4} b^T A^{-1} b + c\right),
\end{equation}
we have
\begin{align*}
A_z & = p' \left(\frac{1}{\sqrt{2\pi} \sigma_2}\right)^d \int \exp\left(\frac{-\|x - \mu_2\|_2^2}{2\sigma_2^2} + \frac{\|x - \mu_2\|_2^2}{2\sigma_1^2} - \frac{\|x - \mu_1\|_2^2}{2\sigma_1^2}\right) dx \\
& = p' \exp\left(\frac{\sigma_2^2}{2} \left(\left\|\frac{\mu_2}{\sigma_2^2} - \frac{\mu_2}{\sigma_1^2} + \frac{\mu_1}{\sigma_1^2}\right\|_2^2 - \frac{\mu_2^T \mu_2}{2\sigma_2^2} + \frac{\mu_2^T \mu_2}{2\sigma_1^2} - \frac{\mu_1^T \mu_1}{2\sigma_1^2}\right)\right) \\
& = p' \exp\left(-2\mu^2\left(\frac{1}{\sigma_1^2} - \frac{\sigma_2^2}{\sigma_1^4}\right)\right).
\end{align*}
In particular, using the fact that $\exp(-x) \ge 1-x$ for $x \ge 0$, we have
\begin{equation*}
A_y - A_z = p' - A_z \le p' \cdot 2\mu^2\left(\frac{1}{\sigma_1^2} - \frac{\sigma_2^2}{\sigma_1^4}\right) \le \frac{2\mu^2}{\sigma_1^2}.
\end{equation*}

Similarly, we can compute
\begin{align*}
B_z & = \frac{(p')^2}{1-p'} \left(\frac{\sigma_1}{\sqrt{2\pi} \sigma_2^2}\right)^d \int \exp\left(\frac{-\|x - \mu_2\|_2^2}{2\sigma_2^2} + \frac{\|x - \mu_2\|_2^2}{2\sigma_1^2} - \frac{\|x - \mu_1\|_2^2}{2\sigma_2^2}\right) dx \\
& = \frac{(p')^2}{1-p'} \left(\frac{\sigma_1}{\sqrt{2} \sigma_2^2\sqrt{\frac{1}{\sigma_2^2} - \frac{1}{2\sigma_1^2}}}\right)^d \exp\left(\frac{1}{4\left(\frac{1}{\sigma_2^2} - \frac{1}{2\sigma_1^2}\right)} \left\|\frac{\mu_2}{\sigma_2^2} - \frac{\mu_2}{\sigma_1^2} + \frac{\mu_1}{\sigma_2^2}\right\|_2^2 - \frac{\mu_2^T \mu_2}{2\sigma_2^2} + \frac{\mu_2^T \mu_2}{2\sigma_1^2} - \frac{\mu_1^T \mu_1}{2\sigma_2^2}\right) \\
& = \frac{(p')^2}{1-p'} \left(\frac{\sigma_1}{\sqrt{2} \sigma_2^2 \sqrt{\frac{1}{\sigma_2^2} - \frac{1}{2\sigma_1^2}}}\right)^d \exp\left(\frac{\mu^2/\sigma_1^4}{4\left(\frac{1}{\sigma_2^2} - \frac{1}{2\sigma_1^2}\right)} - \frac{\mu^2}{\sigma_2^2} + \frac{\mu^2}{2\sigma_1^2}\right) \\
& \stackrel{(a)}{\le} \frac{(p')^2}{1-p'} \left(\frac{\sigma_1}{\sigma_2}\right)^d \exp\left(\frac{\mu^2\sigma_2^2}{2\sigma_1^4} - \frac{\mu^2}{\sigma_2^2} + \frac{\mu^2}{2\sigma_1^2}\right) \\
& \le \frac{(p')^2}{1-p'} \left(\frac{\sigma_1}{\sigma_2}\right)^d \exp\left(\frac{\mu^2}{2\sigma_2^2} - \frac{\mu^2}{\sigma_2^2} + \frac{\mu^2}{2\sigma_1^2}\right) \\
& \le \frac{(p')^2}{1-p'}\left(\frac{\sigma_1}{\sigma_2}\right)^d,
\end{align*}
where inequality (a) uses the fact that $\frac{1}{2\sigma_1^2} \le \frac{1}{2\sigma_2^2}$.

Combining the inequalities and using the assumption~\eqref{EqnSratioD}, we conclude that
\begin{equation*}
\E_{x \sim \mprob_1} [y] - \E_{x \sim \mprob_1} [z] = O\left(\frac{\mu^2}{\sigma_1^2}\right) + O\left(\frac{1}{n}\right).
\end{equation*}

Finally, we compute
\begin{align*}
\E_{x \sim \mprob_1} [z^2] & = \left(\frac{p'}{1-p'}\right)^2 \left(\frac{\sigma_1}{\sigma_2}\right)^{2d} \Bigg((1-p') \int \exp\left(\frac{-\|x - \mu_2\|_2^2}{\sigma_2^2} + \frac{\|x - \mu_2\|_2^2}{\sigma_1^2}\right) \\
& \qquad \qquad \cdot \frac{1}{(\sqrt{2\pi} \sigma_1)^d} \exp\left(\frac{-\|x - \mu_1\|_2^2}{2\sigma_1^2}\right) dx \\
& \qquad + p' \int \exp\left(\frac{-\|x - \mu_2\|_2^2}{\sigma_2^2} + \frac{\|x - \mu_2\|_2^2}{\sigma_1^2}\right) \frac{1}{(\sqrt{2\pi} \sigma_2)^d} \exp\left(\frac{-\|x - \mu_1\|_2^2}{2\sigma_2^2}\right) dx \Bigg) \\
& := A'_z + B'_z.
\end{align*}
Again using the designation~\eqref{EqnMus} and the formula~\eqref{EqnGaussInteg}, we have
\begin{align*}
A'_z & = \frac{(p')^2}{1-p'} \left(\frac{\sigma_1}{\sqrt{2\pi} \sigma_2^2}\right)^d \int \exp\left(\frac{-\|x - \mu_2\|_2^2}{\sigma_2^2} + \frac{\|x - \mu_2\|_2^2}{\sigma_1^2} - \frac{\|x - \mu_1\|_2^2}{2\sigma_1^2}\right) dx \\
& = \frac{(p')^2}{1-p'} \left(\frac{\sigma_1}{\sqrt{2} \sigma_2^2\sqrt{\frac{1}{\sigma_2^2} - \frac{1}{2\sigma_1^2}}}\right)^d \exp\left(\frac{1}{4 \left(\frac{1}{\sigma_2^2} - \frac{1}{2\sigma_1^2}\right)} \left\|\frac{2\mu_2}{\sigma_2^2} - \frac{2\mu_2}{\sigma_1^2} + \frac{\mu_1}{\sigma_1^2}\right\|_2^2 - \frac{\mu_2^T \mu_2}{\sigma_2^2} + \frac{\mu_2^T \mu_2}{\sigma_1^2} - \frac{\mu_1^T \mu_1}{2\sigma_1^2}\right) \\
& = \frac{(p')^2}{1-p'} \left(\frac{\sigma_1}{\sqrt{2} \sigma_2^2\sqrt{\frac{1}{\sigma_2^2} - \frac{1}{2\sigma_1^2}}}\right)^d \exp\left(\frac{(-2\mu/\sigma_2^2 + 3\mu/\sigma_1^2)^2}{4\left(\frac{1}{\sigma_2^2} - \frac{1}{2\sigma_1^2}\right)} - \frac{\mu^2}{\sigma_2^2} + \frac{\mu^2}{2\sigma_1^2}\right) \\
& \le \frac{(p')^2}{1-p'} \left(\frac{\sigma_1}{\sigma_2}\right)^d \exp\left(\frac{(-2\mu/\sigma_2^2 + 3\mu/\sigma_1^2)^2}{4\left(\frac{1}{\sigma_2^2} - \frac{1}{2\sigma_1^2}\right)} - \frac{\mu^2}{\sigma_2^2} + \frac{\mu^2}{2\sigma_1^2}\right), \\
\end{align*}
and
\begin{align*}
B'_z & = \frac{(p')^3}{1-p'} \left(\frac{\sigma_1^2}{\sqrt{2\pi}\sigma_2^3}\right)^d \int \exp\left(\frac{-\|x - \mu_2\|_2^2}{\sigma_2^2} + \frac{\|x - \mu_2\|_2^2}{\sigma_1^2} - \frac{\|x - \mu_1\|_2^2}{2\sigma_2^2}\right) dx \\
& = \frac{(p')^3}{1-p'} \left(\frac{\sigma_1^2}{\sqrt{2} \sigma_2^3 \sqrt{\frac{3}{2\sigma_2^2} - \frac{1}{\sigma_1^2}}}\right)^d \exp\left(\frac{1}{4\left(\frac{3}{2\sigma_2^2} - \frac{1}{\sigma_1^2}\right)} \left\|\frac{2\mu_2}{\sigma_2^2} - \frac{2\mu_2}{\sigma_1^2} + \frac{\mu_1}{\sigma_2^2}\right\|_2^2 - \frac{\mu_2^T \mu_2}{\sigma_2^2} + \frac{\mu_2^T \mu_2}{\sigma_1^2} - \frac{\mu_1^T \mu_1}{2\sigma_2^2}\right) \\
& = \frac{(p')^3}{1-p'} \left(\frac{\sigma_1^2}{\sqrt{2} \sigma_2^3 \sqrt{\frac{3}{2\sigma_2^2} - \frac{1}{\sigma_1^2}}}\right)^d \exp\left(\frac{(-\mu/\sigma_2^2 + 2\mu/\sigma_1^2)^2}{4\left(\frac{3}{2\sigma_2^2} - \frac{1}{\sigma_1^2}\right)} - \frac{3\mu^2}{2\sigma_2^2} + \frac{\mu^2}{\sigma_1^2}\right) \\
& \le \frac{(p')^3}{1-p'} \left(\frac{\sigma_1}{\sigma_2}\right)^{2d} \exp\left(\frac{(-\mu/\sigma_2^2 + 2\mu/\sigma_1^2)^2}{4\left(\frac{3}{2\sigma_2^2} - \frac{1}{\sigma_1^2}\right)} - \frac{3\mu^2}{2\sigma_2^2} + \frac{\mu^2}{\sigma_1^2}\right).
\end{align*}
Considering the exponential terms in the expressions for $A'_z$ and $B'_z$, note that for $A'_z$, we have
\begin{equation*}
\frac{(-2\mu/\sigma_2^2 + 3\mu/\sigma_1^2)^2}{4\left(\frac{1}{\sigma_2^2} - \frac{1}{2\sigma_1^2}\right)} - \frac{\mu^2}{\sigma_2^2} = \frac{\mu^2}{\sigma_2^2}\left(\frac{\left(2 - \frac{3\sigma_2^2}{\sigma_1^2}\right)^2}{4\left(1 - \frac{\sigma_2^2}{2\sigma_1^2}\right)} - 1\right) < 0,
\end{equation*}
assuming $\sigma_2 \le \sigma_1$,
whereas for $B'_z$, we have
\begin{equation*}
\frac{(-\mu/\sigma_2^2 + 2\mu/\sigma_1^2)^2}{4\left(\frac{3}{2\sigma_2^2} - \frac{1}{\sigma_1^2}\right)} - \frac{3\mu^2}{2\sigma_2^2} = \frac{\mu^2}{\sigma_2^2} \left(\frac{\left(1 - \frac{2\sigma_2^2}{\sigma_1^2}\right)^2}{4\left(\frac{3}{2} - \frac{\sigma_2^2}{\sigma_1^2}\right)} - \frac{3}{2}\right) < 0,
\end{equation*}
using the fact that $\sigma_2 \le \sigma_1$. Thus, using the assumption~\eqref{EqnSratioD}, we obtain
\begin{align*}
\E_{x \sim \mprob_1} [z^2] & = A'_z + B'_z = O\left(\frac{1}{n}\right) \exp\left(\frac{\mu^2}{2\sigma_1^2}\right) + O\left(\frac{1}{n^2 p}\right) \exp\left(\frac{\mu^2}{\sigma_1^2}\right) \\
& = O\left(\frac{1}{n}\right) \exp\left(\frac{\mu^2}{\sigma_1^2}\right).
\end{align*}
Finally, we take $\mu = \frac{\sigma_1}{\sqrt{n}}$ to obtain the desired bound~\eqref{EqnKLUpper}. This completes the proof.

\subsection{Proof of Theorem~\ref{ThmUpperBound}}
\label{AppThmUpperBound}

By a similar argument used to derive the bound in Theorem~\ref{ThmExpBound}, the following expected error bound may be derived from the high-probability bound in Theorem~\ref{thm:HybridScreening} for the hybrid estimator:
\begin{equation}
\label{EqnHybExp}
\E\|\muhat_{k_1, k_2}\|_2
\leq \min\left\{\sqrt{d} r_{2k_1, 1}, \sqrt{n}^{1/d} r_{k_2}\right\}.
\end{equation}
In what follows, we will bound these expressions to obtain the desired results.

As shown in the proof of Lemma~\ref{lemma:prop2}\ref{prop2:5}, a ball of radius $C\sigma_2 \sqrt{d} $ around the origin will contain at least $\frac{1}{2}$ of the mass of $np$ distributions.
Thus, if $np \ge 2k_2$, we will have $r_{k_2} \le C\sigma_2 \sqrt{d}$.
 
We now claim that $r_{2k_1, 1} \le \frac{C\sigma_1 \log n}{ \sqrt{n}} := r'$, which we will show by integrating the marginal densities on the interval $[-r',r']$. Note that $\nu_i \leq  \sigma_1$ for all $i$.
We consider two cases: if $\nu_i \geq r'$, then $q_i(r') \geq \frac{c}{\nu_i} \geq \frac{c}{\sigma_1}$, using inequality~\eqref{EqnRegularityMarginal}, so $\int_{[-r',r']}q_i(x)dx \geq \frac{2c r'}{\sigma_1} \geq \frac{2\log n}{\sqrt{n}}$ for large enough $C$. If $\nu_i < r'$, then $\int_{[-\nu_i,\nu_i]} q(x)dx \geq c' \geq \frac{2\log n}{\sqrt{n}}$, as well.
Thus,
\begin{align}
\label{EqnQbd}
 \sum_{i=1}^n \int_{[-r',r']}q_i(x)dx \geq \sum_{i=1}^n \frac{2\log n}{\sqrt{n}} \geq 2\sqrt{n} \log n = 2k_1.
 \end{align}
Combining the results with inequality~\eqref{EqnHybExp} proves inequality~\eqref{EqnUB1}.

We now consider the special cases:
\begin{itemize}
\item[(a)] In the case when $p = \Omega\left(\frac{\sqrt{n} \log n}{n}\right)$, we can use fact that at least $np = \Omega(\sqrt{n}\log n)$ points have marginal variance at most $\sigma_2$. Let $r' \coloneqq \frac{C\sigma_2 \log n}{ p\sqrt{n}}$. By similar reasoning as above, for at least $np$ distributions, we have $\int_{[-r',r']}q_i(x)dx \geq \frac{\log n}{p\sqrt{n}}$. Thus, we can replace inequality~\eqref{EqnQbd} by
\begin{align*}
\sum_{i=1}^n \int_{[-r',r']}q_i(x)dx \geq np \cdot \frac{2\log n}{p \sqrt{n}} \geq 2\sqrt{n}\log n,
\end{align*}
to conclude that $r_{2k_1, 1} = \O\left(\frac{\sigma_2 \log n}{p\sqrt{n}}\right)$. This leads to the stated bound.
\item[(b)]
In this case, we will obtain a better bound by showing that $\| \muhat_{S,k_2}\|_2 \leq  r_{2k_2}$, w.h.p., rather than the looser bound $\|\muhat_{S,k_2}\|_2 \le C' \sqrt{n}^{1/d} r_{k_2}$ used to derive
inequality~\eqref{EqnHybExp} (cf.\ Theorem~\ref{thm:HybridScreening}). Since $r_{2k_2} \le C \sigma_2 \sqrt{d}$, the tighter bound will then follow.

Let $r'\coloneqq C' \sqrt{d \log n } \sigma_2$. As argued in the proof of Theorem~\ref{ThmShorthD}, it suffices to show that $R(f_{r',  r_{2k} }) \leq \frac{k}{2n}$, where $k = k_2$.  We will deal with low-variance and high-variance points separately.

First, consider $i$ such that $\nu_i = \Omega( \sigma_1) = \Omega( \sigma_2 n^{\frac{1}{d}})   \geq C'' \sigma_2 n ^{\frac{1}{d}} $ for large $C''$, and let $v_d$ denote the volume of the ball of radius $1$. Then
\begin{align*}
\P\(X_i \in B(r',r_{2k})\) \leq \P\(X_i \in B(0, r_{2k})\) \leq f_i(0) v_d r_{2k}^d \leq \(\frac{c'}{C''\sigma_2 n^{1/d}}\)^d v_d \sigma_2^d C^d \sqrt{d}^d \leq \frac{1}{n},
\end{align*}
where we use condition~\eqref{EqnRegularityJoint} and the fact that $ \frac{v_d \sqrt{d}^d}{\tilde{C}^d} \leq 1 $ for a sufficiently large constant $\tilde{C}$.

Now consider $i$ such that $ \nu_i \leq \sigma_2$. By condition~\eqref{EqnRegularityJoint}, we have
\begin{align*}
\P\(X_i \in B(r',r_{2k})\)  \leq \exp(-c_1 \log n) \leq \frac{1}{n^{c_1}}.
\end{align*}
For large enough $C'$, we can ensure that $c_1 \geq 1$.
Altogether, we conclude that
\begin{align*}
R(f_{r',r_{2k}}) &= \frac{1}{n} \sum_{i=1}^n \P\(X_i \in B(r',r_{2k})\) \leq \frac{1}{n} < \frac{k_2}{2n},
\end{align*}
which concludes the proof.
\end{itemize}

\section{Proofs for alternative conditions}

In this appendix, we prove the statements of the results in Section~\ref{SecAltCond}.

\subsection{Proof of Theorem~\ref{ThmRelax}}
\label{AppThmRelax}

We first prove claim (i). Note that the result of Lemma~\ref{thm:modalIntervalMain} will still hold, since it only depends on the uniform concentration bound and optimality of the modal interval estimator. Thus, $R(f_{\muhat_{M,r}, r}) \ge \frac{R^*_r}{2}$, w.h.p.

For a fixed value of $r'$, define $\muhat' = \frac{\muhat_{M,r}}{\|\muhat_{M,r}\|_2} \cdot r'$ to be the rescaled version of $\muhat_{M,r}$. By condition (C1), we will have $\|\muhat_{M,r}\|_2 \le r'$ if we can show that $R(f_{\muhat', r}) \le R(f_{\muhat_{M,r}, r})$. Note that
\begin{equation*}
R(f_{\muhat', r}) \le g(r',r),
\end{equation*}
so if we choose $r'$ sufficiently large so that $g(r',r) < \frac{R^*_r}{2}$, the desired inequality will hold.

Turning to claim (ii), note that Lemma~\ref{lemma:shorthlength} continues to hold, since it only relies on the uniform concentration bound and a Chernoff bound. We thus conclude that $R(f_{\muhat_{S,k}, r_{2k}}) \ge \frac{k}{4n} = \frac{R^*_{2k}}{4}$, w.h.p. For a fixed value of $r'$, we define $\muhat' = \frac{\muhat_{M,r_{2k}}}{\|\muhat_{M,r_{2k}}\|_2} \cdot r'$. By condition (C1) (which we only need to assume holds for $r = r_{2k}$), if $R(f_{\muhat', r_{2k}}) \le R(f_{\muhat_{M,r_{2k}}, r_{2k}})$, then $\|\muhat_{M,r_{2k}}\|_2 \le r'$. Furthermore, $R(f_{\muhat', r_{2k}}) \le g(r',r_{2k})$, so we simply need to choose $r'$ such that $g(r', r_{2k}) < \frac{1}{4}$.

For the hybrid estimator, note that Lemma~\ref{LemCubOrigin} shows that the output is always within $\sqrt{d}r_{4\sqrt{n\log n},1}$ of the output. Furthermore, the output of shorth estimator is always with $r'$ of the origin by part (ii). If the shorth estimator lies outside the $S_{\sqrt{n\log n}}^\infty$, then its $\ell_2$ projection on $S_{\sqrt{n\log n}}^\infty$ will only decrease its distance from the origin because (1) the origin belongs to $S_{\sqrt{n\log n}}^\infty$; and (2) $S_{\sqrt{n\log n}}^\infty$ is convex.

\subsection{Proof of Proposition~\ref{LemRmode}}
\label{AppLemRmode}

We first show that for each $r > 0$, the functions $R_i(f_{x,r})$ are unimodal as functions of $x \in \real^d$. Let $q$ be the uniform distribution on the Euclidean ball of radius $r$. Then $p_i \star q$, being a convolution of two log-concave densities, is also log-concave. Log-concave densities by definition are proportional to $e^{-\phi(x)}$ for some convex function $\phi$, and therefore they are unimodal and monotonically decreasing along rays from the mode. Now note that if condition (C3) holds, then $R_i(f_{x,r})$ must also be symmetric around 0. Hence, if $R_i(f_{x,r})$ is unimodal, its unique mode must clearly occur at 0. This proves that conditions (C2) and (C3) together imply condition (C1).

For the second statement, it suffices to verify the inequality
\begin{equation}
\label{EqnRmode}
\sup_{\|x\|_2 = a} R_i(f_{x,r}) \le \frac{1}{\lfloor a/2r \rfloor}, \qquad \forall i.
\end{equation}
Indeed, we would then have
\begin{equation*}
g(a,r) = \sup_{\|x\|_2 = a} \frac{1}{n} \sum_{i=1}^n R_i(f_{x,r}) \le \frac{1}{n} \sum_{i=1}^n \sup_{\|x\|_2 = a} R_i(f_{x,r}) \le \frac{1}{\lfloor a/2r \rfloor}.
\end{equation*}

Thus, it remains to verify inequality~\eqref{EqnRmode}. Focusing on a particular $i$, consider $x \in \real^d$ such that $\|x\|_2 = a$. We know that $R_i(f_{x,r})$ is decreasing on the ray from $0$ to $x$. Furthermore, we can pack $\lfloor\frac{a}{2r}\rfloor$ balls of radius $r$ on the ray, including the balls $B(x_i^*,r)$ and $B(x,r)$ at the endpoints. The total mass of these balls is clearly upper-bounded by 1. Hence,
\begin{equation*}
\left\lfloor \frac{a}{2r} \right\rfloor \cdot R_i(f_{x,r}) \le 1,
\end{equation*}
implying the desired result.

\subsection{Proof of Proposition~\ref{PropElliptical}}
\label{AppPropElliptical}

Let $X$ have an elliptically symmetric density defined as $p_X(x) = f(x^T \Sigma^{-1} x)$ for a decreasing function $f: \real \to \real$. Consider a point $x_0 \in \real^d$ such that $\|x_0\|_2  = r_2$, and consider the ball $B(x_0, r_1) = \{ x \in \real^d ~:~ \|x-x_0\| \leq r_1 \}$.  For analysis purposes, we first transform the elliptically symmetric density to a spherically symmetric, decreasing density. This may be achieved by applying the linear transformation $\Sigma^{-1/2}: \real^d \to \real^d$. Define $Y := \Sigma^{-1/2} X$, let $\Sigma^{-1/2} x_0 = y_0$, and let $\hat B$ be the image of $B(x_0, r_1)$ under the transformation $\Sigma^{-1/2}$. Note that
\begin{equation*}
\hat B = \left\{y \in \real^d ~:~ (y-y_0)^T \Sigma (y-y_0) \leq r_1\right\},
\end{equation*}
and further note that $R(f_{x_0, r_1})$ is equal to the integral of $p_Y(\cdot)$ over $\hat B$; i.e., $\P(Y \in \hat B)$. It is easy to see that $\hat B \subseteq B\left(y_0, \frac{r_1}{\lambda_{\min}(\Sigma)}\right)$. Hence,
\begin{align*}
    R(f_{x_0, r_1}) = \P(Y \in \hat B)
    \leq \P(Y \in B\left(y_0, \frac{r_1}{\lambda_{\min}(\Sigma)}\right).
\end{align*}
We may now use the strategy from Lemma~\ref{lemma:prop2}, to obtain
\begin{align*}
    1 &\geq \P(Y \in B(0, \|y_0\|_2))\\
    &\geq P\left(B(0, \|y_0\|_2), \frac{r_1}{\lambda_{\min}(\Sigma)}\right) \cdot \P\left(Y \in B\left(y_0, \frac{r_1}{\lambda_{\min}(\Sigma)}\right)\right)\\
    &\geq P\left(B\left(0, \frac{r_2}{\lambda_{\max}(\Sigma)}\right), \frac{r_1}{\lambda_{\min}(\Sigma)}\right)\cdot R(f_{x_0, r_1}).
\end{align*}
Since this inequality holds for any $x_2$ with $\| x_2 \|_2 = r_2$, we conclude that
\begin{align*}
    g(r_2, r_1) &\leq \frac{1}{P\left(B\left(0, \frac{r_2}{\lambda_{\max}(\Sigma)}\right), \frac{r_1}{\lambda_{\min}(\Sigma)}\right)}\\
    &\leq C \left(\frac{r_1 \lambda_{\max}(\Sigma)}{r_2 \lambda_{\min}(\Sigma)}\right)^d.
\end{align*}

\subsection{Proof of Proposition~\ref{PropMixture}}
\label{AppPropMixture}

We index the distributions so that $\{R_i\}_{i=1}^s$ are radially symmetric. Note that
\begin{equation*}
g(a,r) = \sup_{\|x\|_2 = a} R(f_{x,r}) \le \frac{1}{n} \sum_{i=1}^n \sup_{\|x\|_2 = a} R_i(f_{x,r}).
\end{equation*}
Furthermore, for each $1 \le i \le s$, we have
\begin{equation*}
\sup_{\|x\|_2 = a} R_i(f_{x,r}) \le \left(\frac{r}{a}\right)^d R_i(f_{0,a}) \le \left(\frac{r}{a}\right)^d.
\end{equation*}
On the other hand, for $i > s$, we have
\begin{equation*}
\sup_{\|x\|_2 = a} R_i(f_{x,r}) \le \frac{r}{a}.
\end{equation*}
Hence,
\begin{equation*}
g(a,r) \le \frac{s}{n} \left(\frac{r}{a}\right)^d + \frac{n-s}{n} \left(\frac{r}{a}\right).
\end{equation*}

Now note that $R^*_{q_{(f(n))}} \ge \frac{f(n)}{2n}$. Thus,
\begin{align*}
g(r',r) & \le \frac{s}{n} \cdot \frac{1}{2^dn} + \frac{n-s}{n} \cdot \frac{1}{2n^{1/d}}
\le \frac{1}{n} + \frac{n-s}{n} \cdot \frac{1}{2n^{1/d}} < \frac{f(n)}{4n} \le \frac{R^*_r}{2},
\end{align*}
using the assumed lower bound on $s$.

\section{Proofs for regression}

In this appendix, we provide the proofs of the statements in Section~\ref{SecRegression}.

\subsection{Proof of Proposition~\ref{PropRegMax}}
\label{AppPropRegMax}

We write
\begin{align*}
\sum_{i=1}^n \E\left[1\left\{|y_i - x_i^T \beta| \le r\right\}\right] & = \sum_{i=1}^n \mprob\left(|y_i - x_i^T \beta| \le r\right) \\
& = \sum_{i=1}^n \mprob\left(|x_i^T (\betastar - \beta) + \epsilon_i| \le r\right).
\end{align*}
Note that conditioned on $x_i$, each summand is maximized uniquely when $x_i^T (\betastar - \beta) = 0$, since the distribution of $\epsilon_i$ is symmetric and unimodal. Since
\begin{equation}
\label{EqnCondExp}
\sum_{i=1}^n \E\left[1 \left\{|y_i - x_i^T \beta| \le r\right\}\right] = \E\left[\sum_{i=1}^n \E\left[1 \left\{|y_i - x_i^T \beta| \le r \right\} \mid \{x_i\}_{i=1}^n\right]\right],
\end{equation}
we see that the right-hand expression in equation~\eqref{EqnCondExp} is therefore maximized when $\beta = \betastar$. On the other hand, we can also argue that the maximizer is unique. Indeed, suppose $\beta \in \real^d$ were such that $\beta \neq \betastar$. The set $\cS := \left\{\{x_i\}_{i=1}^n \subseteq (\real^d)^n: x_i^T (\beta - \betahat) = 0 \quad \forall i\right\}$ has Lebesgue measure 0. We can write
\begin{align*}
\E\left[\sum_{i=1}^n \E\left[1 \left\{|y_i - x_i^T \beta| \le r \right\} \mid \{x_i\}_{i=1}^n\right]\right] & = \int_{\{x_i\} \in \cS} \E\left[1\left\{|y_i - x_i^T \beta| \le r\right\} \mid \{x_i\}_{i=1}^n\right] d\mprob(\{x_i\}) \\
& \qquad + \int_{\{x_i\} \notin \cS} \E\left[1\left\{|y_i - x_i^T \beta| \le r\right\} \mid \{x_i\}_{i=1}^n\right] d\mprob(\{x_i\}).
\end{align*}
Noting that
\begin{align*}
\E\left[1\left\{|y_i - x_i^T \beta| \le r\right\} \mid \{x_i\}_{i=1}^n\right] & = \E\left[1\left\{|y_i - x_i^T \betastar| \le r\right\} \mid \{x_i\}_{i=1}^n\right], \quad \forall \{x_i\} \in \cS, \\
\E\left[1\left\{|y_i - x_i^T \beta| \le r\right\} \mid \{x_i\}_{i=1}^n\right] & < \E\left[1\left\{|y_i - x_i^T \betastar| \le r\right\} \mid \{x_i\}_{i=1}^n\right], \quad \forall \{x_i\} \notin \cS,
\end{align*}
completes the proof.

\subsection{Proof of Theorem~\ref{ThmRegression}}
\label{AppThmRegression}

The proof follows the same approach used to prove estimation error bounds for the modal interval estimator throughout the paper (e.g., Theorem~\ref{cor:modal}). By Lemma~\ref{LemRegConc}, we know that $R_{\betahat} \ge \frac{R_{\betastar}}{2}$, w.h.p. We will be done if we can show that $R_\beta < \frac{R_{\betastar}}{2}$ for all $\beta$ satisfying
\begin{equation}
\label{EqnBetaRange}
\|\beta - \beta\|_2 > \frac{c'n \sigma_{(cd\log n)}}{\lambda_{\min}}.
\end{equation}

First note that
\begin{equation*}
R_{\betastar} = \frac{1}{n} \sum_{i=1}^n \mprob(|\epsilon_i| \le r).
\end{equation*}
Hence, as argued for mean estimation, we certainly have $r \le C'\sigma_{(Cd\log n)}$.

Also note that for any $\beta \in \real^d$, we have
\begin{equation*}
y_i - x_i^T \beta = \epsilon_i + x_i^T (\betastar - \beta) \sim N\left((\betastar - \beta)^T \mu_i',  (\betastar - \beta)^T \Sigma_i' (\betastar - \beta)\right).
\end{equation*}
Let $\cJ$ denote the set of indices of the smallest $d\log n$ of the $\sigma_i$'s. Note that
\begin{equation*}
R_{\betastar} \ge \frac{1}{n} \sum_{i \in \cJ} \mprob(|\epsilon_i| \le r) \ge 2r \cdot \frac{c}{n} \sum_{i=1}^{d\log n} \frac{1}{\sqrt{2\pi} \sigma_{(i)}},
\end{equation*}
since the Gaussian pdf decreases by a factor of $\approx 68\%$ within one standard deviation of 0.

Now suppose $\beta \in \real^d$ satisfies inequality~\eqref{EqnBetaRange}. We have
\begin{equation*}
R_\beta \le \frac{1}{n} \sum_{i=1}^n \mprob\left(|z_i| \le r\right),
\end{equation*}
where $z_i \sim N\left(0, \sigma_i^2 + (\betastar - \beta)^T \Sigma_i' (\betastar - \beta)\right)$.
For $i \notin \cJ$, we write
\begin{equation*}
\mprob(|z_i| \le r) \le 2r \cdot \frac{1}{\sqrt{2\pi} \sqrt{\sigma_i^2 + (\betastar - \beta)^T \Sigma_i' (\betastar - \beta)}} \le \frac{2r}{n \sigma_{(d\log n)} \sqrt{2\pi}},
\end{equation*}
since by the choice of $\beta$, we have
\begin{equation*}
(\betastar - \beta)^T \Sigma_i' (\betastar - \beta) \ge \lambda_{\min} \|\beta - \betastar\|_2^2 \ge n^2 \sigma^2_{(d\log n)}.
\end{equation*}
For $i \in \cJ$, we write
\begin{equation*}
\mprob(|z_i| \le r) \le 2r \cdot \frac{1}{\sqrt{2\pi} \sqrt{\sigma_i^2 + (\betastar - \beta)^T \Sigma_i' (\betastar - \beta)}} \le \frac{2r}{3\sigma_{i}^2 \sqrt{2\pi}},
\end{equation*}
since by the choice of $\beta$, we have
\begin{equation*}
(\betastar - \beta)^T \Sigma_i' (\betastar - \beta) \ge 2\sigma_{(d\log n)}^2 \ge 2 \sigma_i^2.
\end{equation*}
Thus, we conclude that
\begin{equation*}
R_\beta \le \frac{2r}{\sqrt{2\pi}} \cdot \frac{1}{n} \left(\sum_{i \in \cJ} \frac{1}{3\sigma_i^2} + \sum_{i \notin \cJ} \frac{1}{n\sigma_{(d\log n)}}\right) \le \frac{R_{\betastar}}{3} + \frac{c'}{n} < \frac{R_{\betastar}}{2},
\end{equation*}
as wanted. This concludes the proof.

\subsection{Proof of Theorem~\ref{ThmCompReg}}
\label{AppThmCompReg}

For $i \in [n]$, consider the sets 
$$U_i := \{\beta \subseteq \real^d : -r \leq x_i^T \beta \leq +r\}.$$
The set $U_i$ is sandwiched between the two hyperplanes $x_i^T\beta = y_i-r$ and $x_i^T\beta = y_i+r$. Denote these hyperplanes by $H_-(U_i)$ and $H_+(U_i)$, respectively. These $2n$ hyperplanes partition $\real^d$ into a finite number of (possibly unbounded) convex regions, which we denote by $\{R_1, \dots, R_M\}$. Define the function $f(\beta) := \sum_{i=1}^n \mathbbm{1}_{U_i}(\beta)$. Our goal is to find $\hat \beta = \argmax_{\beta \in \real^d} f(\beta)$,
where $\mathbbm{1}_{U_i}$ is the indicator function of $U_i$. It is easy to see that $f(\cdot)$ is constant when restricted to the interior of any fixed region $R_j$ for $j \in [M]$. Also, since $\mathbbm{1}_{U_i}$ is an upper-semicontinuous function for each $i \in [n]$, so is $f$. Thus, the value of $f(\cdot)$ at the vertices $R_j$ is at least as large as the value of $f$ in its interior. Thus, to find the maximum of $f(\cdot)$, we may only consider $\beta \in \real^d$ that correspond to vertices of $R_j$ for $j \in [M]$. All such vertices may be obtained by choosing any $d$ (mutually non-parallel) hyperplanes from among $\{H_-(U_1), \dots, H_-(U_n), H_+(U_1), \dots, H_+(U_M)\}$ and considering their point of intersection. The total number of such points is bounded above by ${2n \choose d}$, and our algorithm may simply list such points and evaluate $f$ at each point in the list.

\section{Auxiliary results}
\label{AppAux}

This appendix contains several technical results invoked throughout the paper.

We will employ the following multiplicative Chernoff bound, which is standard (cf.\ Vershynin~\cite{Ver18} or Boucheron et al.~\cite{BouEtAl16}):
\begin{lemma}
\label{LemChernoff}
Let $X_1, \dots, X_n$ be independent Bernoulli random variables with parameters $\{p_i\}$. Let $S_n = \sum_{i=1}^n X_i$ and $\mu = \E[S_n]$.
\begin{itemize}
\item[(i)] For any $\delta \in (0,1]$, we have
\begin{equation*}
\P\left(S_n  \ge (1 +  \delta) \mu\right) \le \exp\(- \frac{\mu \delta^2}{3}\).
\end{equation*}
and
\begin{equation*}
\P\left(S_n \le  (1 - \delta) \mu\right) \le  \exp\(- \frac{\mu \delta^2}{2}\).
\end{equation*}
\item[(ii)] For $\delta \geq 4 $, we have
\begin{equation*}
\P\left(S_n  \ge  \delta \mu\right) \le \exp\(- c \mu \delta \log \delta\).
\end{equation*}
\end{itemize}
\end{lemma}

We will also use the following result from Boucheron et al.~\cite{BouEtAl16}:

\begin{lemma}(Theorem 12.9 from Boucheron et al.~\cite{BouEtAl16})
\label{BouEtal12.9}
Let $W_1, \hdots, W_n$ be independent  vector-valued random variables
and let $Z = \sup_{s\in \cT} \sum_{i=1}^n W_{i,s}$. Assume that for all $i \leq n$ and $s \in \cT$, we have
$\E W_{i,s} =
0$, and $|W_{i,s}| \leq 1$. Let
\begin{align*}
v & := 2\E Z+ \rho^2, \\
\rho^2 & := \sup_{t\in T} \sum_{i=1}^n \E W^2_{i,s}.
\end{align*}
Then $\Var(Z) \leq v$ and 
\begin{align*}
\P\{Z \ge \E Z + t\} \leq \exp\(- \frac{t}{4} \log\(1 + 2 \log\(1 + \frac{t}{v}\) \)\).
\end{align*}
\label{thm:Bou12.9}
\end{lemma}

We now state and prove a generalization of Theorem 13.7 from Boucheron et al.~\cite{BouEtAl16}:

\begin{theorem}
Let $\cA = \{A_t  : t  \in \cT\}$ be a  countable class of measurable subsets of $\cX$ with VC dimension $V$, such that $A_0 = \emptyset \in \cA$. Let $X_1 , \hdots, X_n $ be independent random variables taking values in $\cX$, with distributions $P_1, \hdots, P_n$, respectively. Assume that for some $\sigma > 0$, we have
\begin{align*}
\frac{1}{n} \sum_{i=1}^n P_i(A_t) \leq \sigma^2, \text{  for every } t \in \cT.
\end{align*}
Let $Z$ and $Z^{-}$ be defined as follows: 
\begin{align*}
Z & = \frac{1}{\sqrt{n}} \sup_{t \in \cT} \sum_{i=1}^n \(\1_{X_i \in A_t} - P_i(A_t)\), \quad \text{ and } \\
Z^{-} & = \frac{1}{\sqrt{n}} \sup_{t \in \cT} \sum_{i=1}^n \(P_i(A_t) - \1_{X_i \in A_t} \).
\end{align*}
If $\sigma \geq 24 \sqrt{ \frac{V}{5n} \log \(\frac{4e^2}{\sigma}\)}$, then
\begin{align*}
\max\(\E Z,\E Z^{-}\) \leq 72 \sigma \sqrt{V \log \frac{4e^2}{\sigma}}.
\end{align*}
\label{thm:vcConvExpTech}
\end{theorem}

\begin{proof}
The following proof is an adaptation of the proof of Theorem 13.7 in Boucheron et al.~\cite{BouEtAl16}. The generalization from identical to non-identical distributions is possible because (1) independence suffices for symmetrization inequality; and (2) after conditioning on $X_1, \hdots, X_n$, it is no longer relevant whether the distributions of the random variables are identical. We include the initial steps of the proof for completeness and direct the reader to Boucheron et al.~\cite{BouEtAl16} for more details.

By the symmetrization inequalities of Lemma 11.4 in Boucheron et al.~\cite{BouEtAl16}, we have
\begin{align}
\label{EqnSym}
&\E \frac{1}{\sqrt{n}} \sup_{t \in \cT} \sum_{i=1}^n\( \1_{X_i \in A_t} - P(A_t)\) \notag \\
 &\qquad\leq 2\E\left[\E\left[ \frac{1}{\sqrt{n}}\sup_{t \in \cT} \sum_{i=1}^n \epsilon_i \1_{X_i \in A_t} \bigg| X_1, \hdots, X_n\right] \right],
\end{align}
where the $\epsilon_i$'s are independent Rademacher variables. Define the random variable 
\begin{align*}
\delta^2_n = \max\(\sup_{t \in \cT} \frac{1}{n} \sum_{i=1}^n \1_{X_i \in A_t}, \; \sigma^2\).
\end{align*}
Clearly, $\delta_n^2 \leq \frac{Z}{\sqrt{n}} + \sigma^2$, so by Jensen's inequality,\footnote{Note that both $Z$ and $Z^-$ are non-negative since $\phi \in \cA$.}
\begin{equation*}
\E \delta_n \leq \sqrt{\E \left(\frac{Z}{\sqrt{n}}\right) + \sigma^2}.
\end{equation*}

Now let $Z_t = \frac{1}{\sqrt{n}} \sum_{i=1}^n \epsilon_i \1_{X_i \in A_t}$. Noting that the Rademacher averages are sub-Gaussian, conditioned on the $X_i$'s, we have
\begin{align*}
\log & \E\left[e^{\lambda(Z_t - Z_{t'})} \Big | X_1, \dots, X_n\right] \\
& \leq \frac{\lambda^2 \(\frac{1}{n}\sum_{i=1}^n (\1_{X_i \in A_t} - \1_{X_i \in A_{t'}}   )^2 \)}{2} \\
                &= \frac{\lambda^2 \(\frac{1}{n}\sum_{i=1}^n (\1_{X_i \in A_t} \neq \1_{X_i \in A_{t'}}   ) \)}{2}.
\end{align*}

Let $d(t,t') = \sqrt{\frac{1}{n}\sum_{i=1}^n (\1_{X_i \in A_t} \neq \1_{X_i \in A_{t'}}   )}$, and let $H(\delta,\cT)$ denote the universal $\delta$-metric entropy (with respect to  $d(\cdot,\cdot)$).
Since the zero function (corresponding to $\emptyset$) belongs to the function class, we have
\begin{align*}
\sup_{t\in \cT} d(t,0) &=
\sup_{t \in \cT} \sqrt{ \frac{1}{n} \sum_{i=1}^n \1_{X_i \in A_t}} \leq \delta_n.
\end{align*}
Therefore, we can apply the discrete version of Dudley's inequality (Lemma 13.1 in Boucheron et al.~\cite{BouEtAl16}) with $\delta_n$ as the maximum radius. Since $\delta_n \geq \sigma$, we can upper-bound the random quantity $H(a\delta_n)$ by the fixed quantity $H(a \sigma)$, for any $a>0$. This implies that
\begin{align*}
&\E\left[ \frac{1}{\sqrt{n}} \sup_{t \in \cT} \sum_{i=1}^n \epsilon_i \1_{X_i \in A_t} \bigg| X_1, \hdots, X_n\right] 
\\
&\leq 3 \sum_{j=0}^{\infty} \delta_n2^{-j} \sqrt{H(\delta_n2^{-j-1}, \cT)} \\
& \leq 3 \sum_{j=0}^{\infty} \delta_n2^{-j} \sqrt{H(\sigma2^{-j-1}, \cT)}.
\end{align*}
Taking the expectation with respect to $X_1, \hdots, X_n$ and combining with inequality~\eqref{EqnSym} we then obtain
\begin{align*}
\E Z &\leq 6 \E \delta_n \cdot \sum_{j=1}^{\infty} 2^{-j} \sqrt{H(\sigma2^{-j-1}, \cT)} 
\\
&\leq 6 \sqrt{\E \left(\frac{Z}{\sqrt{n}}\right) + \sigma^2} \(\sum_{j=1}^{\infty} 2^{-j} \sqrt{H(\sigma2^{-j-1}, \cT)}\).
\end{align*} 
From this step onward, the proof is identical to the proof of Theorem 13.7 in Boucheron et al.~\cite{BouEtAl16}. 
\end{proof}

\begin{theorem}(Theorem 8.3.23 in Vershynin~\cite{Ver18})
Let $\cF$ be  a  class of  Boolean  functions  on  a  probability  space $(\Omega, \Sigma , \mu)$ with  finite  VC  dimension $V \geq 1$. Let $X, X_1, X_2, \hdots, X_n$ be independent random points in $\Omega$ distributed according to the law $\mu$. Then
\begin{align*}
\E\left[ \sup_{f \in \cF} \left| \frac{1}{n} \sum_{i=1}^nf(Xi) - \E f(X) \right| \right] \leq C \sqrt{ \frac{V}{n}}.
\end{align*}
\label{ThmVCAbsDeviation}
\end{theorem}

\end{document}